\documentclass[a4paper,11pt]{article}
\usepackage{amsmath}
\usepackage{parskip}
\usepackage{amsfonts,amssymb,mathtools,amsthm}
\usepackage{graphicx,color,xcolor}

\newcommand{\E}{\ensuremath{\mathbb{E}}}

\def\1{1\!{\rm l}}

\graphicspath{{../}}

\usepackage{geometry}
\usepackage{comment}

\usepackage{amsmath}

\newtheorem{thm}{Theorem}
\newtheorem{prop}{Proposition}
\newtheorem{lem}{Lemma}

\newtheorem{rem}[thm]{Remark}

\newcounter{hypcounter}
\title{Minimax estimation of Functional Principal Components from noisy discretized functional data: the case of smooth processes}

\author{Nassim Bourarach$^*$, Franck Picard$^\dag$, Vincent Rivoirard$^{*,\ddag}$ and Angelina Roche$^\star$ \vspace{0.5cm} \\	
	$^*$CEREMADE, CNRS, Universit\'e Paris-Dauphine, Universit\'e PSL, 75016 Paris, France\\
	$^\dag$Ecole Normale Supérieure de Lyon, CNRS, Université Claude Bernard Lyon 1, 
	Lyon, France \\		
	$^\ddag$Universit\'e Paris-Saclay, CNRS, Inria, LMO, 91405, Orsay, France \\
	$^\star$Université Paris Cité, CNRS, MAP5, F-75006 Paris, France
}

\begin{document}

	\maketitle
	\begin{abstract}
		We study the minimax estimation of covariance eigenfunctions and eigenvalues in functional principal component analysis when $n$ trajectories are observed at $p$ common grid points with additive noise. We consider covariance kernels with arbitrary H\"older smoothness and no prescribed parametric decay of the eigenvalues. In this setting, kernel smoothness and local spectral separation play distinct roles: a minimax inconsistency result over the smoothness-only class shows that kernel regularity alone is not sufficient for minimax-consistent eigenfunction estimation. To capture this interplay, we introduce a class of processes that jointly controls the H\"older smoothness of the covariance kernel and a local relative inverse eigengap quantity at the target index $\ell$. Over this class, we derive non-asymptotic minimax lower bounds for eigenfunction estimation that disentangle sampling variability, discretization and spectral effects, revealing rates of order $\delta_\ell n^{-1}+p^{-2\alpha}$, where $\delta_\ell$ depends on the spectral rank $\ell$. We also obtain non-asymptotic lower bounds for eigenvalue estimation under a relative squared-error loss. We then construct a computable wavelet projection estimator based on Coiflet scaling functions and a quadrature scheme designed to accommodate arbitrary H\"older smoothness. For eigenfunction estimation, this estimator matches the minimax dependence on the sample size and grid resolution, up to the natural spectral factor, for any H\"older index $\alpha>0$. Finally, we show that the proposed framework covers several classical Gaussian processes and Karhunen--Lo\`eve constructions. In particular, a Karhunen--Lo\`eve based criterion links spectral decay, eigenfunction regularity and covariance-kernel smoothness, and yields controlled simulation settings illustrating the predicted phase transitions and least-favourable discretization effects.
	\end{abstract}
	
	\bigskip
	\bigskip
	\section{Introduction}\label{sec:intro}
	
	Functional Data Analysis (FDA) addresses statistical problems in which the observations are realizations of random functions rather than finite-dimensional vectors. Many textbooks focus on this area of Statistics either from a methodological viewpoint (see for example \cite{FerratyVieu,RamsaySilverman}) or from a theoretical one \cite{HsingEubanks}. Readers seeking a concise overview of the various facets of FDA may also find answers in the articles by \cite{Wang2015reviewFDA,gertheiss2023reviewFDA}. 
	
	Intrinsically, functional data are infinite dimensional, which motivates the use of dimension-reduction techniques. The most widely of which is naturally Functional Principal Component Analysis (FPCA) which aims to find a new functional basis that captures the variability of the curves. FPCA can be used as either a visualisation or dimension reduction tool for functional data. This approach has proven useful for regression \cite{Cardot_Vieu_Sarda_Regression}, classification \cite{Wang2021optclassif_FDA}, clustering \cite{Li_Chiou_Classif}, and exploratory analysis \cite{Keshri_Sharma_exploratory_anal_FPCA}, among many other statistical tasks.
	
	From a theoretical viewpoint, understanding the statistical properties of FPCA estimators has received considerable attention. Regardless of the specific type of FPCA estimators, the study of the statistical properties associated with the estimators of the eigenelements of the covariance of the process under study is almost always the first step. First studies were conducted in the case where the functional data are observed at all points without noise. In that case, the rates of convergence is parametric and degrades with the rank of the targeted eigenfunctions \cite{DPR82, mas, reiss2019nonasymptotic}. However, in practice, data are often observed on a grid, with possible noise. In that case, the estimation accuracy depends on several factors: the number of observed curves $n$, the number of sampling points per curve $p$, the regularity of the underlying process, and the structure of the covariance spectrum. These dependencies interact in subtle ways, leading to phase transitions in the achievable estimation rates.
	

	In this article, we consider that all curves share the same sampling points. This assumption is often referred as the common-grid design. In that case, aggregation of information across curves is possible, but unobserved regions must be approximated. Consequently, error bounds are of the form $n^{-1}+p^{-2\alpha}$ combining a sample-size term $n^{-1}$ and a grid-resolution term $p^{-2\alpha}$, where $\alpha$ encodes the Hölder regularity of the covariance kernel. Such decomposition with a double asymptotic (one of the estimation type and one of the approximation type), and the associated phase transition between the two regimes, has already been studied in the minimax framework for related problems such as the mean estimation by \cite{Cai2011mean}.
	
	More specifically, with regard to the estimation of covariance and its eigenelements in the common grid case, there exist results as in \cite{Guanqun,zhou2024theoryFPCA} or  \cite{Xiao} which presents asymptotic upper-bounds in the case of differentiable covariance kernel. Non-asymptotic results are obtained in \cite{descary2018} at the cost of analyticity assumptions, or in \cite{BPRR24}, which provides lower and upper bounds but is only interested in the case of rough processes ($\alpha$-Hölderian covariance kernel with $\alpha<1$). In the present work, we generalize and refine the results of \cite{BPRR24} to more regular processes than those considered, by extending the study to all
	eigenfunctions rather than just the first and second ones. Additionally, we propose a more nuanced treatment of the role played by the regularity of the covariance operator in the resulting error bounds ---as we consider the joint influence of the functional regularity and the spectrum's regularity--- which allows for a better understanding of the inherent statistical difficulty of the problem.
	
	Intuitively, a higher Hölder regularity coefficient $\alpha > 1$ for the autocovariance operator should lead to smaller estimation errors, as the statistical problem becomes “easier” in some sense. However, the theoretical analysis---particularly in the common-grid setting---becomes more challenging. For instance, obtaining sharp upper bounds on the error in this context requires more sophisticated quadrature techniques than simple Riemann sums. To address this, we propose the use of a wavelet-based approximation combined with a tailored quadrature method that accommodates the higher regularity of the underlying functions.
	
	In this work, we extend the minimax theory of FPCA under the noisy common-grid design to covariance kernels with arbitrary H\"older smoothness. Our analysis is non-asymptotic and separates the roles played by functional smoothness, grid resolution, sample size and spectral separation. The main contributions are as follows.
	
	\begin{enumerate}
		\item Although spectral separation is a standard requirement for stable eigendirection recovery, we formalize this limitation in the minimax sense by proving an inconsistency result. This shows that H\"older regularity of the covariance kernel alone does not ensure minimax consistency for eigenfunction estimation.
		
		\item Motivated by this obstruction, we introduce the class $\mathcal P(\alpha,L,\delta_\ell)$, which combines H\"older smoothness of the covariance kernel with a local bound on a relative inverse eigengap quantity at the target index $\ell$. This formulation separates the functional regularity of the process from the spectral difficulty of the target eigendirection, without imposing a parametric decay model on the whole sequence of eigenvalues.
		
		\item We establish non-asymptotic minimax lower bounds for eigenfunction estimation, as well as non-asymptotic lower bounds for eigenvalue estimation. These bounds disentangle the contributions of sampling variability, discretization and spectral separation, and exhibit the phase transitions induced by the common-grid design.
		
		\item We construct a computable wavelet projection estimator based on Coiflet scaling functions and a quadrature scheme tailored to high-order smoothness. The use of vanishing moments allows us to handle arbitrary H\"older regularity, including regimes beyond first-order smoothness. The resulting eigenfunction estimator matches the minimax dependence on the sample size and grid resolution, thereby establishing minimax optimality for eigenfunction estimation.
		
		\item We show that the proposed framework covers a broad range of Gaussian processes, including classical examples and processes generated through Karhunen--Lo\`eve expansions. In particular, we provide a Karhunen--Lo\`eve based criterion showing how polynomial eigenvalue decay, combined with suitable regularity bounds on the eigenfunctions, yields covariance kernels with prescribed H\"older regularity. This criterion also yields simulation settings in which the eigenelements and the kernel smoothness are known explicitly, allowing us to illustrate the predicted phase transitions and the least-favourable discretization mechanisms behind the lower bounds.
	\end{enumerate}

	\paragraph{Outline of the paper.}
	Section \ref{sec:model} introduces the model and regularity classes. 
	Section \ref{sec:lbounds} presents the minimax lower bounds, including an inconsistency result and explicit rates for both eigenfunctions and eigenvalues. 
	Section \ref{sec:ub} constructs the wavelet projection estimator and proves its optimality. 
	Section \ref{sec:illustration} provides examples of processes belonging to the regularity classes we introduce and Section \ref{sec:simu} reports simulation results. 
	Technical proofs and auxiliary lemmas are gathered in Section \ref{sec:proof}.
	\paragraph{Notation:} 
	We denote the set of non-negative integers by $\mathbb{N}$. For integers $a \le b$, the discrete interval $\{a, \dots, b\}$ is represented by $[\![ a,b]\!]$, and for any set $\mathcal{D}$, the indicator function is denoted by $1_{\mathcal{D}}$. The space of square-integrable functions on the unit interval, $\mathbb{L}^2([0,1])$, is equipped with the standard inner product $\langle f,g \rangle \coloneqq \int_0^1 f(t)g(t) \, dt$ and the induced norm $\|f\| \coloneqq \langle f,f \rangle^{1/2}$. For a measurable function $f$, we define the $\mathbb{L}^1$ and supremum norms as $\|f\|_1 \coloneqq \int |f|$ and $\|f\|_{\infty} \coloneqq \sup_x |f(x)|$, respectively. Regarding algebraic structures, $\|u\|_{\ell_2}$ refers to the Euclidean norm of a vector $u$, while for a matrix or operator $\mathbf{A}$, the Hilbert-Schmidt and operator norms are denoted by $\|\mathbf{A}\|_{\operatorname{HS}} \coloneqq (\mathrm{tr}(\mathbf{A}^\top \mathbf{A}))^{1/2}$ and $\|\mathbf{A}\|_{\mathcal{L}}$. For $f, g \in \mathbb{L}^2([0,1])$, the rank-one operator $f \otimes g$ is defined such that $(f \otimes g)(h) = \langle h,g \rangle f$. Finally, for real sequences $(a_j)$ and $(b_j)$, we write $a_j = o(b_j)$ if $a_j/b_j \to 0$, $a_j \sim b_j$ if $a_j/b_j \to 1$, and $a_j \lesssim b_j$ if there exists a constant $C > 0$ such that $a_j \le C b_j$ for all sufficiently large $j$.
	\section{Model and regular classes}\label{sec:model}
	Let $n$ and $p$ be two positive integers. We observe $p$ noisy evaluations of $n$ independent and identically distributed (i.i.d.) realizations of a centered random function $X$, denoted by $X_1, \dots, X_n$, on the regular grid $t_j \coloneqq (j-1)/p$. The observation model is given by:
	\begin{equation}
		\label{eq:design}
		Y_i(t_j) = X_i(t_j) + \varepsilon_{i,j}, \quad (i,j) \in [\![1,n]\!]\times[\![1,p]\!],
	\end{equation}
	where the errors $\varepsilon_{i,j}$ are i.i.d. centered Gaussian random variables with variance $\sigma^2 > 0$, assumed to be independent of $(X_i)_{i=1}^n$. We assume that $X$ is a centered mean-squared continuous random function on $[0,1]$ satisfying $\mathbb{E}[\left\|X\right\|^2] < \infty$, ensuring the existence of an almost surely continuous version compatible with the evaluations at $t_j$.
	Let $\mathcal P$ be the set of all such probability distributions on $\mathbb L^2([0,1])$. We denote by $P_X\in\mathcal P$ the distribution of $X$ and the associated covariance kernel as
	\begin{equation}\label{eq:def_K}
		K_{P_X}(s,t)=\mathbb{E}\left[X(s)X(t)\right],\qquad (s,t)\in [0,1].
	\end{equation}
	The kernel $K_{P_X}$ is then continuous (hence in $\mathbb{L}^2([0,1]^2)$) and a fortiori symmetric positive semi-definite. We can then define the covariance operator associated to $P_X$, which is the integral operator associated to the covariance kernel:
	\begin{align}\label{eq:def_Gamma}
		\Gamma_{P_X}: \mathbb{L }_2 & \mapsto \mathbb{L }_2 \\
		f(\cdot) & \mapsto  \int_{[0,1]}K_{P_X}(v,\cdot)f(v)\,dv.
	\end{align}
	It is then a well-defined compact, self-adjoint, positive semi-definite, trace-class (and Hilbert-Schmidt) operator on $\mathbb{L}^2([0,1])$. We assume that its eigenvalues are distinct. We denote by $(\lambda_j(\Gamma_{P_X}))_{j\geq 1}$ the sequence of eigenvalues, sorted in decreasing order, and by $(\psi_j(\Gamma_{P_X}))_{j\geq 1}$ the associated sequence of eigenfunctions.
	Note that when there is no ambiguity, the index $P_X$ and/or the dependence on $\Gamma_{P_X}$ in the eigenelements will be omitted to simplify the notation.
	
	We now introduce the regularity classes of processes that we consider. Let $L>0$, $\beta\in(0,1]$, $m\in\mathbb{N}$ and set $\alpha=m+\beta$. We define
	\begin{eqnarray}\label{PalphaL}
		\mathcal P(\alpha,L)&=&\Bigg\{P_X\in\mathcal P:\ K_{P_X} \text{ is }m-\text{times differentiable with }\nonumber\\
		&&\hspace{1cm}\left|\frac{\partial^m K_{P_X}}{\partial s^m}(s,t)-\frac{\partial^m K_{P_X}}{\partial s^m}(s',t)\right|\leq L|s-s'|^\beta, s,s',t\in[0,1]\Bigg\}. 
	\end{eqnarray}
	Note that this regularity class generalizes Assumption 2 in \cite{zhou2024theoryFPCA} which corresponds to the class $\mathcal P(\alpha,L)$ with $\alpha=2$.
	
	We prove in Theorem~\ref{thm:inconsistency} below that under Model~\eqref{eq:design}
	$$\inf _{\widehat{\psi}_{\ell}} \sup _{P_{X} \in \mathcal{P}(\alpha,L)} \mathbb{E}\left[\left\|\widehat{\psi}_{\ell}-\psi_{\ell}(\Gamma_{P_X})\right\|^{2}\right] \geq c>0,$$
	where the minimum is taken over all estimators, that is, all measurable functions of the observations $\{Y_i(t_j), i=1,\hdots,n, j=1,\hdots,p\}$ taking values in $\mathbb{L}^2([0,1])$ with unit norm. This means that the estimation of eigenfunctions cannot be consistent in the minimax sense in this regularity class. This motivates us to introduce a smaller regularity class, by adding constraints on the eigenvalues.
	We define for any positive integer $\ell$  
	\begin{equation}
		r_\ell (\Gamma_{P_X}):=\begin{cases*}
			\frac{\lambda_{1}(\Gamma_{P_X})\lambda_{2}(\Gamma_{P_X})}{\left|\lambda_{1}(\Gamma_{P_X})-\lambda_{2}(\Gamma_{P_X}) \right|^2}& if $\ell=1$ \\
			\max_{k\in\{-1,1\}}\frac{\lambda_{\ell}(\Gamma_{P_X})\lambda_{\ell+k}(\Gamma_{P_X})}{\left|\lambda_{\ell}(\Gamma_{P_X})-\lambda_{\ell+k}(\Gamma_{P_X})  \right|^2} & otherwise.
		\end{cases*}\label{eq:relative_eigengap}.
	\end{equation}
	The quantity $r_\ell$ is related to the notion of \textit{relative rank} (see \cite{mas}, \cite{jirak2022}). It provides a weighted measure of the worst local separation of the $\ell$-th eigenvalue from its neighbouring eigenvalues. Now, let $\ell\in\mathbb{N}\setminus\{0\}$ and $\delta_\ell>0$. We define the following subclass of $\mathcal P(\alpha,L)$: 
	\begin{equation}\label{eq:constrained_class}
		\mathcal P(\alpha,L,\delta_{\ell})=\Big\{P_X\in\mathcal P\left(\alpha,L\right):\ r_\ell(\Gamma_{P_X})\leq\delta_\ell\Big\}.
	\end{equation}
	In Section~\ref{sec:illustration}, we derive the order of $r_\ell(\Gamma_{P_X})$ under standard assumptions (polynomial and exponential) on the decay rate of the sequence of eigenvalues. More generally, Section~\ref{sec:illustration} provides several examples of processes belonging to the class $\mathcal P(\alpha,L,\delta_{\ell})$. Some are classical Gaussian processes, such as the Brownian motion, the Brownian bridge, the integrated Brownian motion and the fractional Brownian motion. In addition, Theorem~\ref{thm:regu_kern} introduces a broader class of processes constructed via the Karhunen-Loève decomposition and belonging to the class $\mathcal P(\alpha,L,\delta_{\ell})$.
	
	Having introduced the general model and the regularity classes of processes under consideration, we now turn to the fundamental question of the intrinsic statistical difficulty of estimating the eigenelements of the covariance operator. To this end, Section \ref{sec:lbounds} establishes minimax lower bounds for the estimation of both eigenfunctions and eigenvalues, starting with a general inconsistency result that highlights why additional assumptions on the spectrum are necessary.
	\section{Minimax lower-bounds}\label{sec:lbounds}
	\subsection{Inconsistency result for the estimation of eigenfunctions}
	We start by establishing that, under the sole assumption of kernel smoothness, the estimation of eigenfunctions cannot be consistent in general. The following non-asymptotic theorem formalizes this fundamental limitation.
	\begin{thm}{(Inconsistency result)\\}\label{thm:inconsistency}
		Let $n\in\mathbb{N}\setminus\{0\},\alpha,L\in\mathbb{R}^*_+$. Let $p\in\mathbb{N}$ such that $p\geq4$, $\ell\in\mathbb{N}$ such that $ \ell+1<\frac{p}{2}$.
		Then there exists $c>0$ that does not depend on $\alpha,L$ or $\sigma$ such that, 
		$$\inf _{\widehat{\psi}_{\ell}} \sup _{P_{X} \in \mathcal{P}(\alpha,L)} \mathbb{E}\left[\left\|\widehat{\psi}_{\ell}-\psi_{\ell}(\Gamma_{P_X})\right\|^{2}\right] \geq c, $$
		where the infimum is taken over all measurable functions of the data.
	\end{thm}
	This theorem shows that kernel functional regularity alone is not sufficient to ensure the consistency of eigenfunction estimators.
	Although the functional regularity assumptions on $K$ induces certain constraints on the spectrum (through the interplay of the quantities $\alpha, L$ and the eigenelements), these are insufficient to guarantee consistency. 
	The proof of this result is given in Subsection \ref{subsec:proof_inconsistency} and is based on Tsybakov's two-hypothesis method \cite{Tsybakov2009IntroductionTN}. 
	\subsection{Lower-bound for the estimation of eigenfunctions}
	By restricting the regularity class of processes under study and considering the class $ \mathcal{P}(\alpha,L,\delta_\ell)$, we obtain the following lower bound.
	\begin{thm}{(General minimax lower-bound for the $\ell$-th eigenfunction estimation)\\}\label{thm:eigfun_lb}
		Let $n\in\mathbb{N}\setminus\{0\},\alpha,L\in\mathbb{R}^*_+$. Let $p\in\mathbb{N}$ such that $p\geq \max\left(\left(2+\sqrt 2\right)^{\frac{1}{2\alpha}},4\right),$ $\ell\in\mathbb{N}$ such that $ \ell+1<\frac{p}{2}$ and $\delta_\ell\in\mathbb{R}^*_+$. Then, there exists a constant $c_\sigma>0$, that does not depend on $\ell$ and $c_\ell>0$ that does not depend on $\sigma^2$ such that
		$$\inf _{\widehat{\psi}_{\ell}} \sup _{P_{X} \in \mathcal{P}(\alpha,L,\delta_\ell)} \mathbb{E}\left[\left\|\widehat{\psi}_{\ell}-\psi_{\ell}(\Gamma_{P_X})\right\|^{2}\right] \geq \frac{c_\sigma\delta_\ell}{n}+\frac{c_\ell}{p^{2\alpha}},$$
		where the infimum is taken over all measurable functions of the data.
	\end{thm}
	Theorem~\ref{thm:eigfun_lb} is a direct consequence of the results of Propositions~\ref{prop:eigfun_lb_n} and \ref{prop:eigfun_lb_p} proved in Subsection~\ref{subsec:proof_eigfun_lb}. 
	
	Note that the bound, established in the non-asymptotic setting and under very mild assumptions, clearly separates the effects of the sample size and the grid size. We retrieve the phase-transition phenomenon described in the literature on the estimation of the mean function of discrete functional data (see \cite{Cai2011mean}): when $n$ is small compared to $p^{2\alpha}$, the rates of convergence is determined solely by the size $n$ of the sample. Conversely, when $n$ is large compared to $p^{2\alpha}$, the discretization has a non-negligible impact on the quality of estimation of eigenfunctions. This phenomenon, for the minimax lower-bound, has already been highlighted in \cite{BPRR24}, but only in the context of the estimation of the first or second eigenfunction, for non-regular processes ($\alpha<1$) and without taking into account for the effect of the spectrum and the rank of the eigenfunction on the bound (captured here by $\delta_\ell$ and $c_\ell$). This last point is critical because, as Theorem \ref{thm:inconsistency} shows, without precise control of the spectrum, the lower bound can be larger than a strictly positive constant. This is a point that is often implicitly handled in the literature by having a fixed decay on the eigenvalues (as in Equations \eqref{def:pol_regime} and \eqref{def:exp_regime}). However, the present result is more general and does not require specifying a parametric form for the decay of the eigenvalues.
	\subsection{Lower-bound for the estimation of eigenvalues}
	We now study the case of eigenvalues estimation. Intuitively, the eigenvalue captures the overall energy along an eigendirection. Hence, one might expect the lower bounds to be smaller. This intuition is confirmed by the following theorem.
	\begin{thm}{(General minimax lower-bound for the $\ell $-th eigenvalue estimation)}\label{Thm_eigval}\\
		Let $n\in\mathbb{N}\setminus\{0\},\alpha,L\in\mathbb{R}^*_+$. Let $p\in\mathbb{N}$ such that $p\geq \max\left(\left(2+\sqrt 2\right)^{\frac{1}{2\alpha}},4\right)$, $\ell\in\mathbb{N}$ such that $ \ell+1<\frac{p}{2}$ and $\delta_\ell\in\mathbb{R}^*_+$.
		Then, there exist $c'>0$ that does not depend on $\sigma$ and $\ell$ and $c_{\sigma}'>0$ that does not depend on $\ell$ such that
		$$\inf_{\widehat{\lambda}_{\ell}} \sup _{P_{X} \in \mathcal{P}(\alpha, L)} \mathbb{E}\left[\left(\frac{\widehat{\lambda}_{\ell}-\lambda_{\ell}(\Gamma_{P_X})}{\lambda_{\ell}(\Gamma_{P_X})}\right)^{2}\right] \geq \frac{c_{\sigma}'}{n}+\frac{c'}{p^{4\alpha}},$$
		where the infimum is taken over all measurable functions of the data.
	\end{thm}
	Theorem \ref{Thm_eigval} is a direct consequence of the results in Propositions \ref{prop:eigval_lb_n} and \ref{prop:eigval_lb_p} proved in Subsection \ref{subsec:proof_eigval_lb}.
	
	To the best of our knowledge, no minimax lower bound has been established for eigenvalue estimation in the functional data setting, whether the observations are discrete or fully observed. 
	In contrast to eigenfunction estimation, the functional regularity of the kernel appears to be sufficient, and no condition on $r_\ell$ is required. As previously, we also observe a phase transition phenomenon, occurring for $n\approx p^{4\alpha}$. Therefore, as soon as $p^{4\alpha}$ is much marger than $n$, even with partially observed trajectories, one attains the optimal parametric rate of convergence, as if the entire trajectories were observed.
	\section{Wavelet estimation of eigenfunctions and minimax upper bounds}\label{sec:ub}
	The lower bounds established in Section~\ref{sec:lbounds} provide a benchmark for any estimator of $\psi_{\ell}(\Gamma_{P_X})$. We now construct an estimator that achieves these rates, thereby establishing minimax optimality. To this end, we adopt a wavelet projection approach, which naturally adapts to the smoothness of the covariance kernel and accommodates discretized observations.
	\subsection{A specific wavelet projection estimator}\label{def_estim}
	In this section we introduce the wavelet formalism used in the construction of our estimator.
	Wavelet bases offer a flexible representation of smooth functions and are particularly well suited
	to approximate covariance kernels with Hölder regularity. They combine orthogonality, locality
	and multiscale structure, providing a natural way to control approximation errors. 
	In what follows, we briefly recall the main principles of
	wavelet multiresolution analysis, establish the link with classical Fourier decompositions, and
	justify the use of Coiflet scaling functions in our estimator.
	
	Wavelet theory relies on the notion of MultiResolution Analysis (MRA), which consists in
	constructing two sequences of nested closed subspaces $(V_j)_{j\in\mathbb{Z}}$ and $(W_j)_{j\in\mathbb{Z}}$
	of $\mathbb{L}_2(\mathbb{R})$ satisfying
	\[
	V_j \subset V_{j+1}, \qquad V_{j+1} = V_j \oplus W_j, \qquad
	\overline{\bigcup_{j\in\mathbb{Z}} V_j} = \mathbb{L}_2(\mathbb{R}), \qquad
	\bigcap_{j\in\mathbb{Z}} V_j = \{0\}.
	\]
	The space $V_j$ represents the information contained in the signal up to the resolution level
	$2^{-j}$, while $W_j$ gathers the additional details at the same scale. The subspace $V_j$
	is generated by the dilations and translations of a scaling function $\phi$, whereas
	$W_j$ is generated by the corresponding wavelet function~$\psi$:
	\[
	\phi_{j,k}(t) = 2^{j/2}\phi(2^j t - k), \qquad
	\psi_{j,k}(t) = 2^{j/2}\psi(2^j t - k), \qquad j,k\in\mathbb{Z}.
	\]
	For any $j_0\in\mathbb{Z}$, the family $\{\phi_{j_0,k}, \psi_{j,k}: j\ge j_0,\,k\in\mathbb{Z}\}$ then forms an orthonormal
	basis of $\mathbb{L}_2([0,1])$ after boundary correction. Every function $f\in \mathbb{L}_2([0,1])$ can be
	decomposed as
	\[
	f = \sum_{k}\langle f, \phi_{j_0,k}\rangle \phi_{j_0,k}
	+ \sum_{j\ge j_0}\sum_k \langle f, \psi_{j,k}\rangle \psi_{j,k}.
	\]
	This representation can be seen as a localized analogue of the Fourier expansion:
	the index $j$ acts as a frequency band, while $k$ controls the spatial localization.
	
	When $f$ has Hölder regularity $\alpha>0$, the wavelet coefficients
	$|\langle f, \psi_{j,k}\rangle|$ decay as $2^{-j(\alpha+1/2)}$, implying that the orthogonal
	projection
	\[
	\Pi_J f = \sum_k \langle f, \phi_{J,k}\rangle \phi_{J,k}
	\]
	onto the subspace $V_J$ satisfies the classical approximation bound
	\begin{equation}\label{eq:approx_func_wav}
		\|f - \Pi_J f\|_{L^2}^2 \lesssim 2^{-2J\alpha}. 
	\end{equation}
	This result is standard in wavelet approximation theory
	(see, for instance, Proposition 9.3 of \cite{Mallat1999}).
	This allows wavelet projections to provide optimal approximations for smooth functions.
	
	Among compactly supported orthonormal wavelets, the Coiflet family
	introduced by Daubechies \cite{Daubechies1992} is of particular interest for our application. A Coiflet system of order
	$m$ is defined such that the scaling and the wavelet functions possess $m$ and $m+1$ vanishing
	moments respectively:
	\begin{equation}
		\label{eq:van_mom}
		\int_0^1 t^r \phi(t)\,dt = 1_{\{r=0\}}, \qquad \int_0^1 t^r \psi(t)\,dt = 0, \qquad
		r = 0,\ldots,m.
	\end{equation}
	This property ensures that projection errors for $m$-times differentiable functions decrease
	at the optimal rate $O(2^{-2Jm})$ and will allow us to use a precise quadrature to compute the wavelet coefficients from the discrete observations. Moreover, Coiflets enjoy high smoothness together with
	compact support, which allows efficient numerical implementation and controlled boundary
	behaviour.
	
	In our setting, only the scaling functions are required to construct the projection estimator. To build our estimator we use the Coiflet scaling basis $(\phi_{j,k})_{j\in\mathbb{N}\setminus\{0\},(k)\in[\![0,2^j-1]\!]}$ with vanishing moments.  This choice ensures that the exact and approximate projection operators
	\begin{equation}\label{eq:approx_proj}
		\begin{aligned}
			\Pi_p(f):=\sum_{k=0}^{p-1}\left \langle f,\phi_{\operatorname{log}_2(p),k} \right \rangle\phi_{\operatorname{log}_2(p),k},\\
			\widetilde{\Pi}_p(f):=\sum_{k=0}^{p-1}p^{-\frac{1}{2}}f\left(\frac{k}{p}\right)\phi_{\operatorname{log}_2(p),k}
		\end{aligned}
	\end{equation}
	are sufficiently close to each other.
	
	In fact, the estimation of the covariance kernel $K = \mathbb{E}[X \otimes X]$ can be carried out in successive steps, namely, approximation, quadrature, and estimation, based respectively on
	\begin{description}
		\item[] $$K_{\Pi_p} := \mathbb{E}\left [\Pi_p(X)\otimes \Pi_p(X) \right ],$$
		\item[] $$K_{\widetilde{\Pi}_p} :=\mathbb{E}\left [ \widetilde{\Pi}_p(X)\otimes \widetilde{\Pi}_p(X) \right ]$$
		and
		\item[] $$\widehat{K_{\widetilde{\Pi}_p}}(s, t) :=\frac{1}{n} \sum_{i=1}^{n}\widetilde{\Pi_p}(Y_i)(s)\cdot \widetilde{\Pi_p}(Y_i)(t)-\sigma^2p^{-1}\left (\sum_{k=0}^{p-1}\phi_{\operatorname{log}_2(p),k}(s)\phi_{\operatorname{log}_2(p),k}(t) \right ).$$
	\end{description}
	The strong approximation properties of $K_{\Pi_p}$ follow directly from the classical wavelet approximation result \eqref{eq:approx_func_wav}. The accuracy of the Coiflet-based quadrature employed in $K_{\widetilde{\Pi}_p}$, which relies on the vanishing moment property of the Coiflet scaling function, is established in Lemma \ref{lem:approx_scal_prod}. Finally, the reliable estimation by $\widehat{K_{\widetilde{\Pi}_p}}$ is guaranteed under a fourth-order moment assumption on $X.$ Altogether, these results support the use of the Coiflet scaling basis in the proposed estimator and ensure its minimax-optimal performance with respect to both $n$ and~$p$.
	
	Note that the second term of $\widehat{K_{\widetilde{\Pi}_p}}$ is used to debias the estimator. It comes from the fact that the diagonal observations $\left(\widehat{K_{\widetilde{\Pi}_p}}(t_j, t_j)\right)_{1\leq j\leq p}$ are used to approximate the functions. Most existing FDA methods ignore diagonal observations because (under the i.i.d. noise assumption) this is the only part of the observed covariance kernel impacted by noise. Nevertheless, as we are interested in non-asymptotic results in $p$, this introduces an additional major theoretical difficulty in having a quadrature method on a non-regular grid of $[0,1]^2$ that achieves exactly the desired $O(p^{-2\alpha})$ rate. 
	
	By denoting the debiasing term
	$$C_\sigma(s,t):=\sigma^2p^{-1}\left (\sum_{k=0}^{p-1}\phi_{\operatorname{log}_2(p),k}(s)\phi_{\operatorname{log}_2(p),k}(t) \right )$$
	we observe that the estimator can be written as:
	\begin{equation}\label{eq:Cov_kern_estimator}
		\begin{aligned}
			\widehat{K_{\widetilde{\Pi}_p}}(s, t) &=\frac{1}{n} \sum_{i=1}^{n}\widetilde{\Pi}_p(Y_i)(s)\cdot \widetilde{\Pi}_p(Y_i)(t)-C_\sigma(s,t)\\
			&=\frac{1}{np} \sum_{i=1}^{n}\sum_{k=0}^{p-1}\sum_{k'=0}^{p-1}Y_i\left(\frac{k}{p}\right)Y_i\left(\frac{k'}{p}\right)\phi_{\operatorname{log}_2(p),k}(s)\phi_{\operatorname{log}_2(p),k'}(t)-C_\sigma(s,t).\\
		\end{aligned}
	\end{equation}
	Since the  estimator $\widehat{K_{\widetilde{\Pi}_p}}$ of the covariance kernel can be decomposed on the finite basis $(\phi_{\log_2(p),k})_{(k)\in[\![0,p-1]\!]}$ we can compute  the eigenelements of its associated integral operator, denoted by $\big(\widehat{\widetilde{\lambda_{\ell,\phi}}},\widehat{\widetilde{\psi_{\ell,\phi}}}\big)$, which correspond respectively the $\ell$-th eigenvalue and eigenfunction of $\widehat{\Gamma_{\widetilde{\Pi}_p}}$, with $\widehat{\Gamma_{\widetilde{\Pi}_p}}$ being the integral operator associated to $\widehat{K}_{\widetilde{\Pi}_p}$:
	$$
	\begin{aligned}
		\widehat{\Gamma_{\widetilde{\Pi}_p}}: \mathbb{L }_2 & \mapsto \mathbb{L }_2 \\
		f(\cdot) & \mapsto  \int_{[0,1]}\widehat{K_{\widetilde{\Pi}_p}}(v,\cdot)f(v)\,dv.
	\end{aligned}
	$$ 
	See Lemma~\ref{lem:spectral_relation} in Section~\ref{sec:technicallemmas}. The next section describes the final estimator of $\psi_\ell$.
	\subsection{Minimax upper bounds for eigenfunction estimation}
	Before stating the main upper bound result, we recall that the eigenfunctions of the covariance operator are defined only up to a sign. Indeed, if $\psi_\ell$ is a unit-norm eigenfunction associated with a given eigenvalue, then its opposite, $-\psi_\ell$, is also a unit-norm eigenfunction corresponding to the same eigenvalue.  To obtain a meaningful evaluation of our estimator's performance, we follow the standard convention of measuring the error up to this sign ambiguity. This is achieved by aligning the sign of the true eigenfunction with the estimated one before computing their distance, which justifies the following definition:
	\begin{equation*}
		\psi_{\ell,\pm}=\operatorname{sign}\left(\left\langle\widehat{\widetilde{\psi_{\ell,\phi}}}, \psi_{\ell}\right\rangle\right) \times \psi_{\ell}.
	\end{equation*}
	This allows us to present the following upper bound theorem for the estimation error of the $\ell$-th eigenfunction.
	\begin{thm}\label{thm:UB}
		Let $n\in{\mathbb N}\setminus\{0\}$ and $p\in{\mathbb N}\setminus\{0\}$ such that $p=2^J$ for some $J\in\mathbb{N}\setminus\{0\}$. Denote $\eta_\ell$ the squared inverse eigengap:
		\begin{equation*}
			\eta_{1}=\left(\lambda_{1}-\lambda_{2}\right)^{-2} \text{, and for any }\ell \geq 2,\quad
			\eta_{\ell}=\max_{k\in\left\{1,-1\right\}}\frac{1}{\left(\lambda_{\ell}-\lambda_{\ell+k}\right)^2} .
		\end{equation*}
		We assume that $\eta_\ell<+\infty$,
		$$\underset{t\in[0,1]}{\sup}\mathbb{E}\left[X^4(t)\right]\leq M$$
		and that the father wavelet $\phi$  is compactly supported and satisfies for some $m\in{\mathbb N}\setminus\{0\}$
		$$\int_0^1 t^r \phi(t)\,dt = 0, \qquad
		r = 1,\ldots,m.$$
		Assume that there exist $0<\alpha\leq m+1$ and $L>0$ such that $P_X\in\mathcal P(\alpha,L)$. 
		Then there exists $C_{\sigma,M}$ that only depends on $\sigma^2$  and $M$ and $C_{\phi,\alpha}$ that depends only on $\phi$ and $\alpha$ such that 
		\begin{align*}
			\mathbb{E}\left[\left \|\widehat{\widetilde{\psi_{\ell,\phi}}}-\psi_{\ell,\pm}  \right \|^2\right]&\leq \eta_\ell\left(\frac{C_{\sigma,M}}{n}+C_{\phi,\alpha}\frac{L^2}{p^{2\alpha}}\right).
		\end{align*}
	\end{thm}
	Once again, we observe a phase transition when $n$ is of the same order as $p^{2\alpha}$. This relationship highlights which component dominates the estimation error, depending on the interplay between the intrinsic regularity of the observations, the sample size and the number of sampling points. The error decomposes into two distinct contributions: a purely statistical component and a purely approximation component. In particular, the noise level does not affect the constant $C_{\phi,\alpha}.$
	
	To establish this result, the total error is decomposed into three terms -- estimation error, quadrature error, and approximation error -- reflecting the three-step construction of the estimator described previously. The crucial ingredient is the use of a basis with sufficiently many vanishing moments, which ensures that the quadrature error matches the order of the approximation error and yields the optimal decay rate $p^{-2\alpha}.$ Coiflets bases satisfy these conditions.
	
	Another noteworthy aspect is that, although the estimator relies on projection onto a basis -- typically suggesting the presence of a tuning parameter -- no smoothing parameter is required in the present common-grid setting. Remarkably, interpolation via Coiflets alone suffices to achieve minimax optimality with respect to both $n$ and $p$. The only essential practical requirement is the choice of a wavelet basis with enough vanishing moments, namely at least as many as the order of the kernel derivatives.
	\section{Examples of processes belonging to the regularity class}\label{sec:illustration}
	The class $\mathcal{P}(\alpha,L,\delta_\ell)$ introduced in \eqref{eq:constrained_class} captures two complementary aspects that govern the statistical difficulty of FPCA for discretely observed noisy curves: the functional regularity of the covariance kernel, characterized by the pair $(\alpha,L)$ and the spectral conditioning of the $\ell$-th eigendirection, expressed through $r_\ell$ and the constraint $r_\ell \le \delta_\ell$.  In the sequel, we provide some examples of processes belonging to this regularity class. We first consider general processes based on polynomial or exponential decay rates for eigenvalues illustrated by some classical processes. Then, we present a construction based on the Karhunen-Loève decomposition.
	\subsection{Processes based on polynomial or exponential decay rates for eigenvalues}
	Most existing theoretical works on functional data analysis rely on assumptions about the decay rate on the eigenvalues to capture the regularity of functional data. Usually, it is commonly assumed (see for instance \cite{Hall_Horowitz_2007, zhou2024theoryFPCA}) that these eigenvalues decrease either at a polynomial or at an exponential rate. We show in the first point of Lemma~\ref{lem:section5.1_pol}  and in Lemma~\ref{lem:section5.1_exp} that these regimes allow to bound the quantity $r_\ell$ by some suitable values $\delta_\ell$. Combining this result with regularity conditions on $K$, we then obtain processes belonging to the class $\mathcal{P}(\alpha,L,\delta_\ell)$ (points 2, 3 and 4 of Lemma~\ref{lem:section5.1_pol}). To the best of our knowledge, however, no clear mathematical result elucidates the relationship between the decreasing rate of the eigenvalues and the regularity of the kernel.  In subsequent Theorem~\ref{thm:regu_kern}, we address this gap by proving that, under suitable regularity assumptions on the eigenfunctions, the decay rate of the eigenvalues characterizes the smoothness of the kernel.
	\subsubsection{Polynomial decay rates for eigenvalues}
	We first consider polynomial decay rates for eigenvalues.
	Assume there exists an integer $J>0$ and two constants $c_1,c_2>0$ such that 
	\begin{equation}\label{def:pol_regime}
		\lambda_j(\Gamma_{P_X})\leq c_1j^{-(\gamma+1)} \mbox{ and }  \lambda_j(\Gamma_{P_X})-\lambda_{j+1}(\Gamma_{P_X})\geq c_2j^{-(\gamma+2)}, \qquad j\geq J.
	\end{equation}
	This condition corresponds to Assumption 3 of \cite{zhou2024theoryFPCA}.
	Following classical stochastic processes fall into our general framework as proven in Lemma~\ref{lem:section5.1_pol} below. 
	\begin{enumerate}
		\item \textbf{Standard Brownian motion.}
		Let $X$ be the standard Brownian motion on $[0,1].$ Its covariance kernel is given by $K_{P_X}(s,t)=\min(s,t)$ so that $P_X\in\mathcal{P}\left(1,1\right)$. Its eigenvalues are 
		$$\lambda_j(\Gamma_{P_X})=\pi^{-2}\big(j-\frac12\big)^{-2},\qquad j\ge1.$$
		\item \textbf{Brownian bridge on \([0,1]\).}
		Let $X$ be the standard Brownian bridge on $[0,1]$. Its covariance kernel is given by $K_{P_X}(s,t)=\min(s,t)-st$ so that $P_X\in\mathcal{P}\left(1,1\right)$. Its eigenvalues are  
		$$\lambda_j\left(\Gamma_{P_X}\right)=\pi^{-2} j^{-2},\qquad j\ge 1.$$
		\item \textbf{\(k\)-fold integrated Brownian motion.} 
		Let $X$ be the $k$-fold temporally integrated standard Brownian motion on $[0,1]$ (for $k\ge 1$) defined in \cite{Gao2003_IB} as :
		$$X(t)=\int_0^t \int_0^{x_k} \cdots \int_0^{x_2} B\left(x_1\right) d x_1 d x_2 \cdots d x_m, $$
		with $B$ the standard brownian motion on [0,1]. Its covariance kernel is 
		$$K_{P_X}(s,t)=\int_0^{\min(s,t)}\frac{(s-u)^k (t-u)^k}{(k!)^2}\,du.$$
		\item \textbf{Fractional Brownian motion.}
		Let $X$ be the fractional Brownian motion on $[0,1]$ with Hurst coefficient $H\in(0,1)$ (see \cite{Bronski2003_fBm}). 
		The covariance kernel is
		$$K_{P_X}(s,t)=\frac{1}{2}\big(s^{2H}+t^{2H}-|s-t|^{2H}\big).$$
	\end{enumerate}
	The next lemma establishes that these processes are in fact special instances of our general framework. The key point in each setting is to control the quantity $r_\ell$ so that $\delta_\ell>0$ can be chosen such that the considered process belongs to $\mathcal{P}(\alpha, L, \delta_\ell)$.
	\begin{lem}\label{lem:section5.1_pol}
		We have the following results.
		\begin{enumerate}
			\item Assume that the sequence of eigenvalues  $(\lambda_j(\Gamma_{P_X}))_{j\geq 1}$ satisfies the conditions of Equation~\eqref{def:pol_regime} then there exists a constant $c>0$ such that 
			$$
			r_\ell(\Gamma_{P_X})\leq c\ell^2, \qquad\ell\geq 1.
			$$
			\item The standard Brownian motion and the standard Brownian bridge satisfy Equation~\eqref{def:pol_regime} with $\gamma=1$ and their associated distribution $P_X$ belongs to $\mathcal{P}(1,1,C\ell^2)$. 
			\item The $k$-fold integrated Brownian motion satisfies Equation~\eqref{def:pol_regime} with $\gamma=2k+2$ and its associated distribution belongs to $\mathcal{P}(2k+1,L_k,C_k\ell^2)$ for two constants $L_k$ and $C_k$. 
			\item The fractional Brownian motion satisfies Equation~\eqref{def:pol_regime} with $\gamma=2H$ and its associated distribution belongs to $\mathcal{P}(2H,L_H,C_H\ell^2)$ for two constants $L_H$ and $C_H$. 
		\end{enumerate}
	\end{lem}
	\subsubsection{Exponential decay rates for eigenvalues}
	We now consider exponential decay rates for eigenvalues. Assume there exists an integer $J>0$ and two constants $c_1,c_2>0$ such that 
	\begin{equation}\label{def:exp_regime}
		\lambda_j(\Gamma_{P_X})\leq c_1\exp(-j\gamma) \mbox{ and }  \lambda_j(\Gamma_{P_X})-\lambda_{j+1}(\Gamma_{P_X})\geq c_2\exp(-j\gamma), \qquad j\geq J.
	\end{equation}
	We then obtain the following result.
	\begin{lem}\label{lem:section5.1_exp}
		Assume that the eigenvalues sequence $(\lambda_j(\Gamma_{P_X}))_{j\geq 1}$ satisfies the conditions of Equation~\eqref{def:exp_regime}. Then, there exists a constant $c>0$ such that 
		$$
		r_\ell(\Gamma_{P_X})\leq c, \qquad \ell\geq 1. 
		$$
	\end{lem}
	\subsection{Processes based on the Karhunen-Loève decomposition}\label{subsec:useful_constr}
	This section presents a construction of processes belonging to the regularity class $\mathcal P(\alpha,L,\delta_\ell)$ by using the Karhunen-Loève decomposition. We first state a general result relating the Karhunen-Loève decomposition of a Gaussian process 
	\begin{equation}\label{eq:KKL}
		X(t) := \sum_{j=1}^{\infty} \sqrt{\lambda_j} \, \xi_j \, \psi_j(t),
	\end{equation}
	to the regularity of its covariance kernel. We recall that in the Gaussian case $(\xi_j)_{j=1}^\infty\overset{i.i.d.}{\sim}\mathcal{N}(0,1)$, which will be assumed in the sequel. We have:
	$$
	K_{P_X}(s,t) = \sum_{j=1}^{+\infty}\lambda_j\psi_j(s)\psi_j(t),\qquad (s,t)\in [0,1]^2.
	$$
	
	We first introduce the following regularity class for $K\equiv K_{P_X}$. Let $m \in \mathbb{N}$ and $\beta \in (0,1]$. We denote by $C^{m,\beta}([0,1]^2)$ the space of bivariate functions that are $m$-times differentiable with $\beta$-Hölder continuous derivatives of order $m$. That is, $K \in C^{m,\beta}([0,1]^2)$ if $K \in C^m([0,1]^2)$ and if there exists $L \in \mathbb{R}_+^*$ such that, for all multi-indices $\mathbf{v}$ with $|\mathbf{v}| = m$,
	$$|\partial^{\mathbf{v}}K(s,t) - \partial^{\mathbf{v}}K(s',t')|
	\leq L \, \|(s,t) - (s',t')\|^\beta.$$
	\begin{thm}\label{thm:regu_kern}
		Assume that $(\psi_j)_{j\geq 1}$ is an orthonormal basis of $\mathbb L^2([0,1])$ and $(\lambda_j)_{j\geq 1}$ is a summable sequence of positive numbers such that
		\begin{equation}\label{eq:derivative_bound}
			\begin{array}{ll}
				\exists\, \varsigma \in \mathbb{R}_+, \forall\, \ell \in \mathbb{N}, \exists\, M_\ell \in \mathbb{R}_+^*, \forall\, u\in [0,1],& |\partial^\ell \psi_j(u)| \leq M_\ell j^{\ell+ \varsigma}, \\
				\exists c,\gamma >0, \forall j\geq 1, & \lambda_j \leq cj^{-(\gamma+1)}. 
			\end{array}
		\end{equation}
		Let $\alpha := \gamma - 2\varsigma > 0$, $m:=\max\left\{k\in\mathbb{Z}, k<\alpha\right\}$, $\beta:=\alpha-m$ and 
		$$
		K(s,t) = \sum_{j=1}^{+\infty}\lambda_j\psi_j(s)\psi_j(t).
		$$
		\begin{itemize}
			\item If $\alpha$ is not an integer (so $\beta \in (0,1)$), then $K \in C^{m,\beta}([0,1]^2)$.
			\item If $\alpha$ is an integer (so $\beta = 1$), then $K \in C^{m,\delta}([0,1]^2)$ for all $\delta \in (0,1)$.
		\end{itemize}
	\end{thm}
	\begin{rem}
		In the last case, under our assumptions, we do not necessarily have $K \in C^{m,1}([0,1]^2)$. However, this result can be strengthened under more stringent conditions on the basis $(\psi_j)_{j \geq 1}$. Specifically, if, for some integer $\alpha,$ the partial sums of the series $\sum_j \partial^{\operatorname{v}}(\psi_j(s)\psi_j(t))$ are uniformly bounded for all multi-indices $|\operatorname{v}| = m+1$, then the conclusion improves to $K \in C^{m,1}([0,1]^2).$ This condition is satisfied for highly structured bases exhibiting cancellation effects, such as the trigonometric basis associated with Brownian motion.
	\end{rem}
	\begin{rem}
		In the case where $(\psi_j)_{j\geq 1}$ is the Fourier basis on $[0,1]$, namely 
		$$\psi_1(t) = 1,\quad \psi_{2k}(t) = \sqrt{2} \cos(2\pi k t),\quad \psi_{2k+1}(t) = \sqrt{2} \sin(2\pi k t),\quad k \geq 1 ,$$
		Condition~\eqref{eq:derivative_bound} is verified with $\varsigma = 0$ and $M_\ell = \sqrt{2}(2\pi)^\ell$.\\
		If, moreover, $(\lambda_j)_{j\geq 1}$ is a summable sequence of positive numbers such that there exists two constants $c>0$ and $\gamma>0$ such that $\lambda_j \leq cj^{-(\gamma+1)}, j\geq 1$, all the assumptions of Theorem~\ref{thm:regu_kern} are fulfilled for $\alpha=\gamma$. This implies that the distribution $P_X$ of any process $X$ with kernel $$
		K_{P_X}(s,t) = \sum_{j\geq 1}\lambda_j\psi_j(s)\psi_j(t), 
		$$
		belongs to $\mathcal P(\alpha,L,\delta_\ell)$, for a sufficiently large $L$ and $\delta_\ell\geq r_\ell(\Gamma_{P_X})$.
	\end{rem}
	This set of results therefore provides the rigorous bridge between the  functional regularity class $\mathcal{P}(\alpha,L,\delta_\ell)$ employed in our minimax analysis and the parametric spectral models that will be used for our simulations. In practice it allows us to choose spectral exponents $\gamma$ for which the induced kernel regularity $\alpha$ matches the theoretical regimes studied in Sections \ref{sec:lbounds}-\ref{sec:ub}, while at the same time knowing exactly the eigenelements of the associated processes. This point is crucial for the next section which is devoted to simulations.
	\section{Simulations}\label{sec:simu}
	We propose a simulation protocol to illustrate the theoretical results in a setting where we control 
	\begin{itemize}
		\item[(i)] the regularity of the covariance kernel to ensure that we are dealing with distributions belonging to the appropriate class of random processes,
		\item[(ii)] the eigenelements of the autocovariance operator, to be able to evaluate the error.
	\end{itemize}
	Two simulation frameworks are considered. The first one, called \textit{the minimax aliasing model}, is an adversarial construction specifically designed to saturate the discretisation term $p^{-2\alpha}$ predicted by the minimax theory; it constitutes our main scenario. The second setting is \textit{the Fourier model} already introduced in Subsection~\ref{subsec:useful_constr}, which serves as a complementary, more favorable scenario. Following Model~\eqref{eq:design}, in both cases, we observe
	\begin{equation}\label{eq:sample}
		Y_i(t_j)=X_i(t_j)+\varepsilon_{i,j},
		\qquad
		t_j=\frac{j-1}{p},\quad j=1,\dots,p,
	\end{equation}
	with $\varepsilon_{i,j}\overset{i.i.d.}{\sim}\mathcal{N}(0,\sigma^2)$ and $\sigma^2 = 0.1$. Both models differ only in the choice of the signal process~$X$. 
	\subsection{Fourier model}\label{subsec:fourier_model}
	To construct a stochastic process for which we know the regularity of the covariance kernel and its eigenelements explicitly, we rely on the construction developed in Subsection~\ref{subsec:useful_constr}. We simulate the signal such that:
	$$ X^{(\alpha)}\left ( \cdot \right )=\sum_{k=1}^{1001} \sqrt{\lambda_{k}^{*,(\alpha)}} \xi_k \psi_{k}^*\left(\cdot\right),$$
	with $(\xi_k)_{k=1}^{1001}\overset{i.i.d.}{\sim}\mathcal{N}(0,1)$. We set $\lambda_{k}^{*,(\alpha)} =c_{\alpha} k^{-(1+\alpha)}$, $\psi_{1}^*(t)= 1, \psi_{2k}^*(t) = \sqrt{2} \cos(2 \pi k t)$, $\psi_{2k+1}^*(t) = \sqrt{2} \sin(2 \pi k t)$, for $k\in\mathbb{N}\setminus\left\{0\right\}$. To study the effect of the regularity of the process (and the associated eigengap), we fix the total energy of the signal and consider a noise to signal ratio of $10$\%, thus we ensure that:
	$\sum_{k=1}^{1000}\lambda_{k}^{*,(\alpha)}=1$, for all $\alpha$. This model is not adversarial, and in practice the decay of the error in~$p$ turns out to be much faster than the minimax rate~$p^{-2\alpha}$, illustrating that one can generally expect more favorable rates than the worst case.
	\subsection{Minimax aliasing model}\label{subsec:minimax_model}
	Unlike the Fourier model, the minimax aliasing scenario does not rely on a single fixed process observed on increasingly fine grids. Instead, for each grid size~$p$ and regularity parameter~$\alpha$, a different process $X^{(p,\alpha)}$ is constructed, so that the distribution of the process depends explicitly on~$p$. This is consistent with the minimax interpretation: the model plays the role of a least-favourable element. The signal is
	\[
	X_i^{(p,\alpha)}(t)
	=
	\sqrt{\lambda_1^{*,(p,\alpha)}}\,\xi_{i1}\,\psi_1^{*,(p,\alpha)}(t)
	+
	\sqrt{\lambda_2^{*,(p,\alpha)}}\,\xi_{i2}\,\psi_2^{*,(p,\alpha)}(t),
	\qquad
	(\xi_{i1},\xi_{i2})\overset{i.i.d.}{\sim}\mathcal N(0,I_2).
	\]
	The first eigenelement is fixed:
	$\psi_1^{*,(p,\alpha)}(t)=\sqrt{2}\sin(2\pi t)$ and
	$\lambda_1^{*,(p,\alpha)}=0.60$.  The second eigenelement is defined via
	\[
	r_{p,\alpha}:=p^{-2\alpha},
	\qquad
	\psi_2^{*,(p,\alpha)}(t)
	=
	(1-r_{p,\alpha})
	+
	\sqrt{1-(1-r_{p,\alpha})^2}\;\sqrt{2}\sin(2\pi p\, t),
	\]
	and
	$\lambda_2^{*,(p,\alpha)}=(1-r_{p,\alpha})^{-2}\cdot 0.25$.
	
	\paragraph{Aliasing mechanism.}
	The key property is that the oscillatory part of $\psi_2^{*,(p,\alpha)}$ vanishes exactly on the observation grid, since
	$\sin(2\pi p\, t_{j})=\sin(2\pi(j-1))=0$.
	Hence, on the grid, the second eigenfunction is seen as a constant profile equal to $1-r_{p,\alpha}$, and the choice of $\lambda_2^{*,(p,\alpha)}$ ensures that the discrete contribution of this second mode to the covariance is always $0.25\,\mathbf{1}\mathbf{1}^\top$, regardless of~$\alpha$. Thus, $\alpha$ only controls the amount of variations in eigenfunctions that can not be observed due to the discrete design as:
	\[
	\|\psi_2^{*,(p,\alpha)}-1\|_{L^2([0,1])}^2 = 2\,p^{-2\alpha}.
	\]

	This scenario should therefore be read as an adversarial test of minimax type: it illustrates the intrinsic difficulty of off-grid reconstruction predicted by the theory. In the code, estimation is carried out by focusing on the second eigenelement.
	\subsection{Error assessment and estimation procedure}
	
	\textbf{Error assessment.} In all simulations, we focus on the estimation of the second eigenelement. For the second eigenfunction, we use the same sign convention as in the theoretical part and compare the estimator to $\psi_{2,\pm}^*$. For a given Monte Carlo replicate, the $\mathbb{L}_2([0,1])$-error is evaluated numerically on the fine grid $\{j/4096:0\leq j\leq 4095\}$ as
	\[
	\left \|\widehat{\widetilde{\psi_{2,\phi}}}-\psi_{2,\pm}^* \right \|^2
	\approx
	\frac{1}{4096}\sum_{j=0}^{4095}
	\left(
	\widehat{\widetilde{\psi_{2,\phi}}}\left(\frac{j}{4096}\right)-\psi_{2,\pm}^*\left(\frac{j}{4096}\right)
	\right)^2.
	\]
	We report the \emph{population squared error} and the \emph{mean squared error (MSE)}, defined by
	\[
	\mathrm{Population}(p,\alpha):=
	\left \| \widetilde{\psi_{2,\phi}}-\psi_{2,\pm}^* \right \|^2,
	\qquad
	\mathrm{MSE}(n,p,\alpha):=
	\mathbb{E}\left[
	\left \|\widehat{\widetilde{\psi_{2,\phi}}}-\psi_{2,\pm}^* \right \|^2
	\right].
	\]
	Here $\widetilde{\psi_{2,\phi}}$ denotes the second eigenfunction of the population projected operator $\Gamma_{\widetilde{\Pi}_p}$, so that $\mathrm{Population}(p,\alpha)$ isolates the pure discretisation error. In practice, $\mathrm{MSE}(n,p,\alpha)$ is approximated by the empirical average over $40$ Monte Carlo replicates.
	
	For the second eigenvalue, in line with the theoretical results, we consider the relative squared error (RSE) and report both its population and empirical mean versions:
	\[
	\mathrm{Population\text{-}RSE}(p,\alpha):=
	\left(
	\frac{\widetilde{\lambda_{2,\phi}}-\lambda_2^*}{\lambda_2^*}
	\right)^2,
	\qquad
	\mathrm{Mean\text{-}RSE}(n,p,\alpha):=
	\mathbb{E}\left[
	\left(
	\frac{\widehat{\widetilde{\lambda_{2,\phi}}}-\lambda_2^*}{\lambda_2^*}
	\right)^2
	\right].
	\]
	Here $\widetilde{\lambda_{2,\phi}}$ denotes the second eigenvalue of the population projected operator $\Gamma_{\widetilde{\Pi}_p}$. Thus, $\mathrm{Population\text{-}RSE}(p,\alpha)$ measures the pure discretisation error, whereas $\mathrm{Mean\text{-}RSE}(n,p,\alpha)$ also incorporates statistical variability.

	\textbf{Estimation Procedure.} A specific approximation is available for the coefficients of the $\log_2(p)$-th coiflets approximation level representation (see \eqref{eq:approx_proj}). To compute our estimator, we transform our data into its representation in the $\log(p)$ scale function basis using our approximation of the scalar products. With this finite-dimensional representation of the observations, we can define the empirical version of the covariance kernel projected onto the wavelet basis, which is then a matrix. Then to proceed we can simply diagonalize this matrix to obtain the representations of the eigenfunctions sought in the wavelet basis as allowed by Lemma \ref{lem:spectral_relation}.
	Then, to reconstruct the eigenfunctions from their representation in the basis, we use the \textit{pywt.Wavelet} function (from \cite{pywt}) to compute the coiflet mother wavelet function $\phi$ (we used \textit{coif2}, coiflets with only 2 vanishing moments), which we evaluate on a very fine dyadic grid, and from this function we compute the $\log_2(p)$-th scale functions $\left(\phi_{\log_2(p),k}\right)_{k=0}^{p-1}$ on the chosen error evaluation grid.
	Before using the coefficients of the eigenfunctions in the basis and the scaling functions to reconstruct the eigenfunctions on the error evaluation grid, we performed a
	symmetric reflection of the vectors of coefficients to allow the scaling functions to overlap boundaries without wrapping or zeroing, enabling accurate evaluation near edges.

	\subsection{Numerical results}
	We investigate the effect of the grid size by simulating processes on regular grids $t_j=\frac{j-1}{p},1\leq j\leq p$, with $p\in\{2^3,\dots,2^{10}\}$. Then we set $n\in\{2^9, \dots, 2^{14}\}$, $\alpha\in\{1,1.5\}$ and $\sigma^2 = 0.1$. Each configuration is simulated $40$ times.
	
	\paragraph{Effect of sample size~$n$.}
	Firstly, the decay of the empirical Mean Squared Error (MSE) matches the theoretical decay in sample size for all regularities (see Figure~\ref{fig:mse_vs_n}), as the error decays in $n^{-1}$. This holds for both models: the minimax aliasing construction does not alter the statistical order in~$n$.
	
	\begin{figure}
		\centering
		\includegraphics[width=0.85\textwidth]{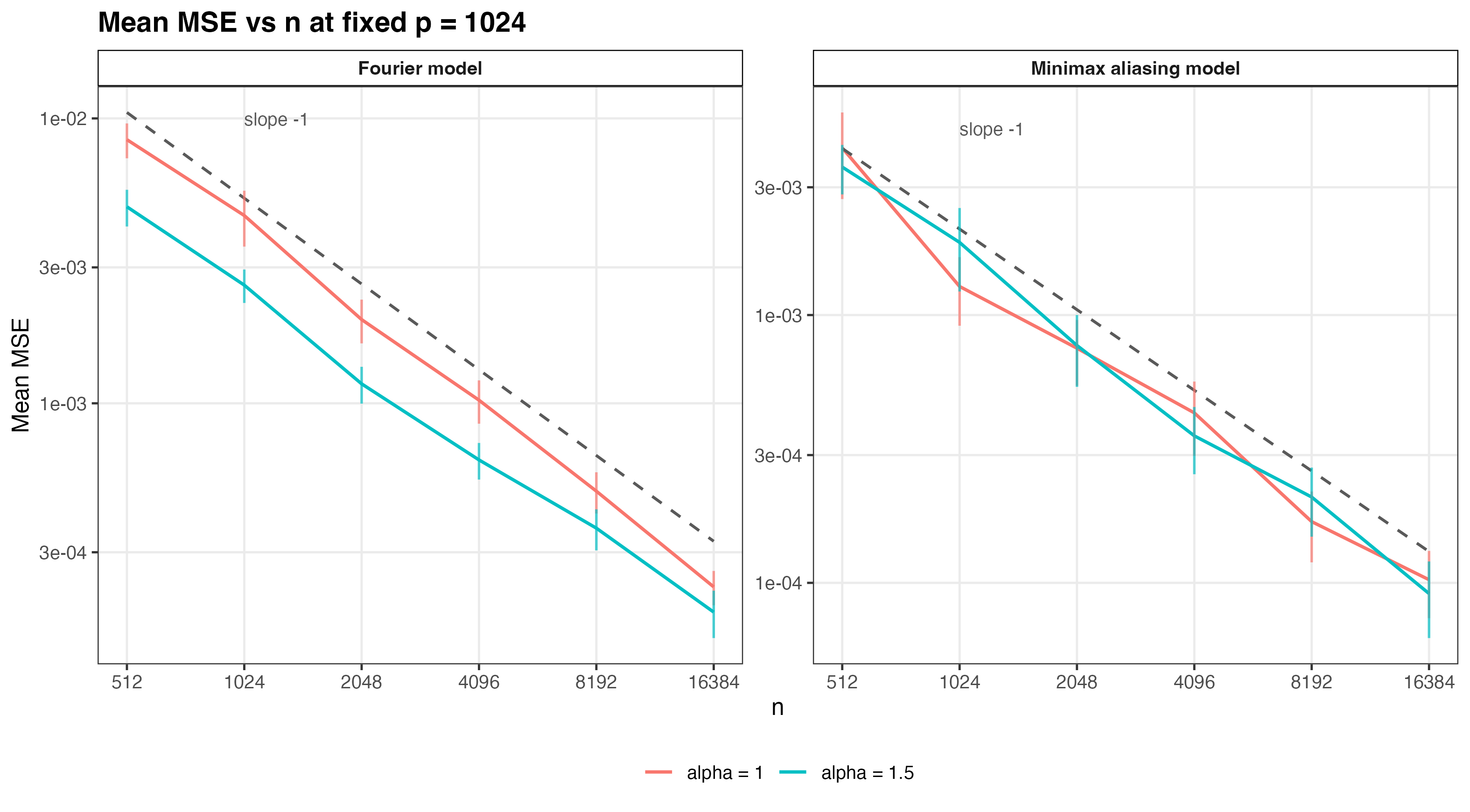}
		\caption{Mean MSE against sample size~$n$ (log-log scale) for both models and
			regularities $\alpha \in \{1,1.5\}$, at fixed grid size $p=1024$ and noise level
			$\sigma^2=0.1$.}
		\label{fig:mse_vs_n}
	\end{figure}
	
	\paragraph{Effect of grid size~$p$ -- population level.}
	Concerning the decay of the MSE with respect to the grid size, the population MSE (Figure~\ref{fig:pop_mse_vs_p}) isolates the discretisation error. In the minimax aliasing model, the curves and slopes are aligned with the reference slope~$p^{-2\alpha}$, confirming that this scenario saturates the theoretical discretisation bound predicted by the minimax theory. In the Fourier model, the decay is significantly faster than~$p^{-2\alpha}$, illustrating that one can generally expect more favourable rates than the worst case. In fact, the decay observed is $p^{-4}$ regardless of $\alpha$ which seems to follow the best approximation error reachable with the \textit{coif2} used. The numerical results also show that the greater the regularity, the smaller the MSE is for a fixed sample size and grid size, which is indeed what our theory predicted.
	
	\begin{figure}
		\centering
		\includegraphics[width=0.85\textwidth]{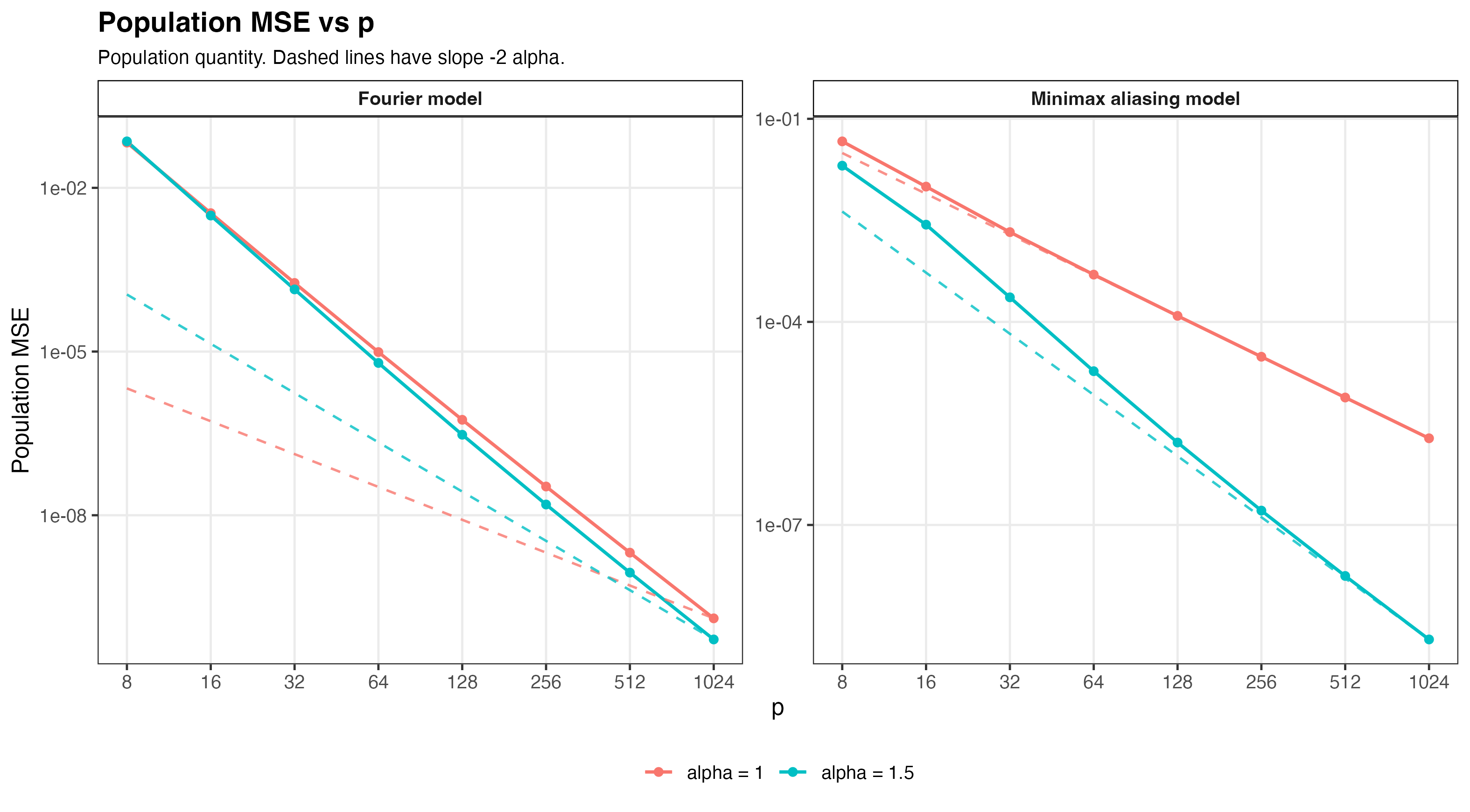}
		\caption{Population eigenfunction MSE against grid size~$p$ (log-log scale) for
			both models and $\alpha \in \{1,1.5\}$. Reference lines
			of slope~$-2\alpha$ are shown.}
		\label{fig:pop_mse_vs_p}
	\end{figure}
	
	\paragraph{Effect of grid size~$p$ -- empirical level.}
	Figure~\ref{fig:mean_mse_vs_p} shows the empirical counterpart (mean MSE over 40 replicates). For small to moderate values of~$p$, the curves track the discretisation regime with slopes close to~$2\alpha$. As~$p$ increases at fixed~$n$, the curves level off and the local slopes drop towards~$0$: the error then becomes dominated by the statistical noise~$1/n$ rather than by the discretisation, which corresponds to the phase transition. The theoretical transition occurs when $p^{-2\alpha}\approx n^{-1}$, i.e.\ around $p^*\approx n^{1/(2\alpha)}$. With $n=16384=2^{14}$, this gives $p^*\approx 2^7=128$ for $\alpha=1$ and $p^*\approx 2^{14/3}\approx 25$ for $\alpha=1.5$. On the minimax aliasing panel of Figure~\ref{fig:mean_mse_vs_p}, the plateau indeed sets in around $p\approx 64$--$128$ for $\alpha=1$ and $p\approx 32$--$64$ for $\alpha=1.5$, in agreement with these predictions.
	
	\begin{figure}
		\centering
		\includegraphics[width=0.85\textwidth]{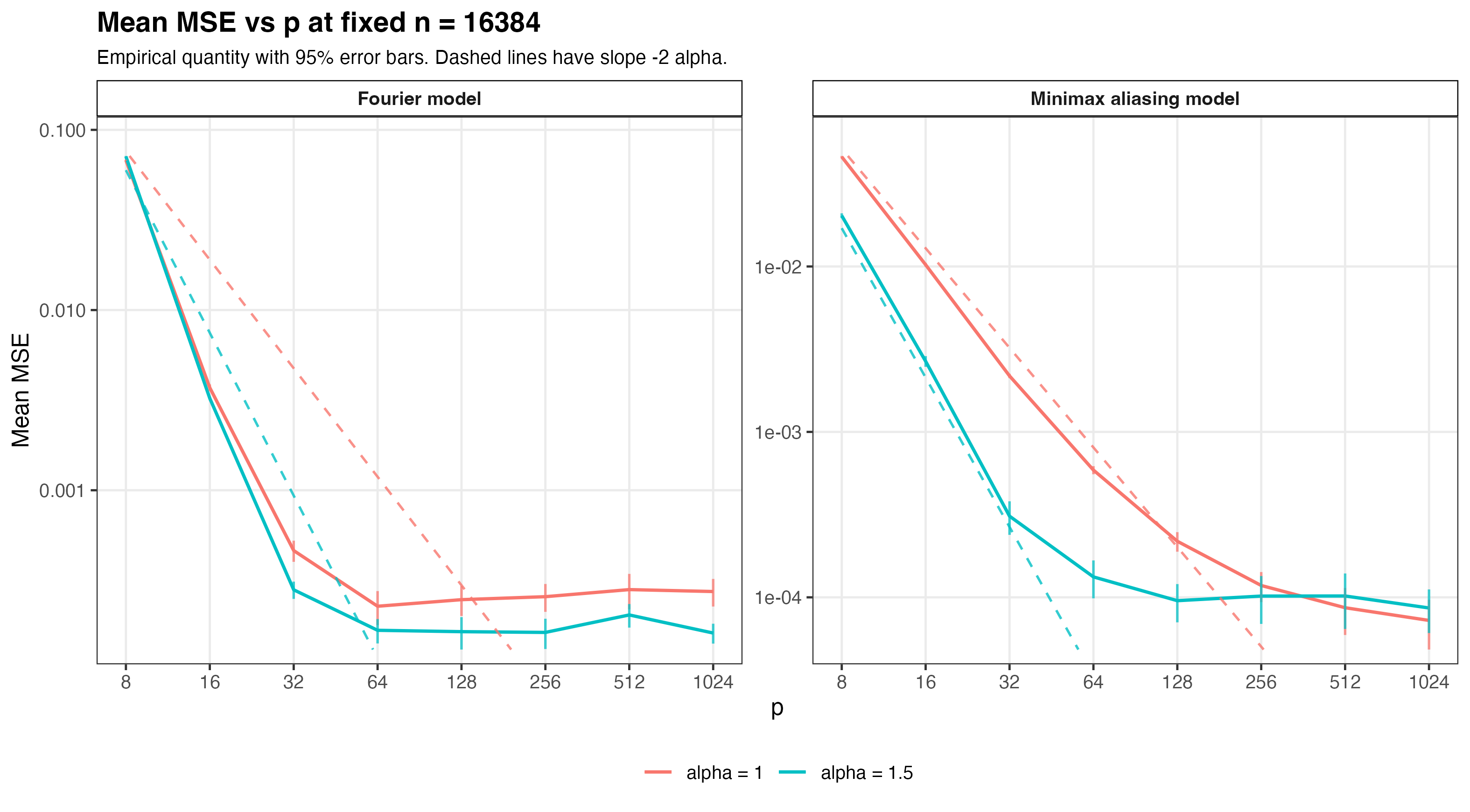}
		\caption{Mean eigenfunction MSE against grid size~$p$ (log-log scale) for both
			models and $\alpha \in \{1,1.5\}$, at fixed $n=16384$.}
		\label{fig:mean_mse_vs_p}
	\end{figure}
	
	\paragraph{Eigenvalue error (minimax aliasing model).}
	Figures~\ref{fig:eigval_vs_p} focuses on the population relative squared eigenvalue error in the minimax aliasing model. The error decays as~$p^{-4\alpha}$, with local slopes close to~$4\alpha$. This is consistent with the lower bound established in Theorem~\ref{Thm_eigval}.
	
	\begin{figure}
		\centering
		\includegraphics[width=0.85\textwidth]{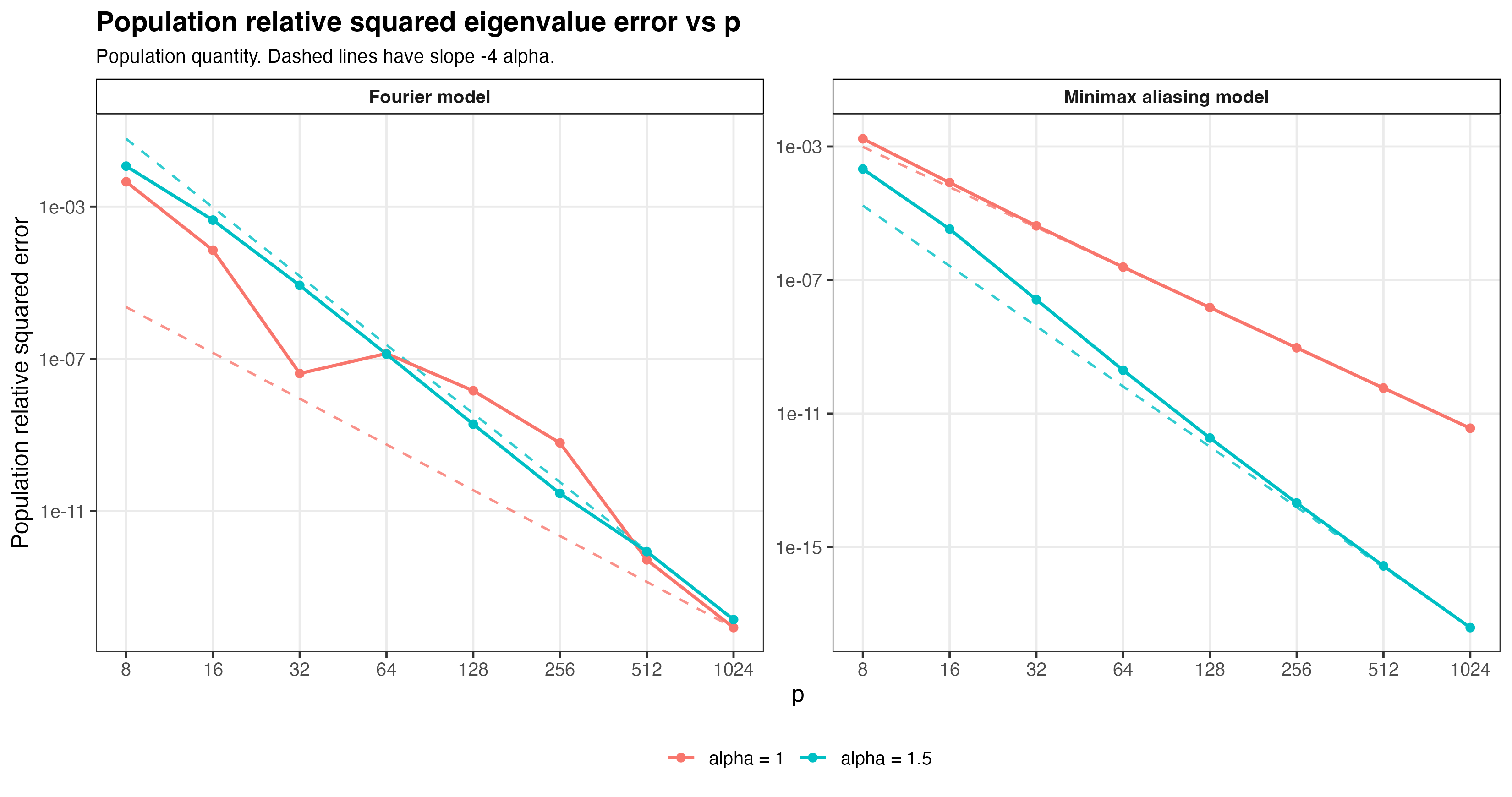}
		\caption{Population corrected relative squared eigenvalue error against~$p$
			(log-log scale) for the minimax aliasing model and
			$\alpha \in \{1,1.5\}$.}
		\label{fig:eigval_vs_p}
	\end{figure}
	\section*{Conclusion and perspectives}
	In this article we present the convergence rates for estimating the eigenvalues and eigenfunctions of the covariance operator of a sample of functional data that have been observed on a fixed grid. The rates that we prove for the eigenfunctions are minimax optimal. Our theoretical results apply to covariance operators of arbitrary smoothness, including those that are not twice differentiable. This notably covers classical processes such as Brownian motion or Brownian bridges whose covariance kernels are not differentiable on the diagonal. In particular, we extend the recent work \cite{zhou2024theoryFPCA}.\\
	However, further improvement could be made in order to understand the exact dependence of the constants on the rank $\ell$ of the estimated eigenfunction. To the best of our knowledge, this problem remains open. In the case of fully observed functional data, \cite{mas} 
	showed that the convergence rate with respect to the rank $\ell$ is of order $\ell^2$ up to logarithmic factors, under polynomial eigenvalue decay. This rate has been proven to be optimal for standard estimators based on the empirical covariance operator. In the discretely observed setting, the intricate interplay between the sample size, the grid resolution and the regularity of the kernel and/or the functional data makes it difficult to characterize precisely the dependence of the constants in the upper and lower bounds on the rank. This challenging problem is left for future research. \\
	Another challenging issue is determining the minimax convergence rate for estimating eigenvalues. This is left for further research. \\
			\section*{Acknowledgements}
			The authors acknowledge the support of the French Agence Nationale de la Recherche (ANR) under reference ANR-24-CE40-2439 (FUNMathStatproject). 
			\appendix
			\section{Proofs}\label{sec:proof}
			\subsection{Proof of Theorem~\ref{thm:inconsistency}}\label{subsec:proof_inconsistency}
			Let $n\in\mathbb{N}\setminus\{0\},p\in\mathbb{N}\setminus\{0\}$ such that $p\geq 4$, let $\ell\in\mathbb{N}$ such that $\ell+1<\frac{p}{2}$ and let $b_n\in]0,1[$ that we will choose later on. Let $\beta\in(0,1]$ and $m\in\mathbb{N}$ and define $\alpha=m+\beta$.
			We introduce for $j\geq $1,
			\[
			\begin{aligned}
				a_j : [0,1] &\longrightarrow [-\sqrt{2},\sqrt{2}] \\
				x &\longmapsto \sqrt{2}\sin(2\pi j x).
			\end{aligned}
			\]
			For $k\in\{0,1\},j\in[\![1,\ell -1]\!],t \in[0,1]$ let
			$$\psi_{j,k}^{*}(t):=a_{j}(t),$$
			and
			\begin{equation}\label{eq:def_etas}
				\begin{aligned}
					&\psi_{\ell,0}^{*}(t):=a_{\ell}(t),\quad &&\psi_{\ell+1,0}^{*}(t):=a_{\ell+1}(t),\\
					&\psi_{\ell,1}^{*}(t):=b_na_{\ell}(t)-\sqrt{1-b_n^2} a_{\ell+1}(t), \quad&&\psi_{\ell+1,1}^{*}(t):=b_na_{\ell+1}(t)+ \sqrt{1-b_n^2} a_\ell (t).\\
				\end{aligned}
			\end{equation}
			
			We introduce for all $k\in\{0,1\}$, $j\in[\![1,\ell+1]\!]$, $\xi_{j,k}\overset{i.i.d.}{\sim}  \mathcal{N}\left ( 0,1 \right )$ and for all $(t,u)\in [0,1]^2$:
			\begin{equation*}
				Z_k(t):=\sum_{j=1}^{\ell+1 }\sqrt{\lambda_{j,k}^*}\xi_{j,k} \psi_{j,k}^*(t)\quad\text{ and }\quad
				K_k(u,t):=\mathbb{E}\left[Z_k(t)Z_k(u)\right], 
			\end{equation*}
			with $\lambda_{1,k}^*>\hdots>\lambda_{\ell,k}^*>\lambda_{\ell+1,k}^*$ such that 
			\begin{eqnarray}\label{eq:cond_val_p_n}
				4(2\pi)^{\alpha}\left(\sum_{j=1}^{\ell-1}j^{\alpha}\lambda_{j,k}^*+2\lambda_{\ell,k}^*\left(\ell^{\alpha}+(\ell+1)^{\alpha}\right)\right)\leq L  &\text{ for }&k=0,1\nonumber\\
				\lambda_{j,0}^*=\lambda_{j,1}^* &\text{ for }&j\in[\![1,\ell-1[\!],
			\end{eqnarray}
			and
			
			As $K_k(u,t):=\sum_{j=1}^{\ell+1 }\lambda_{j,k}^* \psi_{j,k}^*(t)\psi_{j,k}^*(u)$ we get that for any given $(t,u)\in]0,1[^2$,
			\begin{equation}\label{eq:derivative_K}
				\frac{\partial^m K_k}{\partial u^m}(u,t)=\sum_{j=1}^{\ell+1 }\lambda_{j,k}^* \psi_{j,k}^*(t)\frac{\partial^m}{\partial u^m} \psi^*_{j,k}(u).
			\end{equation}
			
			Hence for $(u,u',t)\in ]0,1[^3$, on the one hand,
			\begin{align*}
				\left| \frac{\partial^m K_0}{\partial u^m}(u,t) - \frac{\partial^m K_0}{\partial u^m}(u',t)\right|
				&\leq \sum_{j=1}^{\ell+1}\lambda_{j,0}^*\left | \psi^*_{j,0}(t) \right |\left | \frac{\partial^m \psi^*_{j,0}}{u^m}(u)-\frac{\partial^m \psi^*_{j,0}}{u^m}(u') \right |\\
				&\leq \sum_{j=1}^{\ell+1}\lambda_{j,0}^* \left| a_j(t) \right| \left| a_j^{(m)}(u) - a_j^{(m)}(u') \right|
			\end{align*}
			as $\psi^*_{j,0}=a_j$.
			And using $\underset{x\in [0,1]}{\operatorname{max}} (a_j(x))=\sqrt{2}$ and Lemma \ref{lemm:trigo} we obtain
			\begin{equation}\label{eq:cond_K_0_1}
				\begin{aligned}
					\left| \frac{\partial^m K_0}{\partial u^m}(u,t) - \frac{\partial^m K_0}{\partial u^m}(u',t)\right|
					&\leq  4\sum_{j=1}^{\ell+1 }(2\pi j)^{\alpha}\lambda_{j,0}^*\left| u-u' \right|^{\beta}\\
					&\leq L\left| u-u' \right|^{\beta}, \text{by Equation \eqref{eq:cond_val_p_n}.}
				\end{aligned}.
			\end{equation}
			
			On the other hand, using the same arguments we can bound
			\begin{eqnarray*}
				\left| \frac{\partial^m K_1}{\partial u^m}(u,t) - \frac{\partial^m K_1}{\partial u^m}(u',t)\right|
				&  \leq& \sum_{j=1}^{\ell+1 }\lambda_{j,1}^*\left | \psi^*_{j,1}(t) \right |\left | \frac{\partial^m \psi^*_{j,1}}{u^m}(u)-\frac{\partial^m \psi^*_{j,1}}{u^m}(u') \right |\\
				&\hspace{-6cm}\leq & \hspace{-3cm}\sum_{j=1}^{\ell-1 }\lambda_{j,1}^*\left | \psi^*_{j,1}(t) \right |\left | a_j^{(m)}(u)-a_j^{(m)}(u') \right | \\
				&&\hspace{-3cm}+ \lambda_{\ell,1}^*\left | \psi^*_{\ell,1}(t) \right |\left(b_n\left|a_\ell^{(m)}(u)-a_\ell^{(m)}(u')\right|+\sqrt{1-b_n^2}\left|a_{\ell+1}^{(m)}(u)-a_{\ell+1}^{(m)}(u')\right|\right)\\
				&&\hspace{-3cm}+ \lambda_{\ell+1,1}^*\left | \psi^*_{\ell+1,1}(t) \right |\left(b_n\left|a_{\ell+1}^{(m)}(u)-a_{\ell+1}^{(m)}(u')\right|+\sqrt{1-b_n^2}\left|a_{\ell}^{(m)}(u)-a_{\ell}^{(m)}(u')\right|\right).
			\end{eqnarray*}
			Combining the last equation with \eqref{eq:def_etas} and Lemma \ref{lemm:trigo} we obtain
			\begin{equation*}
				\begin{aligned}
					\left| \frac{\partial^m K_1}{\partial u^m}(u,t) - \frac{\partial^m K_1}{\partial u^m}(u',t)\right|
					&\leq 4\left(\sum_{j=1}^{\ell -1}(2\pi j)^{\alpha}\lambda_{j,1}^* + \left( (2\pi(\ell+1))^{\alpha}\left ( \lambda_{\ell+1,1}^*b_n+\lambda_{\ell,1}^*\sqrt{1-b_n^2} \right )\right.\right.\\
					&\quad\quad\quad\left.+(2\pi \ell)^{\alpha}\left( \lambda_{\ell+1,1}^*\sqrt{1-b_n^2}+\lambda_{\ell,1}^*b_n \right)
					\right)\Bigg) \left | u-u' \right |^{\beta},
				\end{aligned}
			\end{equation*}
			and as $b_n\in]0,1[$ we have 
			\begin{equation*}\label{eq:cond_K_1_1}
				\begin{aligned}
					\left| \frac{\partial^m K_1}{\partial u^m}(u,t) - \frac{\partial^m K_1}{\partial u^m}(u',t)\right|
					&\leq 4(2\pi)^{\alpha}\left(\sum_{j=1}^{\ell -1}j^{\alpha}\lambda_{j,1}^* +
					\left((\ell+1)^{\alpha}+\ell ^{\alpha}\right)\left( \lambda_{\ell+1,1}^*+\lambda_{\ell,1}^* \right)
					\right) \left | u-u' \right |^{\beta}\\
					&\leq 4(2\pi)^{\alpha}\left(\sum_{j=1}^{\ell -1}j^{\alpha}\lambda_{j,1}^* +
					2\lambda_{\ell,1}^*\left((\ell+1)^{\alpha}+\ell ^{\alpha}\right)
					\right) \left | u-u' \right |^{\beta}\\        
					&\leq L\left| u-u' \right|^{\beta}, \text{by Equation \eqref{eq:cond_val_p_n}.}
				\end{aligned}
			\end{equation*}
			Then $P_{Z_k} \in \mathcal{P}(\alpha,L)$ for all $k\in\{0,1\}$ and we have
			$$\inf _{\widehat{\psi}_{\ell}} \sup _{P_{X} \in \mathcal{P}(\alpha,L)} \mathbb{E}\left[\left\|\widehat{\psi}_{\ell}-\psi_{\ell}(\Gamma_{P_X})\right\|^{2}\right] \geq \inf _{\hat{\psi}_{\ell}} \sup _{k=0,1} \mathbb{E}\left[\left\|\widehat{\psi}_{\ell}-\psi_{\ell, k}^{*}\right\|^{2}\right].$$
			We can therefore focus on the control of $\inf _{\hat{\psi}_{\ell}} \sup _{k=0,1} \mathbb{E}\left[\left\|\widehat{\psi}_{\ell}-\psi_{\ell, k}^{*}\right\|^{2}\right]$.
			
			Let $\widehat{\psi}_{\ell}$ be an estimator and $\widehat{\phi}$ defined by
			$$
			\widehat{\phi}:=\arg \min _{k=0,1}\left\|\widehat{\psi}_{\ell}-\psi_{\ell, k}^{*}\right\|^{2},
			$$
			we have for $k=0,1$,
			$$
			\left\|\widehat{\psi}_{\ell}-\psi_{\ell, k}^{*}\right\| \geq \frac{1}{2}\left\|\psi_{\ell, \widehat{\phi}}^{*}-\psi_{\ell, k}^{*}\right\|.
			$$
			Using the orthogonality of $a_\ell$ and $a_{\ell+1}$, we have
			$$
			\begin{aligned}
				\left\|\widehat{\psi}_{\ell}-\psi_{\ell, k}^{*}\right\|^{2} &\geq \frac{1}{4}\mathbf{1}_{\{\widehat{\phi} \neq k\}}\left\|\psi_{\ell,0}^{*}-\psi_{\ell,1}^{*}\right\|^{2}
				\geq\frac{1}{4}\mathbf{1}_{\{\widehat{\phi} \neq k\}}\left\|\left ( 1-b_n \right ) a_{\ell}+\sqrt{1-b_n^2}a_{\ell+1} \right\|^{2}\\
				&\geq\frac{1}{4}\mathbf{1}_{\{\widehat{\phi} \neq k\}}\left(\left ( 1-b_n \right )^2+(1-b_n^2)\right)\\
				&\geq\frac{1}{2}\mathbf{1}_{\{\widehat{\phi} \neq k\}}\left ( 1-b_n \right ).
			\end{aligned}
			$$
			Then,
			\begin{equation}\label{eq:borne_prop_1_proof}
				\inf _{\widehat{\psi}_{\ell}} \sup _{P_{X} \in \mathcal{P}(\alpha,L)} \mathbb{E}\left[\left\|\widehat{\psi}_{\ell}-\psi_{\ell}(\Gamma_{P_X})\right\|^{2}\right] \geq \frac{1}{2}(1-b_n)  \inf _{\widehat{\phi}} \max _{k=0,1} \mathbb{P}(\widehat{\phi} \neq k).
			\end{equation}
			We now establish a lower bound on $\mathbb{P}(\widehat{\phi} \neq k).$ This is achieved by applying Theorem 2.2 of \cite{Tsybakov2009IntroductionTN}, which is restated in the appendix as Lemma~\ref{lemm1}. 
			Our objective is to prove the existence of a constant $H_{\max}^2 < 2$ such that
			\begin{equation}\label{eq:cond_H}
				H^{2}\left(\left(P_{0}^{\mathrm{obs}}\right)^{\otimes n},
				\left(P_{1}^{\mathrm{obs}}\right)^{\otimes n}\right)
				\le H_{\max}^{2}.
			\end{equation}
			Here, for two probability measures $P$ and $Q$ dominated by the Lebesgue measure $\mu$, with respective densities $p$ and $q$ with respect to $\mu$, the Hellinger distance is defined by
			$$H(P,Q)
			\coloneqq
			\left(\frac{1}{2}\int (\sqrt{p}-\sqrt{q})^{2}\, d\mu \right)^{1/2}.$$
			Moreover, $P_{k}^{\mathrm{obs}}$ denotes the distribution of the random vector
			$$\mathbf{Y}_k^{\mathrm{obs}}:=
			\bigl(Y_k(t_1), \ldots, Y_k(t_p)\bigr),$$
			where
			$$Y_k(t_j) = Z_k(t_j) + \varepsilon_{k,j},$$
			with independent Gaussian noise variables $$\varepsilon_{k,j} \overset{\text{i.i.d.}}{\sim} \mathcal{N}(0,\sigma^2)$$ for $j \in [\![1,p]\!]$, and where $\sigma^2>0$ is fixed.
			Then, in this case by Lemma \ref{lemm1}, we would have
			\begin{equation*}
				\inf _{\widehat{\phi}} \max _{k=0,1} \mathbb{P}(\widehat{\phi} \neq k) \geq \frac{1}{2}\left(1-\sqrt{H_{\max }^{2}\left(1-H_{\max }^{2} / 4\right)}\right)>0.
			\end{equation*}
			In our case, we remark that:
			$$\mathbf{Y}_k^{ o b s}\sim \mathcal{N}(0,G_k),$$
			with
			\begin{equation}\label{eq:def_matrice_obs}
				\left[G_k\right]_{\ell',k'}=\mathbb{E}\left[Y_{k}\left(t_{\ell'}\right) Y_{k}\left(t_{k'}\right)\right]=\sum_{j=1}^{\ell+1} \lambda_{j,k}^*\psi_{j,k}^*(t_{k'})\psi_{j,k}^*(t_{\ell'})+\mathbf{1}_{\ell'=k'}\sigma^2.
			\end{equation}
			We shall subsequently use the identity
			\begin{equation}\label{eq:link_Hellinger_A}
				H^{2}\left(\left(P_{0}^{\mathrm{obs}}\right)^{\otimes n},
				\left(P_{1}^{\mathrm{obs}}\right)^{\otimes n}\right)
				= 2 - 2A\left(P_{0}^{\mathrm{obs}}, P_{1}^{\mathrm{obs}}\right)^{n},
			\end{equation}
			where $$A(P,Q) \coloneqq \int \sqrt{pq}\, d\mu$$ denotes the Hellinger affinity between two probability measures $P$ and $Q$ that admit densities $p$ and $q$ with respect to the Lebesgue measure $\mu.$
			In our case where the variables are Gaussian with equal mean vectors, we have an explicit formula for the Hellinger affinity: By using Lemma \ref{lemm2}, we get
			\begin{equation}\label{eq:def_affinite_gauss}
				A\left(P_{0}^{o b s}, P_{1}^{o b s}\right)^n=\left(\frac{\operatorname{det}\left(G_{0} G_{1}\right)^{1 / 4}}{\operatorname{det}\left(\left(G_{0}+G_{1}\right) / 2\right)^{1 / 2}}\right)^n,
			\end{equation}
			and focus now on the control of this quantity.  
			
			Denote $\mathbf a_j = (a_j(t_1),\hdots,a_j(t_{p}))^t$, using \eqref{eq:def_matrice_obs} and Lemma \ref{lemm:trigo} we have
			$$G_0=\sum_{j=1}^{\ell+1} \lambda_{j,0}^*\boldsymbol{a}_j\boldsymbol{a}_j^{t}+\sigma^2I_p$$
			and
			\begin{align*}
				G_1&=\sum_{j=1}^{\ell -1}\lambda_{j,1}^*\boldsymbol{a}_j\boldsymbol{a}_j^{t}+\lambda_{\ell,1}^*\left(b_n^2\boldsymbol{a}_{\ell}\boldsymbol{a}_{\ell}^t+\left ( 1-b_n^2 \right )\boldsymbol{a}_{\ell+1}\boldsymbol{a}_{\ell+1}^t-b_n\sqrt{1-b_n^2}\left(\boldsymbol{a}_{\ell}\boldsymbol{a}_{\ell+1}^t+\boldsymbol{a}_{\ell+1}\boldsymbol{a}_{\ell}^t\right)\right)\\
				&\quad +\lambda_{\ell+1,1}^*\left(\left ( 1-b_n^2 \right )\boldsymbol{a}_{\ell}\boldsymbol{a}_{\ell}^t+b_n^2\boldsymbol{a}_{\ell+1}\boldsymbol{a}_{\ell+1}^t+b_n\sqrt{1-b_n^2}\left(\boldsymbol{a}_{\ell}\boldsymbol{a}_{\ell+1}^t+\boldsymbol{a}_{\ell+1}\boldsymbol{a}_{\ell}^t\right)\right)+\sigma^2I_p\\
				&=\sum_{j=1}^{\ell -1}\lambda_{j,1}^*\boldsymbol{a}_j\boldsymbol{a}_j^{t}+\left ( \lambda_{\ell+1,1}^*+\Delta_1b_n^2 \right )\boldsymbol{a}_{\ell}\boldsymbol{a}_{\ell}^t+\left ( \lambda_{\ell,1}^*-\Delta_1b_n^2 \right )\boldsymbol{a}_{\ell+1}\boldsymbol{a}_{\ell+1}^t\\
				&\quad -\Delta_1 b_n\sqrt{1-b_n^2}\left(\boldsymbol{a}_{\ell}\boldsymbol{a}_{\ell+1}^t+\boldsymbol{a}_{\ell+1}\boldsymbol{a}_{\ell}^t\right)+\sigma^2I_p,
			\end{align*}
			where $I_p$ is the identity matrix of dimension $p$ and 
			\begin{equation*}
				\Delta_k:=\lambda_{\ell,k}^*-\lambda_{\ell+1,k}^*
			\end{equation*}
			for all $k\in\{0,1\}$ is a difference between consecutive eigenvalues of the $k$-th process.
			
			We then compute the determinants in \eqref{eq:def_affinite_gauss} by analyzing $G_0$ and $G_1$ in terms of eigen-values/vectors.

			For that purpose, we set $v_i:=\frac{\boldsymbol{a}_i}{\left \| \boldsymbol{a}_i \right \|_{\ell_2}}$ for all $i\in[\![1,\ell+1 ]\!]$ so that $\left\|v_{i}\right\|_{\ell_{2}}=1$ and $v_{\ell+2}, \ldots, v_{p}$ an orthonormal basis of $\operatorname{span}\left(v_{1},\dots ,v_{\ell}\right)^{\perp}$ in $\mathbb{R}^p$, and $V$ the orthogonal matrix
			$$
			V:=\left[v_{1} ; v_{2} ; \cdots ; v_{p}\right] .
			$$
			We set
			\begin{equation*}
				\begin{aligned}
					\kappa_j:&=(\lambda_{j,0}^*\left \| \boldsymbol{a}_j \right \|^2_{\ell_2}\mathbf{1}_{j\in[\![1,\ell+1 ]\!]}+\sigma^2),\\
					\kappa_{\ell}^{(\Delta_1)}&:=\left(\left ( \lambda_{\ell+1,1}^*+\Delta_{1}b_n^2 \right )\left \| \boldsymbol{a}_{\ell} \right \|^2_{\ell_2}+\sigma^2\right),\\
					\kappa_{\ell+1}^{(\Delta_1)}&:=\left(\left ( \lambda_{\ell,1}^*-\Delta_{1}b_n^2 \right )\left \| \boldsymbol{a}_{\ell+1} \right \|^2_{\ell_2}+\sigma^2\right).
				\end{aligned}
			\end{equation*}
			
			We can then write that
			$$G_0v_j=\kappa_j v_j$$
			and
			$$G_1v_j=\begin{cases}
				\kappa_j v_j & \text{ if } j\in [\![1,p]\!]\setminus\left\{\ell,\ell+1\right\}\\ 
				\kappa_{\ell}^{(\Delta_1)}v_{\ell}-\Delta_1b_n\sqrt{1-b_n^2}\left \| \boldsymbol{a}_{\ell} \right \|_{\ell_2}\left \| \boldsymbol{a}_{\ell+1} \right \|_{\ell_2}v_{\ell+1} & \text{ if } j=\ell \\ 
				\kappa_{\ell+1}^{(\Delta_1)}v_{\ell+1}-\Delta_1b_n\sqrt{1-b_n^2}\left \| \boldsymbol{a}_{\ell+1} \right \|_{\ell_2}\left \| \boldsymbol{a}_{\ell} \right \|_{\ell_2}v_{\ell}
				& \text{ if } j=\ell+1 . 
			\end{cases}\\
			$$
			
			Therefore by setting
			\begin{equation*}
				\left [ \tilde G_0 \right ]_{j,k}=\begin{cases}
					\kappa_j & \text{ if }  j=k \\ 
					0 & \text{ else}  
				\end{cases}
			\end{equation*}
			and 
			\begin{equation*}
				\left [ \tilde G_1 \right ]_{j,k}=\begin{cases}
					\kappa_j & \text{ if } j=k,j\in [\![1,p]\!]\setminus\left\{\ell,\ell+1\right\} \\ 
					\kappa_{\ell}^{(\Delta_1)}&\text{ if } j=k=\ell \\
					\kappa_{\ell+1}^{(\Delta_1)}&\text{ if } j=k=\ell+1  \\
					-\Delta_1b_n\sqrt{1-b_n^2}\left \| \boldsymbol{a}_{\ell+1} \right \|_{\ell_2}\left \| \boldsymbol{a}_{\ell} \right \|_{\ell_2}&\text{ if }k-1=j=\ell  \\
					-\Delta_1b_n\sqrt{1-b_n^2}\left \| \boldsymbol{a}_{\ell} \right \|_{\ell_2}\left \| \boldsymbol{a}_{\ell+1} \right \|_{\ell_2}&\text{ if }j-1=k=\ell  \\
					0 & \text{ else}  \\
				\end{cases},
			\end{equation*}
			we have
			$$
			G_{0}=V\tilde G_{0} V^{t},\quad
			G_{1}=V\tilde G_1 V^{t}.
			$$
			This implies that
			$$G_0G_1=V\left(\tilde G_{0}\tilde G_1\right) V^{t},\quad \frac{G_0+G_1}{2}=V\left(\frac{\tilde G_0+\tilde G_1}{2}\right) V^{t}.$$
			Then, using the above equation, we know that
			$$\left [ \tilde G_0\tilde G_1 \right ]_{j,k}=\begin{cases}
				\kappa_j^2 & \text{ if } j=k,j\in [\![1,p]\!]\setminus\left\{\ell,\ell+1\right\} \\ 
				\kappa_{\ell}\kappa_{\ell}^{(\Delta_1)}&\text{ if } j=k=\ell  \\
				\kappa_{\ell+1}\kappa_{\ell+1}^{(\Delta_1)}&\text{ if } j=k=\ell+1  \\
				-\kappa_{\ell+1}\Delta_1b_n\sqrt{1-b_n^2}\left \| \boldsymbol{a}_{\ell} \right \|_{\ell_2}\left \| \boldsymbol{a}_{\ell+1} \right \|_{\ell_2}&\text{ if }k-1=j=\ell  \\
				-\kappa_{\ell}\Delta_1b_n\sqrt{1-b_n^2}\left \| \boldsymbol{a}_{\ell+1} \right \|_{\ell_2}\left \| \boldsymbol{a}_{\ell} \right \|_{\ell_2}&\text{ if }j-1=k=\ell  \\
				0 & \text{ else}  \\
			\end{cases},$$
			and that
			\begin{equation*}
				\left [ \frac{\tilde G_0+\tilde G_1}{2} \right ]_{j,k}=\begin{cases}
					\kappa_j & \text{ if } j=k,j\in [\![1,p]\!]\setminus\left\{\ell,\ell+1\right\} \\ 
					\frac{1}{2}\left( \kappa_{\ell}+\kappa_{\ell}^{(\Delta_1)}\right)&\text{ if } j=k=\ell \\
					\frac{1}{2}\left( \kappa_{\ell+1}+\kappa_{\ell+1}^{(\Delta_1)}\right)&\text{ if } j=k=\ell+1  \\
					-\frac{1}{2}\Delta_1b_n\sqrt{1-b_n^2}\left \| \boldsymbol{a}_{\ell} \right \|_{\ell_2}\left \| \boldsymbol{a}_{\ell+1} \right \|_{\ell_2}&\text{ if }k-1=j=\ell  \\
					-\frac{1}{2}\Delta_1b_n\sqrt{1-b_n^2}\left \| \boldsymbol{a}_{\ell+1} \right \|_{\ell_2}\left \| \boldsymbol{a}_{\ell} \right \|_{\ell_2}&\text{ if }j-1=k=\ell  \\
					0 & \text{ else}  \\
				\end{cases},
			\end{equation*}
			so we are able to compute the determinants that are involved in \eqref{eq:def_affinite_gauss}.
			
			We have that
			\begin{equation}\label{eq:dets}
				\begin{aligned}
					\operatorname{det}\left (G_0G_1 \right )&=\left ( \underset{i\in [\![1,p]\!]\setminus\left\{\ell,\ell+1\right\}}{\prod}\kappa_i^2 \right )\left ( \kappa_{\ell}\kappa_{\ell+1}\kappa_{\ell}^{(\Delta_1)}\kappa_{\ell+1}^{(\Delta_1)}-\kappa_{\ell}\kappa_{\ell+1}\Delta_1^2b_n^2\left ( 1-b_n^2 \right )\left \| \boldsymbol{a}_{\ell} \right \|_{\ell_2}^2\left \| \boldsymbol{a}_{\ell+1} \right \|_{\ell_2}^2 \right ),\\ 
					\operatorname{det}\left ( \frac{G_0+G_1}{2} \right )&=\left ( \underset{i\in [\![1,p]\!]\setminus\left\{\ell,\ell+1\right\}}{\prod}\kappa_i \right )\frac{1}{4}\\
					&\times\left ( \left ( \kappa_{\ell}+\kappa_{\ell}^{(\Delta_1)} \right )\left ( \kappa_{\ell+1}+\kappa_{\ell+1}^{(\Delta_1)} \right )-\Delta_1^2b_n^2\left ( 1-b_n^2 \right )\left \| \boldsymbol{a}_{\ell} \right \|_{\ell_2}^2\left \| \boldsymbol{a}_{\ell+1} \right \|_{\ell_2}^2 \right ).
				\end{aligned}
			\end{equation}
			To ensure that \eqref{eq:def_affinite_gauss} is strictly positive which will imply \eqref{eq:cond_H} we set
			\begin{equation}\label{eq:setting_eigvals}
				\left\{\begin{matrix}
					\kappa_{\ell}^{(\Delta_1)}:=\kappa_{\ell}\\ 
					\kappa_{\ell+1}^{(\Delta_1)}:=\kappa_{\ell+1}
				\end{matrix}\right.
				\Leftrightarrow \left\{\begin{matrix}
					\lambda_{\ell,0}^*=\lambda_{\ell+1,1}^*+\Delta_{1}b_n^2\\ 
					\lambda_{\ell+1,0}^*=\lambda_{\ell,1}^*-\Delta_{1}b_n^2
				\end{matrix}\right.\\
				\Leftrightarrow \left\{\begin{matrix}
					\lambda_{\ell+1,1}^*=\lambda_{\ell,0}^*-\Delta_{1}b_n^2\\ 
					\lambda_{\ell,1}^*=\lambda_{\ell+1,0}^*+\Delta_{1}b_n^2
				\end{matrix}\right..
			\end{equation}
			This results in
			\begin{equation}\label{eq:Delta_1}
				\Delta_1=\lambda_{\ell,1}^*-\lambda_{\ell+1,1}^*=2\Delta_1b_n^2-\Delta_0\Leftrightarrow 
				\Delta_1 = \frac{-\Delta_0}{1-2b_n^2},
			\end{equation}
			as $\Delta_k>0$ we have the constraint
			\begin{equation*}\label{eq:constraint_b_n}
				1\geq b_n^2>\frac{1}{2}.
			\end{equation*}
			Then, by substituting the quantities defined in Equation \eqref{eq:setting_eigvals} into Equations \eqref{eq:dets} we can compute the affinity in \eqref{eq:def_affinite_gauss}.
			As a result, to ensure that the affinity is bounded from below away from zero it suffices to have $1>C>0$ such that
			\begin{equation*}
				A\left(P_{0}^{o b s}, P_{1}^{o b s}\right)^n=\left(\frac{\operatorname{det}\left(G_{0} G_{1}\right)^{1 / 4}}{\operatorname{det}\left(\left(G_{0}+G_{1}\right) / 2\right)^{1 / 2}}\right)^n
				=\left (\frac{\left ( \kappa_{\ell}^2\kappa_{\ell+1}^2-\kappa_{\ell}\kappa_{\ell+1}\Delta_1^2b_n^2\left ( 1-b_n^2 \right )\left \| \boldsymbol{a}_{\ell} \right \|_{\ell_2}^2\left \| \boldsymbol{a}_{\ell+1} \right \|_{\ell_2}^2 \right )^{\frac{1}{4}}}{\left ( \kappa_{\ell}\kappa_{\ell+1}-\frac{\Delta_1^2}{4}b_n^2\left ( 1-b_n^2 \right )\left \| \boldsymbol{a}_{\ell+1} \right \|_{\ell_2}^2\left \| \boldsymbol{a}_{\ell} \right \|_{\ell_2}^2 \right )^{\frac{1}{2}}}  \right )^n\geq C,
			\end{equation*}
			and as $\frac{1}{2}<b_n^2\leq 1$ it is sufficient to have
			\[
			\left (\frac{\left ( \kappa_{\ell}^2\kappa_{\ell+1}^2-\kappa_{\ell}\kappa_{\ell+1}\Delta_1^2\left ( 1-b_n^2 \right )\left \| \boldsymbol{a}_{\ell} \right \|_{\ell_2}^2\left \| \boldsymbol{a}_{\ell+1} \right \|_{\ell_2}^2 \right )^{\frac{1}{4}}}{\left ( \kappa_{\ell}\kappa_{\ell+1}-\frac{\Delta_1^2}{8}\left ( 1-b_n^2 \right )\left \| \boldsymbol{a}_{\ell} \right \|_{\ell_2}^2\left \| \boldsymbol{a}_{\ell+1} \right \|_{\ell_2}^2 \right )^{\frac{1}{2}}}  \right )^n\geq C
			\]
			which  is equivalent to 
			\[
			\left (C^{\frac{4}{n}}-1  \right )\kappa_{\ell}^2\kappa_{\ell+1}^2+\left (1-\frac{C^{\frac{4}{n}}}{4}  \right )\Delta_1^2\kappa_{\ell}\kappa_{\ell+1}z_\ell (1-b_n^2)+\frac{C^{\frac{4}{n}}}{64}\Delta_1^4z_\ell ^2(1-b_n^2)^2 \leq 0
			\]
			with $z_{\ell}:=\left \| \boldsymbol{a}_{\ell} \right \|^2_{\ell_2}\left \| \boldsymbol{a}_{\ell+1}  \right \|^2_{\ell_2}.$\\
			
			By interpreting the last inequality as an inequality involving a polynomial on the variable $(1-b_n^2)$ we can compute the determinant of the polynomial:
			$$\Delta=\left (1-\frac{C^{\frac{4}{n}}}{4}  \right )^2\Delta_1^4\kappa_{\ell}^2\kappa_{\ell+1}^2z_\ell ^2-\frac{C^{\frac{4}{n}}}{16}\left (C^{\frac{4}{n}}-1  \right )\Delta_1^4\kappa_{\ell}^2\kappa_{\ell+1}^2z_\ell ^2=\left (1-\frac{7}{16}C^{\frac{4}{n}}  \right )\Delta_1^4\kappa_{\ell}^2\kappa_{\ell+1}^2z_\ell ^2> 0.$$
			Hence we can set $(1-b_n^2)$ to be  close to the largest root (as we want the lower-bound in \eqref{eq:borne_prop_1_proof} to be the largest possible) :
			\begin{equation}\label{eq:def_b_n_eigfun}
					1-b_n^2:=\frac{-\left (1-\frac{C^{\frac{4}{n}}}{4}  \right )\Delta_1^2\kappa_{\ell}\kappa_{\ell+1}z_\ell +\sqrt{\Delta}}{\frac{C^{\frac{4}{n}}}{32}\Delta_1^4z_\ell ^2}=\frac{32{\kappa_{\ell}\kappa_{\ell+1} }}{\Delta_1^2z_\ell  C^{\frac{4}{n}}}\left ( \frac{C^\frac{4}{n}}{4} -1 + \sqrt{1-\frac{7}{16} C^{\frac{4}{n}}}\right ).
				\end{equation}
				We are then able to lower bound $1-b_n$ as $b_n\leq1$ we have
				\begin{equation*}
					1-b_n=\frac{1-b_n^2}{1+b_n}\geq \frac{1-b_n^2}{2}.
				\end{equation*}
				Then using \eqref{eq:def_b_n_eigfun} and the following minoration
				\begin{equation*}
					\begin{aligned}
						\frac{C^\frac{4}{n}}{4} -1 + \sqrt{1-\frac{7}{16} C^{\frac{4}{n}}}=\frac{1-\frac{7}{16} C^{\frac{4}{n}}-\left(1-\frac{C^\frac{4}{n}}{4}\right)^2}{1-\frac{C^\frac{4}{n}}{4} + \sqrt{1-\frac{7}{16} C^{\frac{4}{n}}}}\geq \frac{C^{\frac{4}{n}}\left(1-C^{\frac{4}{n}}\right)}{32}
					\end{aligned}
				\end{equation*}
				we obtain 
				\begin{equation*}
					1-b_n\geq \frac{\kappa_{\ell}\kappa_{\ell+1}}{2\Delta_1^2z_\ell}\left(1-C^{\frac{4}{n}}\right).
				\end{equation*}
				By Lemma \ref{lemm:trigo} and the definition of $\kappa_j$, 
				\begin{equation}\label{eq:prelast_1-b_n}
					\begin{aligned}
						1-b_n&\geq \frac{{(\lambda_{\ell,0}^*\left \| \boldsymbol{a}_{\ell} \right \|^2_{\ell_2}+\sigma^2)(\lambda_{\ell+1,0}^*\left \| \boldsymbol{a}_{\ell+1} \right \|^2_{\ell_2}+\sigma^2) }}{2\Delta_1^2\left \| \boldsymbol{a}_{\ell} \right \|^2_{\ell_2}\left \| \boldsymbol{a}_{\ell+1} \right \|^2_{\ell_2}}\left(1-C^{\frac{4}{n}}\right)\\
						&\geq  \frac{{(\lambda_{\ell,0}^*(p-1)+\sigma^2)(\lambda_{\ell+1,0}^*(p-1)+\sigma^2) }}{2\Delta_1^2(p-1)^2}\left(1-C^{\frac{4}{n}}\right)\\
						&\geq \frac{\lambda_{\ell,0}^*\lambda_{\ell+1,0}^* }{2\Delta_1^2}\left(1-C^{\frac{4}{n}}\right).
					\end{aligned}
				\end{equation}
				Then, by combining Equations \eqref{eq:Delta_1} and \eqref{eq:def_b_n_eigfun}, we obtain the following equation, 
				\begin{equation*}
					\Delta_1^2-\Delta_0\Delta_1-\frac{64\kappa_{\ell}\kappa_{\ell+1}}{z_\ell C^{\frac4n}}\left(\frac{C^{\frac4n}}4-1+\sqrt{1-\frac{7}{16}C^{\frac{4}n}}\right)=0. 
				\end{equation*}
				By solving this polynomial in $\Delta_1$, setting $\Delta_1$ to be its positive root and using Lemma \ref{lemm:trigo}) we obtain:
				\begin{equation*}
					\begin{aligned}
						\Delta_1&=\frac{1}{2}\left(\Delta_0+\sqrt{\Delta_0^2+\frac{256\kappa_{\ell}\kappa_{\ell+1} }{z_\ell C^{\frac4n}}\left(\frac{C^{\frac4n}}4-1+\sqrt{1-\frac{7}{16}C^{\frac{4}n}}\right)}\right)\\
						&= \frac{1}{2}\left(\Delta_0+\sqrt{\Delta_0^2+\frac{256{(\lambda_{\ell,0}^*\left \| \boldsymbol{a}_{\ell} \right \|^2_{\ell_2}+\sigma^2)(\lambda_{\ell+1,0}^*\left \| \boldsymbol{a}_{\ell+1} \right \|^2_{\ell_2}+\sigma^2) }}{\left \| \boldsymbol{a}_{\ell} \right \|^2_{\ell_2}\left \| \boldsymbol{a}_{\ell+1} \right \|^2_{\ell_2}C^{\frac{4}{n}}}\left(\frac{C^{\frac4n}}4-1+\sqrt{1-\frac{7}{16}C^{\frac{4}n}}\right)}\right)\\
						&= \frac{1}{2}\left(\Delta_0+\sqrt{\Delta_0^2+\frac{256{(\lambda_{\ell,0}^*(p-1)+\sigma^2)(\lambda_{\ell+1,0}^*(p-1)+\sigma^2) }}{(p-1)^2C^{\frac{4}{n}}}\left(\frac{C^{\frac4n}}4-1+\sqrt{1-\frac{7}{16}C^{\frac{4}n}}\right)}\right).
					\end{aligned}
				\end{equation*}
				Therefore, our last steps to bound from below \eqref{eq:prelast_1-b_n} which can be written as
				\begin{equation}\label{eq:last_1-b_n}
					1-b_n\geq \frac{2\lambda_{\ell,0}^*\lambda_{\ell+1,0}^* }{\left(\Delta_0+\sqrt{\Delta_0^2+\frac{256{(\lambda_{\ell,0}^*(p-1)+\sigma^2)(\lambda_{\ell+1,0}^*(p-1)+\sigma^2) }}{(p-1)^2C^{\frac{4}{n}}}\left(\frac{C^{\frac4n}}4-1+\sqrt{1-\frac{7}{16}C^{\frac{4}n}}\right)}\right)^2}\left(1-C^{\frac{4}{n}}\right),
				\end{equation}
				is to first observe that the function $f(x) = \frac{x}{4} - 1 + \sqrt{1-\frac{7}{16}x}$ is concave for $x \in [0,1]$. Since $f(1)=0$ and $f'(1) = -1/24$, the tangent inequality implies $f(x) \le \frac{1}{24}(1-x)$. Applied to $x=C^{4/n}$, this yields:
				\begin{equation}\label{eq:ineq_C^4/n}
					\frac{C^{\frac4n}}4-1+\sqrt{1-\frac{7}{16}C^{\frac{4}n}} \le \frac{1}{24}\left(1-C^{\frac{4}{n}}\right).
				\end{equation}
				Let 
				\begin{equation*}
					R :=\Delta_0^2+\frac{256{(\lambda_{\ell,0}^*(p-1)+\sigma^2)(\lambda_{\ell+1,0}^*(p-1)+\sigma^2) }}{(p-1)^2C^{\frac{4}{n}}}\left(\frac{C^{\frac4n}}4-1+\sqrt{1-\frac{7}{16}C^{\frac{4}n}}\right).
				\end{equation*}
				Substituting the Inequality \eqref{eq:ineq_C^4/n} we have:
				\begin{equation*}
					R \le \Delta_0^2 + \frac{32(\lambda_{\ell,0}^*(p-1)+\sigma^2)(\lambda_{\ell+1,0}^*(p-1)+\sigma^2)}{3(p-1)^2} \left(C^{-\frac{4}{n}} - 1\right).
				\end{equation*}
				Using the inequality $\sqrt{a^2+b} \le a + \frac{b}{2a}$ with $a=\Delta_0$, we can obtain the following upper bound:
				\begin{align*}
					(\Delta_0+\sqrt{R})^2 &\le \left( 2\Delta_0 + \frac{16(\lambda_{\ell,0}^*(p-1)+\sigma^2)(\lambda_{\ell+1,0}^*(p-1)+\sigma^2)}{3\Delta_0(p-1)^2} \left( C^{-\frac{4}{n}} - 1 \right) \right)^2 \nonumber \\
					&= 4\Delta_0^2 \left( 1 + \frac{8(\lambda_{\ell,0}^*(p-1)+\sigma^2)(\lambda_{\ell+1,0}^*(p-1)+\sigma^2)}{3\Delta_0^2(p-1)^2} \left( C^{-\frac{4}{n}} - 1 \right) \right)^2.
				\end{align*}
				Finally, Combining the last inequality with Equation \eqref{eq:last_1-b_n}:
				\begin{equation}\label{eq:LAST_1-b_n_with_n}
					1-b_n \ge \frac{ \lambda_{\ell,0}^*\lambda_{\ell+1,0}^* \left(1-C^{\frac{4}{n}}\right)}{2\Delta_0^2 \left[ 1 + \frac{8 (\lambda_{\ell,0}^*(p-1)+\sigma^2)(\lambda_{\ell+1,0}^*(p-1)+\sigma^2)}{3\Delta_0^2(p-1)^2} \left(C^{-\frac{4}{n}} - 1\right) \right]^2}.
				\end{equation}
				To obtain a lower bound independent of $n$, we set $\lambda_{\ell,0}^*=c_{L,m,\ell}\left(1+\frac{1}{\sqrt{n}}\right)$ and $\lambda_{\ell+1,0}^*=c_{L,m,\ell}$ with $c_{L,m,\ell}>0$ sufficiently small to ensure that \eqref{eq:cond_val_p_n} is satisfied.\\
				We then lower-bound the numerator terms using $\lambda_{\ell+1,0}^* \ge c_{L,m,\ell}$ and the inequality $1-C^{4/n} \ge \frac{4\ln(1/C)}{n} C^4$ (derived from $1-e^{-x} \ge xe^{-x}$ and $C^{4/n} \ge C^4$). For the denominator, we upper-bound the terms using $\lambda_{\ell,0}^* \le 2c_{L,m,\ell}$ and $(C^{-4/n}-1)/\Delta_0^2 \le \frac{4\ln(1/C)}{c_{L,m,\ell}^2} C^{-4}$ (using $e^x-1 \le xe^x$ and $C^{-4/n} \le C^{-4}$). Substituting these bounds yields:
				\begin{equation*}
					1-b_n\geq \frac{2 \ln(1/C) C^4}{\left[ 1 + \frac{32 \ln(1/C) C^{-4}}{3 c_{L,m,\ell}^2 (p-1)^2} (2c_{L,m,\ell}(p-1)+\sigma^2)(c_{L,m,\ell}(p-1)+\sigma^2) \right]^2}>0.
				\end{equation*}
				We can conclude as we ensured that with our particular choice of $b_n$ we have that
				for all $n\in\mathbb{N}$, $$A\left(P_{0}^{o b s}, P_{1}^{o b s}\right)^n\geq C>0,$$ which implies, by virtue of \eqref{eq:link_Hellinger_A}, the existence of $H_{max}^2<2$ such that \eqref{eq:cond_H} is verified. We can therefore apply Lemma \ref{lemm1} in combination with \eqref{eq:borne_prop_1_proof} and the last inequality to obtain for a constant $c>0$
				\begin{equation}\label{eq:near_lbound_eigfun}
					\inf _{\widehat{\psi}_{\ell}} \sup _{P_X \in \mathcal{P}(\alpha,L)} \mathbb{E}\left[\left\|\widehat{\psi}_{\ell}-\psi_{\ell}(\Gamma_{P_X})\right\|^{2}\right] \geq c.
				\end{equation}
				\subsection{Proof of Theorem~\ref{thm:eigfun_lb}}\label{subsec:proof_eigfun_lb}
				\begin{prop}{(Minimax lower-bound for the $\ell $-th eigenfunction estimation)}\\ \label{prop:eigfun_lb_n}
					Let $n\in\mathbb{N}\setminus\{0\},p\in\mathbb{N}$ such that $p\geq 4,\ell\in\mathbb{N}\setminus\{0\}$ such that $\ell+1<\frac{p}{2}$ and $\alpha,L,\delta_\ell\in\mathbb{R}^*_+$.
					We have that there exists a constant $c>0$ 
					such that
					$$\inf _{\widehat{\psi}_{\ell}} \sup _{P_{X} \in \mathcal{P}(\alpha,L,\delta_\ell)} \mathbb{E}\left[\left\|\widehat{\psi}_{\ell}-\psi_{\ell}(\Gamma_{P_X})\right\|^{2}\right] \geq c \frac{\delta_\ell}{n},$$
					where the infimum is taken over all measurable functions of the data.
				\end{prop}
				\begin{proof}
					The proof of Proposition~\ref{prop:eigfun_lb_n} follows the same steps as that of Theorem~\ref{thm:inconsistency} up to Equation~\eqref{eq:LAST_1-b_n_with_n}. From that point on, we impose
					$$\frac{\lambda^*_{\ell+1,0}\lambda^*_{\ell,0}}{\Delta_0^2} = \delta_\ell,$$
					by choosing the eigenvalues to exhibit a suitable exponential decay. Finally, the fact that $C^{4/n}$ and $C^{-4/n}$ converge to 1 at rate $1/n$ yields the desired result.
				\end{proof}
				
				\begin{prop}{(Minimax lower-bound for the $\ell $-th eigenfunction estimation)}\\\label{prop:eigfun_lb_p}
					Let $n\in\mathbb{N}\setminus\{0\},\alpha,L\in\mathbb{R}^*_+$. Let $p\in\mathbb{N}$ such that $p\geq \max\left(\left(2+\sqrt 2\right)^{\frac{1}{2\alpha}},4\right)$, $\ell\in\mathbb{N}\setminus\{0\}$ such that $ \ell+1<\frac{p}{2}$ and $\delta_\ell\in\mathbb{R}^*_+$.
					We have that there exists a constant $c>0$ only depending on $\alpha$ such that, 
					$$\inf _{\widehat{\psi}_{\ell}} \sup _{P_{X} \in \mathcal{P}(\alpha,L,\delta_\ell)} \mathbb{E}\left[\left\|\widehat{\psi}_{\ell}-\psi_{\ell}(\Gamma_{P_X})\right\|^{2}\right] \geq cp^{-2\alpha},$$
					where the infimum is taken over all measurable functions of the data.
				\end{prop}
				\begin{proof}
					Let $n\in \mathbb{N}\setminus\{0\},L\in\mathbb{R}^*_+$, $\beta\in(0,1]$ and $m\in\mathbb{N}$ and define $\alpha=m+\beta$. Let $p\in\mathbb{N}$ such that $p\geq \max\left(\left(2+\sqrt 2\right)^{\frac{1}{2\alpha}},4\right)$, $\ell\in\mathbb{N}$ such that $ 2\leq\ell+1<\frac{p}{2}$ (the case $\ell=1$ can be treated in the similar way), let $\delta_\ell\in\mathbb{R}^*_+$ and let $r_p\in(0,1/2]$ that will be calibrated later on.\\
					For $k\in\{0,1\},j\in[\![1,\ell -1]\!],t \in[0,1]$, let
					\begin{gather*}
						\psi_{j,k}^{*}(t):=a_{j}(t),\quad \psi_{\ell,0}^{*}(t):=1,
						\quad \psi_{\ell,1}^{*}(t):=1-r_p+\sqrt{1-(1-r_p)^2}a_{p}(t), 
					\end{gather*}
					with $a_j(t) := \sqrt 2\sin(2\pi j t),\, a_{p}(t):=\sqrt 2\sin(2\pi p t),$ as in the proof of Proposition~\ref{thm:inconsistency}, which ensures $ \left<\psi_{j,k}^*,\psi_{j',k}^* \right>_{\mathbb{L}_2\left([0,1]\right)}=\mathbf{1}_{j=j'}$.
					
					We introduce for all $k\in\{0,1\}$, $j\in[\![1,\ell ]\!]$, $\xi_{j,k}\overset{i.i.d.}{\sim}  \mathcal{N}\left ( 0,1 \right )$ and for all $t,s\in [0,1]$:
					\begin{equation*}
						Z_k(t):=\sum_{j=1}^{\ell }\sqrt{\lambda_{j,k}^*}\xi_{j,k} \psi_{j,k}^*(t)\quad\text{ and }\quad
						K_k:=\mathbb{E}\left[Z_k(s) Z_k(t)\right]=\sum_{j=1}^{\ell }\lambda_{j,k}^*\psi_{j,k}^*(t)\psi_{j,k}^*(s)
					\end{equation*}
					with
					\begin{gather*}
						\lambda_{1,0}^*>\dots>\lambda_{\ell-1,0}^*>\lambda_{\ell,0}^*>0,\\
						\lambda_{\ell,1}^*:=(1-r_p)^{-2}\lambda_{\ell,0}^*
						\text{ and for all } j\in[\![1,\ell-1 ]\!]\quad \lambda_{j,1}^*:=\lambda_{j,0}^*. 
					\end{gather*}
					To this end, set $\lambda_{j,0}^* := c_{\alpha,L,\ell}e^{-\gamma_\ell j}$, $c_{\alpha,L,\ell}>0$ will be set later in the proof. First we calibrate $\gamma_\ell>0$ to ensure that the constraint 
					\begin{equation*}
						r_\ell\left(\Gamma_{Z_k}\right)\leq \delta_\ell,
					\end{equation*}
					with $r_\ell\left(\Gamma_{Z_k}\right)$ defined as in \eqref{eq:relative_eigengap}.\\
					As $0<r_p\leq \frac{1}{2}$ we have $\lambda_{\ell,1}^*>\lambda_{\ell,0}^*$ and thus $r_\ell\left(\Gamma_{Z_0}\right)\leq r_\ell\left(\Gamma_{Z_1}\right)$, thus to ensure that the constraint is satisfied it suffices to ensure
					\begin{equation*}
						r_\ell\left(\Gamma_{Z_1}\right)=\frac{e^{-\gamma_\ell(\ell-1)} e^{-\gamma_\ell\ell} (1-r_p)^{-2}}{\left(e^{-\gamma_\ell(\ell-1)} - e^{-\gamma_\ell\ell}(1-r_p)^{-2}\right)^2} \leq \delta_\ell.
					\end{equation*}
					Multiplying the numerator and denominator by $e^{2\gamma_\ell\ell}$ and defining $C_{r_p} = (1-r_p)^{-2}$, this simplifies to $\frac{C_{r_p} e^{\gamma_{\ell}}}{(e^{\gamma_{\ell}} - C_{r_p})^2} \leq \delta_\ell$. Rearranging for $e^{\gamma_\ell} > C_{r_p}$ yields the quadratic inequality $$e^{2\gamma_\ell} - C_{r_p}(2 + \delta_\ell^{-1})e^{\gamma_\ell} + C_{r_p}^2 \geq 0.$$ 
					Now since the larger root of the polynomial $x\longmapsto x^2-C_{r_p}(2 + \delta_\ell^{-1})x+C_{r_p}^2$ is $x_+=C_{r_p}\Big(1+\frac{1+\sqrt{1+4\delta_\ell}}{2\delta_\ell}\Big)$
					and since $C_{r_p} = (1-r_p)^{-2}$, we obtain that $\gamma_\ell$ must satisfy:
					\begin{equation}\label{ineq:gamma_ell}
						\gamma_\ell \geq \log\left( \frac{1}{(1-r_p)^2} \right) + \log\left( 1 + \frac{1 + \sqrt{1 + 4\delta_\ell}}{2\delta_\ell} \right).
					\end{equation}
					Note that since $r_p \in (0, 1/2]$,  choosing 
					\begin{equation}\label{eq:gamma_ell_geq_log4}
						\gamma_\ell =\log(4) + \log\left( 1 + \frac{1 + \sqrt{1 + 4\delta_\ell}}{2\delta_\ell} \right)
					\end{equation} 
					ensures \eqref{ineq:gamma_ell}.
					To verify that both distributions of $Z_0$ and $Z_1$ belongs to the regularity class $\mathcal P(\alpha,L,\delta_\ell)$ we then ensure that the corresponding kernels are $(\alpha,L)-$Hölder.\\ 
					We proceed in the same way as in the proof of Theorem \ref{thm:inconsistency}. By definition of $\psi_{j,0}^*$ and Lemma~\ref{lemm:trigo}, 
					\begin{eqnarray*}
						\left| \frac{\partial^m K_0}{\partial u^m}(u,t) - \frac{\partial^m K_0}{\partial u^m}(u',t)\right| &=& \left|\sum_{j=1}^{\ell-1}\lambda_{j,0}^*\left(a_j^{(m)}(u)-a_j^{(m)}(u')\right)a_j(t)\right|\\
						&\leq& 2\sqrt{2}(2\pi)^{\alpha}\sum_{j=1}^{\ell-1}\lambda_{j,0}^*j^\alpha|u-u'|^{\beta}\leq \frac{L}2|u-u'|^{\beta}, 
					\end{eqnarray*}
					by choosing 
					$$c_{\alpha,L,\ell}:=\frac{L}{4\sqrt{2}(2\pi)^{\alpha}\sum_{j=1}^{\ell-1}\exp(-\gamma_\ell j)j^\alpha}.$$
					Similarly, we assess the H\"older regularity of the kernel $K_1$ by bounding the variations of its $m$-th derivative:
					\begin{align*}
						\left| \frac{\partial^m K_1}{\partial u^m}(u,t) - \frac{\partial^m K_1}{\partial u^m}(u',t)\right| 
						&\leq \sum_{j=1}^{\ell -1} \lambda_{j,0}^* \left| a_j(t) \right| \left| a_j^{(m)}(u) - a_j^{(m)}(u') \right| \\
						&\quad + \lambda_{\ell,1}^* \left| \psi_{\ell,1}^*(t) \right| \left| \frac{\partial^m \psi_{\ell,1}^*}{\partial u^m}(u) - \frac{\partial^m \psi_{\ell,1}^*}{\partial u^m}(u') \right|.
					\end{align*}
					For the first term, using again the global bound $|a_j(t)| \leq \sqrt{2}$, the H\"older bound $|a_j^{(m)}(u) - a_j^{(m)}(u')| \leq 2\sqrt{2}(2\pi j)^\alpha |u-u'|^\beta$, and the definition of $c_{\alpha,L,\ell}$, we obtain exact control over the partial sum:
					\begin{equation*}
						\sum_{j=1}^{\ell -1} c_{\alpha,L,\ell} e^{-\gamma_\ell j} 2(2\pi j)^\alpha |u-u'|^\beta 
						= \frac{L}{2}|u-u'|^\beta.
					\end{equation*}
					For the second term, we recall that $\lambda_{\ell,1}^* = (1-r_p)^{-2}\lambda_{\ell,0}^*$ and $|\psi_{\ell,1}^*(t)| \leq 1-r_p + \sqrt{2}\sqrt{1-(1-r_p)^2} \leq 2$. Furthermore, the choice $r_p := p^{-2\alpha}$ implies $\sqrt{1-(1-r_p)^2} \leq \sqrt{2}p^{-\alpha}$, which neutralizes the growth of the derivative of $a_{p}$:
					\begin{align*}
						\lambda_{\ell,1}^* \left| \psi_{\ell,1}^*(t) \right| \left| \frac{\partial^m \psi_{\ell,1}^*}{\partial u^m}(u) - \frac{\partial^m \psi_{\ell,1}^*}{\partial u^m}(u') \right|
						&\leq \frac{2\lambda_{\ell,0}^*}{(1-r_p)^2} \cdot \sqrt{1-(1-r_p)^2} \cdot \sqrt{2}(2\pi p)^\alpha |u-u'|^\beta \\
						&\leq \frac{2\lambda_{\ell,0}^*}{(1-r_p)^2} \left( \sqrt{2}p^{-\alpha} \right) \left( \sqrt{2}(2\pi)^\alpha p^\alpha \right) |u-u'|^\beta \\
						&= \frac{4(2\pi)^\alpha}{(1-r_p)^2} \lambda_{\ell,0}^* |u-u'|^\beta.
					\end{align*}
					
					Substituting the definition $\lambda_{\ell,0}^* = c_{\alpha,L,\ell}e^{-\gamma_\ell \ell}$ and the chosen value for $c_{\alpha,L,\ell}$ into the required bound $\lambda_{\ell,0}^* \leq \frac{(1-r_p)^2}{8(2\pi)^\alpha}L$, we obtain the equivalent condition:
					\begin{equation*}
						\frac{L e^{-\gamma_\ell \ell}}{4\sqrt{2}(2\pi)^\alpha \sum_{j=1}^{\ell-1}e^{-\gamma_\ell j}j^\alpha} \leq \frac{(1-r_p)^2 L}{8(2\pi)^\alpha}
						\quad \iff \quad
						\frac{\sqrt{2} e^{-\gamma_\ell \ell}}{\sum_{j=1}^{\ell-1}e^{-\gamma_\ell j}j^\alpha} \leq (1-r_p)^2.
					\end{equation*}
					Since $\gamma_\ell \geq \log(4)$ from Equation \eqref{eq:gamma_ell_geq_log4}, we have $e^{-\gamma_\ell}\leq 1/4$. Bounding the denominator by its term at $j=\ell-1$ (noting that $(\ell-1)^\alpha \geq 1$), the left-hand side is at most $\sqrt{2}e^{-\gamma_\ell} \leq 1/2$. Thus, the inequality holds as soon as $(1-r_p)^2 \geq 1/2$, or equivalently $r_p \leq 1 - 1/\sqrt{2}$. With $r_p = (p-1)^{-2\alpha}$, this is ensured whenever $p^{2\alpha} \geq 2+\sqrt{2}$, which is satisfied as we assumed $p\geq \left(2+\sqrt 2\right)^{\frac{1}{2\alpha}}$.
					
					In the end we obtain 
					\begin{equation*}
						\left| \frac{\partial^m K_1}{\partial u^m}(u,t) - \frac{\partial^m K_1}{\partial u^m}(u',t)\right|\leq L|u-u'|^{\beta}.
					\end{equation*}
					Therefore we have
					\begin{equation}\label{eq:borne_inf_proof_p}
						\inf _{\widehat{\psi}_{\ell}} \sup _{P_{X} \in \mathcal{P}(\alpha,L,\delta_\ell)} \mathbb{E}\left[\left\|\widehat{\psi}_{\ell}-\psi_{\ell}(\Gamma_{P_X})\right\|^{2}\right] \geq \inf _{\hat{\psi}_{\ell}} \sup _{k=0,1} \mathbb{E}\left[\left\|\widehat{\psi}_{\ell}-\psi_{\ell, k}^{*}\right\|^2\right].
					\end{equation}
					Lets first focus on the control of $\inf _{\hat{\psi}_{\ell}} \sup _{k=0,1} \mathbb{E}\left[\left\|\widehat{\psi}_{\ell}-\psi_{\ell, k}^{*}\right\|^{2}\right]$.
					
					Let $\widehat{\psi}_{\ell}$ an estimator of the $\ell$-th eigenfunction and $\widehat{\phi}$ the minimum distance test defined by
					$$
					\widehat{\phi}:=\arg \min _{k=0,1}\left\|\widehat{\psi}_{\ell}-\psi_{\ell, k}^{*}\right\|^{2},
					$$
					we have for $j=0,1$,
					$$
					\left\|\widehat{\psi}_{\ell}-\psi_{\ell, k}^{*}\right\| \geq \frac{1}{2}\left\|\psi_{\ell, \widehat{\phi}}^{*}-\psi_{\ell, k}^{*}\right\|,
					$$
					Using Lemma \ref{lemm:trigo}, we get that
					$$
					\begin{aligned}
						\left\|\psi_{\ell, \widehat{\phi}}^{*}-\psi_{\ell, k}^{*}\right\|^{2} &=\mathbf{1}_{\{\widehat{\phi} \neq k\}}\left\|\psi_{\ell,0}^{*}-\psi_{\ell,1}^{*}\right\|^{2}\\
						&=\mathbf{1}_{\{\widehat{\phi} \neq k\}}\left\|r_p-\sqrt{1-(1-r_p)^2}a_{p}(\cdot)\right\|^{2}\\
						&=\mathbf{1}_{\{\widehat{\phi} \neq k\}}2r_p= 2 p^{-2\alpha}\mathbf{1}_{\{\widehat{\phi} \neq k\}}.
					\end{aligned}
					$$
					
					Then,
					\begin{equation}\label{eq:lb_eigfunc_proof2}
						\inf _{\widehat{\psi}_{\ell}} \sup _{P_{X} \in \mathcal{P}(\alpha,L,\delta_\ell)} \mathbb{E}\left[\left\|\widehat{\psi}_{\ell}-\psi_{\ell}^{*}\right\|^{2}\right] \geq 
						\frac{1}{2} p^{-2\alpha} \inf _{\widehat{\phi}} \max _{k=0,1} \mathbb{P}(\widehat{\phi} \neq k),  
					\end{equation}
					where the minimum in the right-hand side is taken over all test functions i.e. all measurable function $\widehat\phi:\mathbb R^{n\times p}\to \{0,1\}$ of the observations.  
					
					Therefore we now try to bound from below the quantity $\mathbb{P}(\widehat{\phi} \neq k)$. For this purpose we will use Lemma \ref{lemm1}. Our purpose is then to show that there exists $H_{max}^2<2$ such that
					\begin{equation*}
						H^{2}\left(\left(P_{0}^{o b s}\right)^{\otimes n},\left(P_{1}^{o b s}\right)^{\otimes n}\right) \leq H_{\max }^{2}
					\end{equation*}
					where $P_{k}^{o b s}$ is the law of the random vector $\mathbf{Y}_k^{ o b s}:=\left(Y_{k}\left(t_{0}\right), \ldots, Y_{k}\left(t_{p-1}\right)\right)$ such that
					$$
					Y_{k}\left(t_{\ell}\right)=Z_{k}\left(t_{\ell}\right)+\varepsilon_{k,\ell}
					$$
					with $\varepsilon_{k,\ell} \overset{i.i.d.}{\sim} \mathcal{N}\left(0, \sigma^{2}\right)$ for $\ell\in[\![0,p-1]\!]$ and $\sigma^2>0$ fixed.\\
					Then, in this case, we have
					\begin{equation*}
						\inf _{\widehat{\psi}} \max _{k=0,1} \mathbb{P}(\widehat{\psi} \neq k) \geq \frac{1}{2}\left(1-\sqrt{H_{\max }^{2}\left(1-H_{\max }^{2} / 4\right)}\right)>0.
					\end{equation*}
					In our case, we remark that:
					$$\mathbf{Y}_k^{ o b s}\sim \mathcal{N}(0,G_k),$$
					with
					\begin{equation*}
						\left[G_k\right]_{\ell',k'}=\mathbb{E}\left[Y_{k}\left(t_{\ell'}\right) Y_{k}\left(t_{k'}\right)\right]=\sum_{j=1}^\ell \lambda_{j,k}^*\psi_{j,k}^*(t_{k'})\psi_{j,k}^*(t_{\ell'})+\mathbf{1}_{\ell'=k'}\sigma^2.
					\end{equation*}
					Our construction ensures that $a_{p}(t_j)=\sqrt 2\sin(2\pi (j-1))=0$ as $t_j:=\frac{j-1}{p}$ for all $1\leq j \leq p$. This implies that $\lambda_{\ell,0}^*\psi_{\ell,0}^*(t_{k'})\psi_{\ell,0}^*(t_{\ell'})=\lambda_{\ell,1}^*\psi_{\ell,1}^*(t_{k'})\psi_{\ell,1}^*(t_{\ell'})$ and thus $G_0=G_1$.\\
					The distributions of the discrete observations are then the same for the two processes we have chosen, so
					$$
					H^{2}\left(\left(P_{0}^{o b s}\right)^{\otimes n},\left(P_{1}^{o b s}\right)^{\otimes n}\right)=0.
					$$
					We can therefore apply Lemma \ref{lemm1} in combination with \eqref{eq:lb_eigfunc_proof2} and obtain the expected result. 
				\end{proof}
				\subsection{Proof of Theorem \ref{Thm_eigval}}\label{subsec:proof_eigval_lb}
				\begin{prop}{(Minimax lower-bound for the $\ell $-th eigenvalue estimation in regards of the sample size)}{\label{prop:eigval_lb_n}\\
						Let $n\in\mathbb{N}\setminus\{0\},p\in\mathbb{N}$ such that $p\geq 4,\ell\in\mathbb{N}\setminus\{0\}$ such that $\ell+1<\frac{p}{2}$ and $\alpha,L\in\mathbb{R}^*_+$.
						Then, we have the existence of $c>0$ that depends only on $\sigma^2$ such that
						$$\inf _{\widehat{\lambda}_{\ell}} \sup _{P_{X} \in \mathcal{P}(\alpha, L)} \mathbb{E}\left[\left(\frac{\widehat{\lambda}_{\ell}-\lambda_{\ell}(\Gamma_{P_X})}{\lambda_{\ell}(\Gamma_{P_X})}\right)^{2}\right] \geq  \frac{c}{n},$$
						where the infimum is taken over all measurable functions of the data.
					}
				\end{prop}
				
				\begin{proof}{
						Let $n\in\mathbb{N}\setminus\{0\},p\in\mathbb{N}$ such that $p\geq 4,\ell\in\mathbb{N}$ such that $ 2\leq\ell+1<\frac{p}{2}$ (the case $\ell=1$ can be treated in a similar way). Let $L\in\mathbb{R}^*_+$, $\beta\in(0,1]$, $m\in\mathbb{N}$ and $\alpha=m+\beta$.
						We introduce for $j\geq $1,
						\[
						\begin{aligned}
							a_j : [0,1] &\longrightarrow [-\sqrt{2},\sqrt{2}] \\
							x &\longmapsto \sqrt{2}\sin(2\pi j x).
						\end{aligned}
						\]
						For $k\in\{0,1\}$ and $t\in[0,1]$, let the processes be defined by the following Karhunen-Loève expansion with eigenfunctions $\psi_j^*:=a_j$
						\begin{equation*}
							Z_k(t):=\sum_{j=1}^{\ell -1}\sqrt{\lambda_{j}^*}\xi_{j,k}\psi_{j}^*(t) + \sqrt{\lambda_{\ell,k}^*}\xi_{\ell,k}\psi_{\ell}(t)=\sum_{j=1}^{\ell -1}\sqrt{\lambda_{j}^*}\xi_{j,k} a_{j}(t) + \sqrt{\lambda_{\ell,k}^*}\xi_{\ell,k}a_{\ell}(t),\\
						\end{equation*}
						with $\sigma^2\in \mathbb{R}_+^*, \{\xi_{j,k}\}_{j=1,\hdots,\ell}\overset{i.i.d.}{\sim}\mathcal{N}(0,\sigma^2)$ and
						\begin{gather*}
							\lambda_{1,0}^*>\dots>\lambda_{\ell-1,0}^*>\lambda_{\ell,0}^*>0\\
							\lambda_{\ell,1}^*:=\lambda_{\ell,0}^*+b_n\text{ and for all } j\in[\![1,\ell -1]\!]\quad \lambda^*_j:=\lambda_{j,0}^*,
						\end{gather*}
						with $b_n:=\frac{\lambda_{\ell,0}^*}{\sqrt{n}}$.
						
						As for all $(t,u)\in[0,1]^2$
						$$K_k(u,t):=\mathbb{E}\left[Z_k(t)Z_k(u)\right]=\sum_{j=1}^{\ell -1}\lambda_{j}^* \psi_{j}^*(t)\psi_{j}^*(u)+\lambda_{\ell,k}^* \psi_{\ell}(t)\psi_{\ell}(u)$$ 
						we get that for any given $(t,u)\in]0,1[^2$,
						$$\frac{\partial^m K_k}{(\partial u)^m}(u,t)=\sum_{j=1}^{\ell -1}\lambda_{j}^* \psi_{j}^*(t)\frac{\partial^m \psi^*_{j}}{(\partial u)^m}(u)+\lambda_{\ell,k}^* \psi_{\ell}(t)\frac{\partial^m \psi^*_{\ell}}{(\partial u)^m}(u).$$
						Hence we have the following inequalities (with $(u,u',t)\in ]0,1[^3$):
						\begin{align*}
							\left| \frac{\partial^m K_k}{(\partial u)^m}(u,t) - \frac{\partial^m K_k}{(\partial u)^m}(u',t)\right|
							&\leq \sum_{j=1}^{\ell -1}\lambda_{j}^*\left | \frac{\partial^m \psi^*_{j}}{(\partial u)^m}(u)-\frac{\partial^m \psi^*_{j}}{(\partial u)^m}(u') \right |+\lambda_{\ell,k}^*\left | \frac{\partial^m \psi^*_{\ell}}{(\partial u)^m}(u)-\frac{\partial^m \psi^*_{\ell}}{(\partial u)^m}(u') \right |\\
							&\leq 2\sqrt{2}\left (\sum_{j=1}^{\ell -1}(2\pi j)^{\alpha}\lambda_{j}^* +(2\pi \ell)^{\alpha}\left ( \lambda_{\ell,0}^*+b_n \right )  \right )\left | u-u' \right |^{\beta},
						\end{align*}
						where the second inequality is obtained by Lemma \ref{lemm:trigo}.
						
						Thus, by taking $\lambda_1^*,\dots,\lambda_{\ell-1}^*$ and $\lambda_{\ell,0}^*$ sufficiently small such that
						$$2\sqrt{2}\left (\sum_{j=1}^{\ell -1}(2\pi j)^{\alpha}\lambda_{j}^* +(2\pi \ell)^{\alpha}\lambda_{\ell,0}^*  \right )\leq \frac{L}{2}$$
						It is sufficient to guarantee that
						\begin{equation}\label{eq:valp_
								cond_bn}
							b_n\leq \frac{L}{4\sqrt{2}(2\pi\ell)^\alpha},
						\end{equation}
						which can be achieved by choosing $\lambda_{\ell,0}^*$ small enough. In this way, we have that $P_{Z_k} \in \mathcal{P}(\alpha,L)$ for all $k\in\{0,1\},$  hence
						$$\inf _{\widehat{\lambda}_{\ell}} \sup _{P_X \in \mathcal{P}(\alpha,L)} \mathbb{E}\left[\left(\frac{\widehat{\lambda}_{\ell}-\lambda_{\ell}(\Gamma_{P_X})}{\lambda_{\ell}(\Gamma_{P_X})}\right)^{2}\right] \geq \inf _{\hat{\lambda}_{\ell}} \sup _{k=0,1} \mathbb{E}\left[\left(\frac{\widehat{\lambda}_{\ell}-\lambda_{\ell,k}^{*}}{\lambda_{\ell,k}^{*}}\right)^{2}\right].$$
						
						Let $\widehat{\lambda}_{\ell}$ be an estimator of $\lambda_{\ell,k}^*$ and 
						$\widehat{\phi}$ defined by
						$$
						\widehat{\phi}:=\arg \min _{k=0,1}\left(\widehat{\lambda}_{\ell}-\lambda_{\ell, k}^{*}\right)^{2}.
						$$
						Then, for $k=0,1$,
						\[
						\left|\widehat{\lambda}_{\ell}-\lambda_{\ell, k}^{*}\right|
						\geq \frac{1}{2}\left|\lambda_{\ell, \widehat{\phi}}^{*}-\lambda_{\ell, k}^{*}\right|.
						\]
						Hence,
						\begin{align*}
							\left(\frac{\widehat{\lambda}_{\ell}-\lambda_{\ell, k}^{*}}{\lambda_{\ell,k}^{*}}\right)^{2}
							\geq 
							\frac{1}{4}\left(\frac{\lambda_{\ell, \widehat{\phi}}^{*}-\lambda_{\ell, k}^{*}}{\lambda_{\ell,k}^{*}}\right)^{2} 
							=
							\frac{1}{4}\mathbf{1}_{\{\widehat{\phi} \neq k\}}
							\left(\frac{b_n}{\lambda_{\ell,k}^{*}} \right)^2 
							\ge 
							\frac{1}{4}\mathbf{1}_{\{\widehat{\phi} \neq k\}}
							\left(\frac{b_n}{\lambda_{\ell,1}^{*}} \right)^2,
						\end{align*}
						since $\lambda_{\ell,1}^{*}\ge \lambda_{\ell,k}^{*}$. Furthermore, we have:
						\begin{equation*}
							\begin{aligned}
								\mathbf{1}_{\{\widehat{\phi} \neq k\}}\left(\frac{b_n}{\lambda_{\ell,1}^{*}} \right)^2&=\mathbf{1}_{\{\widehat{\phi} \neq k\}}\left(\frac{b_n}{\left(\lambda_{\ell,0}^{*}+b_n\right)} \right)^2\\
								&        =\mathbf{1}_{\{\widehat{\phi} \neq k\}}\left(\frac{1}{\left(\frac{\lambda_{\ell,0}^{*}}{b_n}+1\right)} \right)^2
								=\mathbf{1}_{\{\widehat{\phi} \neq k\}}\left(\frac{1}{\left(\sqrt n+1\right)} \right)^2 \geq \mathbf{1}_{\{\widehat{\phi} \neq k\}}\frac{1}{4n}
							\end{aligned}
						\end{equation*}
						Therefore
						\begin{equation}\label{eq:lb_eigval_cn}
							\inf _{\widehat{\lambda}_{\ell}} \sup _{P_X \in \mathcal{P}(\alpha,L)} \mathbb{E}\left[\left(\frac{\widehat{\lambda}_{\ell}-\lambda_{\ell}^{*}}{\lambda_{\ell}^{*}}\right)^{2}\right] \geq \frac{1}{16n}  \inf _{\widehat{\phi}} \max _{k=0,1} \mathbb{P}(\widehat{\phi} \neq k).
						\end{equation}
						
						Then, let $P_{k}^{o b s}$ be the law of the random observation vector $\mathbf{Y}_k^{ o b s}:=\left(Y_{k}\left(t_{1}\right), \ldots, Y_{k}\left(t_{p}\right)\right)$ such that
						$$
						Y_{k}\left(t_{k'}\right)=Z_{k}\left(t_{k'}\right)+\varepsilon_{k'}
						$$
						with $k'\in[\![1,p]\!],t_{k'}:=\frac{k'-1}{p}$,$\varepsilon_{k'} \overset{i.i.d.}{\sim} \mathcal{N}\left(0, \sigma^{2}\right)$ and $\sigma^2>0$ fixed.\\
						Then, in this case, we have
						$$\mathbf{Y}_k^{ o b s}\sim \mathcal{N}(0,G_k),$$
						with a covariance matrix that can be computed with the help of Lemma \ref{lemm:trigo}
						\begin{equation*}\label{eq:def_matrice_obs_prop2}
							G_k =\sum_{j=1}^{\ell -1}\lambda_j^*\boldsymbol{a}_j\boldsymbol{a}_j^t+\lambda_{\ell,k}^*\boldsymbol{a}_\ell \boldsymbol{a}_\ell ^t+\sigma^2I_p
						\end{equation*}
						where $\boldsymbol{a}_j:=\left(a_j\left(t_{1}\right), \ldots, a_j\left(t_{p}\right)\right)^{t}$ for all $j\in[\![1,\ell ]\!]$ and $I_p$ is the identity matrix of dimension $p$.
						
						As $P_{0}^{o b s}$ and $P_{1}^{o b s}$ are Gaussian with equal mean vectors, we have an explicit formula for the Hellinger affinity by using Lemma \ref{lemm2}, we can then rely on \eqref{eq:def_affinite_gauss}.
						
						Once again, we compute the determinants of $G_0G_1$ and $\frac{G_0+G_1}{2}$.
						
						Let us set $v_j:=\frac{\boldsymbol{a}_j}{\sqrt p}$ for all $j\in[\![1,\ell ]\!]$ so that, by Lemma~\ref{lemm:trigo}, $v_1,\hdots,v_\ell$ be an orthonormal family of vectors of $\mathbb R^p$. Remark that, by Lemma~\ref{lemm:trigo} again, for all $j\in[\![1,\ell ]\!]$, 
						$$
						G_k v_j = \lambda_j^*pv_j+\sigma^2 v_j, 
						$$
						so that $v_j$ is an eigenvector of $G_k$ with associated eigenvalue 
						$$ 
						\kappa_{j,k}:=\lambda_{j,k}^*p+\sigma^2. 
						$$ 
						
						We complete the family $(v_1,\hdots,v_\ell)$ with $v_{\ell+1}, \ldots, v_{p}$ so that $(v_1,\hdots,v_p)$ forms an orthonormal basis of $\mathbb{R}^p$ and define the orthogonal matrix
						$$
						V:=\left[v_{1} , v_{2} , \cdots , v_{p}\right] .
						$$
						We set, for the sake of readability,
						$$ \kappa_{j,k}:=\lambda_{j,k}^*p\mathbf{1}_{j\in[\![1,\ell ]\!]}+\sigma^2$$
						so that 
						$$G_kv_j=\kappa_{j,k} v_j, \qquad j=1,\hdots,p.$$
						We define $G_0'$ and $G_1'$ as the matrices such that $G_0=VG_0'V^t$ and $G_1=VG_1'V^t$. 
						
						Remark that 
						\[
						\kappa_{j,0}=\kappa_{j,1}, \text{ for }j< \ell, \kappa_{\ell,1}=\kappa_{\ell,0}+p b_n \text{ and }\kappa_{j,0}=\kappa_{j,1}=\sigma^2\text{ for }j>\ell. 
						\]
						
						We then obtain
						$$\left [ G_0'G_1' \right ]_{i,j}=\kappa_{i,0}\kappa_{j,1}\mathbf{1}_{i=j}=\begin{cases}
							\kappa_{i,0}^2 & \text{ if } j=i<\ell  \\
							\kappa_{\ell,0}^2+pb_n\kappa_{\ell,0} & \text{ if } j=i=\ell  \\
							\sigma^4 & \text{ if } j=i>\ell  \\
							0 & \text{ else }
						\end{cases}$$
						and
						$$\left [ \frac{G_0'+G_1'}{2} \right ]_{i,j}=\begin{cases}
							\kappa_{i,0}  & \text{ if } j=i<\ell  \\
							\kappa_{\ell,0}+p\frac{b_n}{2}& \text{ if } j=i=\ell  \\
							\sigma^2 & \text{ if } j=i>\ell  \\
							0 & \text{ else }
						\end{cases}.$$
						Then, we compute the determinants 
						$$\operatorname{det}\left ( G_0G_1 \right )=\operatorname{det}\left ( G_0'G_1' \right )=\sigma^{4(p-\ell)}\left ( \kappa_{\ell,0}^2+pb_n\kappa_{\ell,0} \right )\prod_{i=1}^{\ell -1}\kappa_{i,0}^2$$
						and
						$$\operatorname{det}\left ( \frac{G_0+G_1}{2} \right )=\operatorname{det}\left ( \frac{G_0'+G_1'}{2} \right )=\sigma^{2(p-\ell)}\left ( \kappa_{\ell,0}+p\frac{b_n}{2}\right )\prod_{i=1}^{\ell -1}\kappa_{i,0}.$$
						As a result, to ensure that the affinity is bounded from below away from zero it suffices to have $1>C>0$ such that
						\begin{equation*}
							\begin{aligned}
								A\left(P_{0}^{o b s}, P_{1}^{o b s}\right)^n=&\left (\frac{\left (\kappa_{\ell,0}^2+pb_n\kappa_{\ell,0}\right )^{\frac{1}{4}}}{\left ( \kappa_{\ell,0}+p\frac{b_n}{2} \right )^{\frac{1}{2}}}  \right )^n\geq C\\
								\Leftrightarrow\quad\quad &\left (C^{\frac{4}{n}}-1  \right )\kappa_{\ell,0}^2+\left (C^{\frac{4}{n}}-1  \right )\kappa_{\ell,0}pb_n+\frac{C^{\frac{4}{n}}}{4}p^2b_n^2 \leq 0.\\
							\end{aligned}
						\end{equation*}
						Since $0<C<1$, we have $C^{\frac{4}{n}}-1<0$. In particular, for any $b_n\ge 0$,
						\[
						\left (C^{\frac{4}{n}}-1  \right )\kappa_{\ell,0}p b_n \le 0,
						\]
						and therefore
						\[
						\left (C^{\frac{4}{n}}-1  \right )\kappa_{\ell,0}^2
						+\left (C^{\frac{4}{n}}-1  \right )\kappa_{\ell,0}pb_n
						+\frac{C^{\frac{4}{n}}}{4}p^2b_n^2 \le 
						\left (C^{\frac{4}{n}}-1  \right )\kappa_{\ell,0}^2
						+\frac{C^{\frac{4}{n}}}{4}p^2b_n^2.
						\]
						Hence, the desired inequality holds as soon as
						\[
						b_n^2 \le 
						\frac{4\left(1-C^{\frac{4}{n}}\right)}{C^{\frac{4}{n}}}\,
						\frac{\kappa_{\ell,0}^2}{p^2}.
						\]
						By definition of $\kappa_{\ell,0}$, 
						\[
						b_n^2 \le 
						\frac{4\left(1-C^{\frac{4}{n}}\right)}{C^{\frac{4}{n}}}\,
						\frac{(\lambda_{\ell,0}^*p+\sigma^2)^2}{p^2}.
						\]
						With the choice $b_n=\lambda_{\ell,0}^*/\sqrt n$, this is equivalent to 
						$$1\leq 4nC^{-\frac{4}{n}}\left (1-C^{\frac{4}{n}}\right )\left(\frac{\lambda_{\ell,0}^*p+\sigma^2}{\lambda_{\ell,0}^*p}\right)^2;$$
						Since
						$$
						1<\left(\frac{\lambda_{\ell,0}^*p+\sigma^2}{\lambda_{\ell,0}^*p}\right)^2. 
						$$
						it is sufficient to find $C\in (0,1)$ such that 
						$$
						1\leq 4nC^{-\frac{4}{n}}\left (1-C^{\frac{4}{n}}\right )
						$$
						which is possible by taking $C\in (0,1)$ small enough since 
						$$\lim_{n\to+\infty}4nC^{-\frac{4}{n}}\left (1-C^{\frac{4}{n}}\right )=-16\log C.$$

						Thus, we have that
						for all $n\in\mathbb{N},A\left(P_{0}^{o b s}, P_{1}^{o b s}\right)^n\geq C >0$ which implies, by virtue of \eqref{eq:link_Hellinger_A}, the existence of $H_{max}^2<2$ such that \eqref{eq:cond_H} is verified. We can therefore apply Lemma \ref{lemm1} in combination with \eqref{eq:lb_eigval_cn} and the last inequality to obtain for a given $C_0>0$ that does not depend on the parameters of the problem
						\begin{equation*}
							\inf _{\widehat{\lambda}_{\ell}} \sup _{P_X \in \mathcal{P}(\alpha, L)} \mathbb{E}\left[\left(\frac{\widehat{\lambda}_{\ell}-\lambda_{\ell}(\Gamma_{P_X})}{\lambda_{\ell}(\Gamma_{P_X})}\right)^{2}\right] \geq \frac{C_0}{n}.
						\end{equation*}
					}
				\end{proof}
				
				\begin{prop}{(Minimax lower-bound for the $\ell $-th eigenvalue estimation in regards on the grid size)}{\label{prop:eigval_lb_p}\\
						Let $n\in\mathbb{N}\setminus\{0\},\alpha,L\in\mathbb{R}^*_+$. Let $p\in\mathbb{N}$ such that $p\geq \max\left(\left(2+\sqrt 2\right)^{\frac{1}{2\alpha}},4\right),\ell\in\mathbb{N}\setminus\{0\}$ such that $ \ell+1<\frac{p}{2}$ and $\delta_\ell\in\mathbb{R}^*_+$.
						Then we have the existence of $c>0$ that does not depend on $\ell,\sigma^2$ such that
						$$\inf _{\widehat{\lambda}_{\ell}} \sup _{P_{X} \in \mathcal{P}(\alpha, L)} \mathbb{E}\left[\left(\frac{\widehat{\lambda}_{\ell}-\lambda_{\ell}(\Gamma_{P_X})}{\lambda_{\ell}(\Gamma_{P_X})}\right)^{2}\right] \geq \frac{c}{p^{4\alpha}},$$
						where the infimum is taken over all measurable functions of the data.
					}
				\end{prop}
				
				\begin{proof}
					We follow the exact same construction as in the proof of Proposition \ref{prop:eigfun_lb_p}. But this time as we are interested in the eigenvalues some steps change.\\
					Namely, \eqref{eq:borne_inf_proof_p} becomes 
					$$\inf _{\hat{\lambda}_{\ell}} \sup _{P_X \in \mathcal{P}(\alpha,L)} \mathbb{E}\left[\left(\frac{\widehat{\lambda}_{\ell}-\lambda_{\ell}(\Gamma_{P_X})}{\lambda_{\ell}(\Gamma_{P_X})}\right)^{2}\right] \geq \inf _{\hat{\lambda}_{\ell}} \sup _{k=0,1} \mathbb{E}\left[\left(\frac{\widehat{\lambda}_{\ell}-\lambda_{\ell, k}^{*}}{\lambda_{\ell, k}^{*}}\right)^2\right].$$
					Then we note that 
					by letting $\widehat{\lambda}_{\ell}$ be an estimator and $\widehat{\phi}$ defined by
					$$
					\widehat{\phi}:=\arg \min _{k=0,1}\left(\widehat{\lambda}_{\ell}-\lambda_{\ell, k}^{*}\right)^{2},
					$$
					we have for $k=0,1$,
					$$
					\left|\widehat{\lambda}_{\ell}-\lambda_{\ell, k}^{*}\right| \geq \frac{1}{2}\left|\lambda_{\ell, \widehat{\phi}}^{*}-\lambda_{\ell, k}^{*}\right|.
					$$
					Now, we use:
					\[
					\left(\frac{\lambda_{\ell,\widehat\phi}^*-\lambda_{\ell,k}^*}{\lambda_{\ell,k}^*}\right)^2
					\;\ge\;
					\mathbb 1_{\widehat\phi\neq k}
					\left(\frac{\lambda_{\ell,1}^*-\lambda_{\ell,0}^*}{\lambda_{\ell,k}^*}\right)^2.
					\]
					
					Since $\lambda_{\ell,1}^*=\lambda_{\ell,0}^*(1-r_p)^{-2}$, we can normalize by $\lambda_{\ell,1}^*$ to obtain
					\begin{align*}
						\left(\frac{\lambda_{\ell,1}^*-\lambda_{\ell,0}^*}{\lambda_{\ell,k}^*}\right)^2
						&\ge
						\frac{(\lambda_{\ell,0}^*)^2\bigl(1-(1-r_p)^{-2}\bigr)^2}{(\lambda_{\ell,1}^*)^2}=
						(1-r_p)^4\bigl(1-(1-r_p)^{-2}\bigr)^2=
						r_p^2(2-r_p)^2 .
					\end{align*}
					
					Finally, using $r_p\le \frac 12$, we get $2-r_p\ge \frac32$, so
					\[
					r_p^2(2-r_p)^2 \;\ge\; \frac{9}{4}\,r_p^2.
					\]
					
					Combining the displays and recalling $r_p=p^{-2\alpha}$, we obtain
					\[
					\left(\frac{\lambda_{\ell,\widehat\phi}^*-\lambda_{\ell,k}^*}{\lambda_{\ell,k}^*}\right)^2
					\;\ge\;
					\frac{9}{4}\,\mathbb 1_{\widehat\phi\neq k}\,r_p^2
					=
					\frac{9}{4}\,\mathbb 1_{\widehat\phi\neq k}\,p^{-4\alpha},
					\]
					and conclude in the same way that in Proposition \ref{prop:eigfun_lb_p} by observing that the samples generated by the two processes coincide in law.
				\end{proof}
				\subsection{Proof of Theorem~\ref{thm:UB}}\label{subsec:Proof_thm_UB}
				The assumption on  the eigenvalues implies that the associated eigenspaces are all of dimension~1, which allows us to use Lemma \ref{lemm4}. With $\|\cdot\|_{\mathcal{L}}$ the operator norm, we then have:
				\begin{small}
					\begin{equation}\label{eq:bornesup1}
						\begin{split}
							\mathbb{E}\left[\left\|\widehat{\psi}_{\ell,\phi}-\psi_{\ell, \pm}\right\|^2\right] &\leq   \frac{8}{\operatorname{min}_{k\neq \ell} \left|\lambda_{k}-\lambda_{\ell} \right|^2}\mathbb{E}\left[\left\|\widehat{\Gamma_{\widetilde{\Pi}_p}}-\Gamma\right\|^2_{\mathcal{L}}\right]\\
							&\leq \frac{8}{\operatorname{min}_{k\neq \ell} \left|\lambda_{k}-\lambda_{\ell}\right|^2}\mathbb{E}\left[ \left \|\widehat{\Gamma_{\widetilde{\Pi}_p}}-\Gamma_{\widetilde{\Pi}_p}+\Gamma_{\widetilde{\Pi}_p}-\Gamma_{\Pi_p}+\Gamma_{\Pi_p}-\Gamma  \right \|^2_{\mathcal{L}}\right ]\\
							&\leq \frac{24}{\operatorname{min}_{k\neq \ell} \left|\lambda_{k}-\lambda_{\ell}\right|^2}\left( \mathbb{E}\left[\left \|\widehat{\Gamma_{\widetilde{\Pi}_p}}-\Gamma_{\widetilde{\Pi}_p}  \right \|^2_{\mathcal{L}}\right ]+\left \|\Gamma_{\widetilde{\Pi}_p}-\Gamma_{\Pi_p}  \right \|^2_{\mathcal{L}}+\left \|\Gamma_{\Pi_p}-\Gamma  \right \|^2_{\mathcal{L}}\right),
						\end{split}
					\end{equation}
				\end{small}
				with $\Gamma,\Gamma_{\Pi_p},\Gamma_{\widetilde{\Pi}_p},\widehat{\Gamma_{\widetilde{\Pi}_p}}$ the respective integral operators (defined in the same way as in \eqref{eq:def_Gamma}) of $K,K_{\Pi_p},K_{\widetilde{\Pi}_p},\widehat{K_{\widetilde{\Pi}_p}}$ defined in \eqref{eq:def_K} and in Section~\ref{def_estim}).
				We take advantage of the fact that these are Hilbert-Schmidt operators with integral kernels to bound these operator norms by $\mathbb{L}_2([0,1]^2)$ norms
				\begin{equation}\label{eq:bornesup2}
					\begin{aligned}
						\mathbb{E}\left[\left \|\widehat{\Gamma_{\widetilde{\Pi}_p}}-\Gamma_{\widetilde{\Pi}_p}  \right \|^2_{\mathcal{L}}\right ]&\leq\mathbb{E}\left[\left \|\widehat{\Gamma_{\widetilde{\Pi}_p}}-\Gamma_{\widetilde{\Pi}_p}  \right \|^2_{\operatorname{HS}}\right ]=\mathbb{E}\left[\left \|\widehat{K_{\widetilde{\Pi}_p}}-K_{\widetilde{\Pi}_p}  \right \|^2_{\mathbb{L}_2([0,1]^2)}\right ],\\
						\left \|\Gamma_{\widetilde{\Pi}_p}-\Gamma_{\Pi_p}  \right \|^2_{\mathcal{L}}&\leq \left \|\Gamma_{\widetilde{\Pi}_p}-\Gamma_{\Pi_p}  \right \|^2_{\operatorname{HS}}=\left \|K_{\widetilde{\Pi}_p}-K_{\Pi_p}  \right \|^2_{\mathbb{L}_2([0,1]^2)},\\
						\left \|\Gamma_{\Pi_p}-\Gamma  \right \|^2_{\mathcal{L}}&\leq \left \|\Gamma_{\Pi_p}-\Gamma  \right \|^2_{\operatorname{HS}}= \left \|K_{\Pi_p}-K  \right \|^2_{\mathbb{L}_2([0,1]^2)}.
					\end{aligned}
				\end{equation}
				In \eqref{eq:bornesup2}, $\|\cdot\|_{\operatorname{HS}}$ denotes the Hilbert-Schmidt norm.
				The first term is relative to a stochastic error, the second one is a quadrature error term and the last one is an approximation error.
				
				Assuming that $p$ is a dyadic integer, we denote $j=\log_2(p)\in {\mathbb N}$. Denoting $V_j^2$ the space spanned by the functions $(s,t)\longmapsto \phi_{j,k}(s)\phi_{j,k'}(t)$, and $K_j$ the projection of $K$ on $V_j^2$, we obtain:
				\begin{align*}
					K_j(s,t)&=\sum_{k=0}^{p-1}\sum_{k'=0}^{p-1}\langle K,\phi_{j,k}\otimes\phi_{j,k'}\rangle\phi_{j,k}(s)\phi_{j,k'}(t)\\
					&=\sum_{k=0}^{p-1}\sum_{k'=0}^{p-1}\Big(\iint K(u,v)\phi_{j,k}(u)\phi_{j,k'}(v)dudv\Big)\phi_{j,k}(s)\phi_{j,k'}(t)\\
					&=\sum_{k=0}^{p-1}\sum_{k'=0}^{p-1}\Big(\iint \E[X(u)X(v)]\phi_{j,k}(u)\phi_{j,k'}(v)dudv\Big)\phi_{j,k}(s)\phi_{j,k'}(t)\\
					&=\sum_{k=0}^{p-1}\sum_{k'=0}^{p-1}\E\Big[\int X(u)\phi_{j,k}(u)du\int X(v)\phi_{j,k'}(v)dv\Big]\phi_{j,k}(s)\phi_{j,k'}(t)\\
					&=\sum_{k=0}^{p-1}\sum_{k'=0}^{p-1}\E\Big[\langle X,\phi_{j,k}\rangle\langle X,\phi_{j,k'}\rangle\Big]\phi_{j,k}(s)\phi_{j,k'}(t)\\
					&=\E\Big[\Pi_j(X)(s)\Pi_j(X)(t)\Big],
				\end{align*}
				with 
				$$\Pi_j(X)(\cdot):=\sum_{k=0}^{p-1}\langle X,\phi_{j,k}\rangle\phi_{j,k}(\cdot)$$
				denoting the projection of $X$ on $V_j$ the space spanned by the functions $(\phi_{j,k})_k$. It means that $K_j=K_{\Pi_p}$ and using Lemma~4 of \cite{MR3620733}, we have for $T$ a constant depending on $\alpha$ and $\phi$,
				\begin{equation}\label{eq:err_approx}
					\begin{aligned}
						\left \|K_{\Pi_p}-K  \right \|^2_{\mathbb{L}_2([0,1]^2)}&=\left \|K_j-K  \right \|^2_{\mathbb{L}_2([0,1]^2)}\\
						&\leq TL^2 2^{-2j\alpha}=TL^2 p^{-2\alpha}.
					\end{aligned}
				\end{equation}
				
				Also as
				$$ \widetilde{\Pi}_p(X):=\sum_{k=0}^{p-1}p^{-\frac{1}{2}}X\left(\frac{k}{p}\right)\phi_{j,k},$$
				we have
				\begin{align*}
					K_{\widetilde{\Pi}_p}(s,t) &=\mathbb{E}\left [ \widetilde{\Pi_p}(X)(s) \widetilde{\Pi_p}(X)(t) \right ]\\
					&=\frac{1}{p}\sum_{k=0}^{p-1}\sum_{k'=0}^{p-1}\mathbb{E}\left[X\left(\frac{k}{p}\right)X\left(\frac{k'}{p}\right)\right]\phi_{j,k}(s)\phi_{j,k'}(t)\\
					&=\frac{1}{p}\sum_{k=0}^{p-1}\sum_{k'=0}^{p-1}K\left(\frac{k}{p},\frac{k'}{p}\right)\phi_{j,k}(s)\phi_{j,k'}(t).
				\end{align*}
				We have:
				\begin{align*}
					\big|K_{\widetilde{\Pi}_p}(s,t)-K_{\Pi_p}(s,t)\big|
					&= \Bigg|\frac{1}{p}\sum_{k=0}^{p-1}\sum_{k'=0}^{p-1}
					K\left(\frac{k}{p},\frac{k'}{p}\right)\phi_{j,k}(s)\phi_{j,k'}(t)
					-\sum_{k=0}^{p-1}\sum_{k'=0}^{p-1}\langle K,\phi_{j,k}\otimes\phi_{j,k'}\rangle
					\phi_{j,k}(s)\phi_{j,k'}(t) \Bigg| \\
					&\leq\frac{1}{p}\sum_{k=0}^{p-1}\sum_{k'=0}^{p-1}
					\Bigg|K\left(\frac{k}{p},\frac{k'}{p}\right)-p
					\langle K,\phi_{j,k}\otimes\phi_{j,k'}\rangle\Bigg|
					\,|\phi_{j,k}(s)|\,|\phi_{j,k'}(t)|.
				\end{align*}
				
				By Lemma \ref{lem:approx_scal_prod} (with $\alpha=m+\beta$),
				\[
				\Big|K\left(\frac{k}{p},\frac{k'}{p}\right)-p
				\langle K,\phi_{j,k}\otimes\phi_{j,k'}\rangle\Big|
				\le CL p^{-\alpha}.
				\]
				Moreover, since $K$ is continuous on $[0,1]^2$, it is bounded; hence
				\[
				\big|\langle K,\phi_{j,k}\otimes\phi_{j,k'}\rangle\big|
				\le \|K\|_\infty \|\phi_{j,k}\|_1\|\phi_{j,k'}\|_1
				\le C_0\,p^{-1},
				\]
				for a constant $C_0$ depending only on $K$ and $\phi$.
				Therefore, there exists a constant $C'>0$ such that
				\begin{align*}
					\big|K_{\widetilde{\Pi}_p}(s,t)-K_{\Pi_p}(s,t)\big|
					&\leq C' L p^{-\alpha-1}
					\sum_{k=0}^{p-1}\sum_{k'=0}^{p-1}|\phi_{j,k}(s)|\,|\phi_{j,k'}(t)|.
				\end{align*}
				Using the scaling property $\phi_{j,k}(x) = \sqrt{p}\phi(2^j x - k)$, we have:
				\begin{equation*}
					\big|K_{\widetilde{\Pi}_p}(s,t)-K_{\Pi_p}(s,t)\big| \leq CLp^{-\alpha}\sum_{k=0}^{p-1}\sum_{k'=0}^{p-1}|\phi(2^js-k)||\phi(2^jt-k')|.
				\end{equation*}
				Now, since $\phi$ is compactly supported (assume that its support is included into $[-A,A]$, with $0<A<\infty$), we have that for any $j$ and any $s$
				that
				\begin{align*}
					\sum_{k=0}^{p-1}|\phi(2^js-k)|&=\sum_{k\in\mathbb{Z}: \ |2^js-k|\leq A}|\phi(2^js-k)|\\
					&\leq\sum_{k\in\mathbb{Z}: \ |2^js-k|\leq A}\|\phi\|_{\infty}
					\leq (2A+1)\|\phi\|_{\infty},
				\end{align*}
				where we have used that the number of non-vanishing terms in the sum is smaller than $2A+1$.
				Therefore, there exists a constant $\tilde C$ only depending on $\phi$ and $\alpha$ such that
				\begin{equation}\label{eq:err_quadrature}
					\left \|K_{\widetilde{\Pi}_p}-K_{\Pi_p}  \right \|^2_{\mathbb{L}_{2}([0,1]^2)}\leq \tilde C^2 L^2p^{-2\alpha}.
				\end{equation}
				We now control the stochastic error term. 
				First, note that for any $(s,t)\in[0,1]$, since the $(Y_i)_{i=1}^n$  are i.i.d.  
				\begin{equation*}
					\begin{aligned}
						\mathbb{E}\!\left[ \frac{1}{n}\sum_{i=1}^n 
						\widetilde{\Pi_p}(Y_i)(s)\,\widetilde{\Pi_p}(Y_i)(t) \right]
						&= \mathbb{E}\!\left[ \widetilde{\Pi_p}(Y_1)(s)\,\widetilde{\Pi_p}(Y_1)(t) \right]
						\\
						&= \mathbb{E}\!\left[ \widetilde{\Pi_p}(X_1+\varepsilon_{1,\cdot})(s)\,
						\widetilde{\Pi_p}(X_1+\varepsilon_{1,\cdot})(t) \right]\\
						&= K_{\widetilde{\Pi}_p}(s,t)+ \mathbb{E}\!\left[
						\Big( p^{-\frac12}\!\!\sum_{k=0}^{p-1}\varepsilon_{1,k}
						\phi_{\log_2(p),k}(s) \Big)
						\Big( p^{-\frac12}\!\!\sum_{k=0}^{p-1}\varepsilon_{1,k}
						\phi_{\log_2(p),k}(t) \Big)
						\right]\\
						&= K_{\widetilde{\Pi}_p}(s,t)+\sigma^2 p^{-1}\sum_{k=0}^{p-1}\phi_{\log_2(p),k}(s)\phi_{\log_2(p),k}(t).
					\end{aligned}
				\end{equation*}
				And as 
				\begin{equation*}
					\widehat{K_{\widetilde{\Pi}_p}}(s, t) :=\frac{1}{n} \sum_{i=1}^{n}\widetilde{\Pi_p}(Y_i)(s)\cdot \widetilde{\Pi_p}(Y_i)(t)-\sigma^2p^{-1}\left (\sum_{k=0}^{p-1}\phi_{\operatorname{log}_2(p),k}(s)\phi_{\operatorname{log}_2(p),k}(t) \right ),
				\end{equation*}
				we have
				$$\mathbb{E}\!\left[\widehat{K_{\widetilde{\Pi}_p}}\right]=K_{\widetilde{\Pi}_p}.$$
				%
				%
				Since
				$(\phi_{j,k}\otimes\phi_{j,k'})_{j\in\mathbb{N}\setminus\{0\},(k,k')\in[\![0,2^j-1]\!]^2}$
				is an orthonormal family in $\mathbb{L}_2([0,1]^2)$, Parseval's identity gives
				\begin{align*}
					\mathbb{E}\!\left[\left \|\widehat{K_{\widetilde{\Pi}_p}}-K_{\widetilde{\Pi}_p}  \right \|^2_{\mathbb{L}_2([0,1]^2)}\right]
					&=
					\mathbb{E}\!\left[
					\sum^{p-1}_{k,k'=0}
					\left(
					\frac{1}{np}
					\sum_{i=1}^{n}
					\Big(
					Y_i\!\left ( \frac{k}{p} \right )Y_i\!\left ( \frac{k'}{p} \right )
					-
					\mathbb{E}\!\left [Y_i\!\left ( \frac{k}{p} \right )Y_i\!\left ( \frac{k'}{p} \right )  \right ]
					\Big)
					\right)^2
					\right]  \\
					&=
					\sum^{p-1}_{k,k'=0}
					\mathbb{E}\!\left[
					\left(
					\frac{1}{np}\sum_{i=1}^{n}
					Y_i\!\left ( \frac{k}{p} \right )Y_i\!\left ( \frac{k'}{p} \right )
					-
					\mathbb{E}\!\left[
					\frac{1}{np}\sum_{i=1}^{n}
					Y_i\!\left ( \frac{k}{p} \right )Y_i\!\left ( \frac{k'}{p} \right )
					\right]
					\right)^2
					\right]  \\
					&=\sum^{p-1}_{k,k'=0}
					\mathbb{V}\mathrm{ar}\!\left[
					\frac{1}{np}\sum_{i=1}^{n}
					Y_i\!\left ( \frac{k}{p} \right )Y_i\!\left ( \frac{k'}{p} \right )
					\right].
				\end{align*}
				Using the i.i.d.\ assumption on $(Y_i)_{i\ge 1}$, we obtain
				\begin{align*}
					\sum^{p-1}_{k,k'=0}
					\mathbb{V}\mathrm{ar}\!\left[
					\frac{1}{np}\sum_{i=1}^{n}
					Y_i\!\left ( \frac{k}{p} \right )Y_i\!\left ( \frac{k'}{p} \right )
					\right]
					&=\sum^{p-1}_{k,k'=0} \frac{1}{np^2}\,
					\mathbb{V}\mathrm{ar}\!\left[
					Y_1\!\left ( \frac{k}{p} \right )Y_1\!\left ( \frac{k'}{p} \right )
					\right] \\
					&\le \frac{1}{np^2}\sum^{p-1}_{k,k'=0}
					\mathbb{E}\!\left[
					\left(
					Y_1\!\left ( \frac{k}{p} \right )Y_1\!\left ( \frac{k'}{p} \right )
					\right)^2
					\right] \\
					&=\frac{1}{np^2}\sum^{p-1}_{k,k'=0}
					\mathbb{E}\!\left[
					Y_1\!\left ( \frac{k}{p} \right )^2
					Y_1\!\left ( \frac{k'}{p} \right )^2
					\right].
				\end{align*}
				The Cauchy--Schwarz inequality gives:
				\begin{align*}
					\mathbb{E}\!\left[
					Y_1\!\left ( \frac{k}{p} \right )^2
					Y_1\!\left ( \frac{k'}{p} \right )^2
					\right]
					&\le \Big(\mathbb{E}\!\left[Y_1\!\left ( \frac{k}{p} \right )^4\right]\Big)^{1/2}
					\Big(\mathbb{E}\!\left[Y_1\!\left ( \frac{k'}{p} \right )^4\right]\Big)^{1/2}.
				\end{align*}
				Therefore,
				\begin{align*}
					\mathbb{E}\!\left[\left \|\widehat{K_{\widetilde{\Pi}_p}}-K_{\widetilde{\Pi}_p}  \right \|^2_{\mathbb{L}_2([0,1]^2)}\right]
					&\le \frac{1}{np^2}\sum^{p-1}_{k,k'=0}
					\Big(\mathbb{E}\!\left[Y_1\!\left ( \frac{k}{p} \right )^4\right]\Big)^{1/2}
					\Big(\mathbb{E}\!\left[Y_1\!\left ( \frac{k'}{p} \right )^4\right]\Big)^{1/2} \\
					&\le \frac{1}{np^2}\sum^{p-1}_{k,k'=0}
					\sup_{0\le r\le p-1}
					\mathbb{E}\!\left[
					Y_1\!\left ( \frac{r}{p} \right )^4
					\right] \\
					&= \frac{1}{n}\,
					\sup_{0\le r\le p-1}
					\mathbb{E}\!\left[
					Y_1\!\left ( \frac{r}{p} \right )^4
					\right].
				\end{align*}
				Recall that,
				\[
				Y_1\left(\frac{k}{p}\right)=X_1\left(\frac{k}{p}\right)+\varepsilon_{1,k}.
				\]
				Therefore, using $(a+b)^4\le 8(a^4+b^4)$, we get for every $k$,
				\[
				\mathbb{E}\left[Y_1\left(\frac{k}{p}\right)^4\right]
				\le 8\left(
				\mathbb{E}\left[X_1\left(\frac{k}{p}\right)^4\right]
				+\mathbb{E}\left[\varepsilon_{1,k}^4\right]\right).
				\]
				Under the standing assumption $\sup_{t\in[0,1]}\mathbb{E}[X(t)^4]\le M$ and $\varepsilon_{1,k}{\sim}\mathcal{N}\left(0,\sigma^2\right)$ therefore
				$\mathbb{E}[\varepsilon_{1,k}^4]=3\sigma^4$ and we obtain
				\[
				\sup_{0\le k\le p-1}\mathbb{E}\left[Y_1^4\left(\frac{k}{p}\right)\right]
				\le 8\left(M+3\sigma^4\right).
				\]
				We obtain
				\begin{equation}\label{eq:fin_borne_sup_stocha_clean}
					\mathbb{E}\left[\left \|\widehat{K_{\widetilde{\Pi}_p}}-K_{\widetilde{\Pi}_p}
					\right \|^2_{\mathbb{L}_2([0,1]^2)}\right]
					\le \frac{8\left(M+3\sigma^4\right)}{n}.
				\end{equation}
				
				At the end, by combining \eqref{eq:bornesup1}, \eqref{eq:bornesup2}, \eqref{eq:err_approx}, \eqref{eq:err_quadrature} and \eqref{eq:fin_borne_sup_stocha_clean} we get the final result.
				\subsection{Proofs of Section~\ref{sec:illustration}}
				\subsubsection{Proof of Lemma~\ref{lem:section5.1_pol}}
				
				We first prove 1., assume that the eigenvalues satisfy~\eqref{def:pol_regime} then, for $\ell\geq J$, 
				\[
				\frac{\lambda_\ell \lambda_{\ell+1}}{(\lambda_\ell-\lambda_{\ell+1})^2}\leq \frac{c_1^2 \ell^{-(\gamma+1)}(\ell+1)^{-(\gamma+1)}}{c_2^2\ell^{-2(\gamma+2)}}\leq \frac{c_1^2}{c_2^2}\frac{\ell^{-2(\gamma+1)}}{\ell^{-2(\gamma+2)}}\leq \frac{c_1^2}{c_2^2}\ell^2
				\]
				and, for $\ell\geq J+1$, we also have 
				\[
				\frac{\lambda_{\ell-1}\lambda_{\ell}}{(\lambda_{\ell-1}-\lambda_{\ell})^2}\leq \frac{c_1^2 (\ell-1)^{-(\gamma+1)}\ell^{-(\gamma+1)}}{c_2^2(\ell-1)^{-2(\gamma+2)}}\leq \frac{c_1^2}{c_2^2}\frac{(\ell-1)^{-2(\gamma+1)}}{(\ell-1)^{-2(\gamma+2)}}\leq \frac{c_1^2}{c_2^2}(\ell-1)^2\leq \frac{c_1^2}{c_2^2}\ell^2. 
				\]
				Then, for $\ell\geq J+1$, 
				\[
				r_\ell(\Gamma_{P_X})\leq \frac{c_1^2}{c_2^2}\ell^2, 
				\]
				which implies the expected result by choosing the constant $c>0$ large enough.

				For point 2., we note that in the case of standard Brownian motion (resp. Brownian bridge) the kernel $K_{P_X}(s,t)=\min\{s,t\}$ (resp. $K_{P_X}(s,t) = \min\{s,t\}-st$ are not differentiable and satisfy~\eqref{PalphaL} with $\beta=1$ and $L=1$. Hence the associated distribution $P_X\in\mathcal P(1,1)$. Moreover, in the case of Brownian motion, $\lambda_j= \pi^{-2}(j-\frac12)^{-2}$, hence 
				\[
				\lambda_\ell-\lambda_{\ell+1} = \frac1{\pi^2(\ell-\frac12)^2}-\frac1{\pi^2(\ell+\frac12)^2}=\frac1{\pi^2}\frac{2\ell}{(\ell^2-\frac14)^2}\geq \frac2{\pi^2}\ell^{-3}. 
				\]
				This implies the expected result. The case of Brownian bridge is treated similarly.  
				
				We turn now to the case 3. of the $k$-fold integrated Brownian motion. By Theorem 2 of~\cite{Gao2003_IB} we have 
				\begin{equation}\label{ExpIntBM}
					\lambda_j= \left(\left(k_0+j-\frac12\right)\pi\right)^{-(2k+2)} + O\left(\frac1{j^{2k+3}}\exp\left(-j\pi\sin\left(\frac\pi{k+1}\right)\right)\right), 
				\end{equation}
				where $k_0$ is an integer. 
				This implies that there exists $J_1\geq 1$ and $c_1>0$ such that 
				\[
				\lambda_j\leq c_1 j^{-(2k+2)}, \qquad j\geq J_1. 
				\]
				Now, 
				\begin{eqnarray*}
					\lambda_j-\lambda_{j+1} &\geq &\left(\left(k_0+j-\frac12\right)\pi\right)^{-(2k+2)} - \left(\left(k_0+j+\frac12\right)\pi\right)^{-(2k+2)}\\
					&&-\left| \lambda_j-\left(\left(k_0+j-\frac12\right)\pi\right)^{-(2k+2)} \right|\\
					&&-\left|\lambda_{j+1}- \left(\left(k_0+j+\frac12\right)\pi\right)^{-(2k+2)}\right|. 
				\end{eqnarray*}
				
				We begin with the first term, from Taylor's theorem, there exists a constant $c_{2,1}$ and an integer $J_{2,1}$ such that, 
				\begin{eqnarray*}
					\left(\left(k_0+j-\frac12\right)\pi\right)^{-(2k+2)} &-&\left(\left(k_0+j+\frac12\right)\pi\right)^{-(2k+2)} \\
					&=& (j\pi)^{-(2k+2)}\left(\left(1+\frac{(k_0-\frac12)}{j}\right)^{-(2k+2)}-\left(1+\frac{(k_0+\frac12)}{j}\right)^{-(2k+2)}\right)\\
					&=&(j\pi)^{-(2k+2)}\left(1-(2k+2)\frac{(k_0-\frac12)}{j}-\left(1-(2k+2)\frac{(k_0+\frac12)}{j}\right)+o(j^{-1})\right)\\
					&=&(j\pi)^{-(2k+2)}\left(\frac{2k+2}{j}+o(j^{-1})\right)\\
					&\geq & c_{2,1} j^{-(2k+3)}, 
				\end{eqnarray*}
				for $j\geq J_{2,1}$.

				For the second (and third) terms, Eq.~\eqref{ExpIntBM} implies that there exists $J_{2,2}\geq 1$ and a constant $c_{2,2}>0$ such that, for $j\geq J_{2,2}$, 
				\begin{eqnarray*}
					\left| \lambda_j-\left(\left(k_0+j-\frac12\right)\pi\right)^{-(2k+2)} \right|\leq c_{2,2}j^{-(2k+3)}\exp\left(-j\pi\sin\left(\frac\pi{k+1}\right)\right),  
				\end{eqnarray*}
				meaning that there exists $J_{2,2}'\geq J_{2,2}$ such that for $j\geq J_{2,2}'$, 
				$$
				\left| \lambda_j-\left(\left(k_0+j-\frac12\right)\pi\right)^{-(2k+2)} \right|\leq \frac{c_{2,1}}{3}j^{-(2k+3)}. 
				$$
				This implies finally that 
				\[
				\lambda_j-\lambda_{j+1}\geq \frac{c_{2,1}}{3}j^{-(2k+3)}, 
				\]
				for $j> \max\{J_1,J_{2,1},J_{2,2}'\}$.

				The case 4 can be handled similarly by using Theorem 1 of \cite{Bronski2003_fBm} that states that there exists a constant $c_H>0$ such that, for all $\delta>0$, 
				\[
				\lambda_j = \frac{c_H}{j^{2H+1}}+o\left(j^{-\frac{(2H+2)(4H+3)}{4H+5}+\delta}\right). 
				\]
				\subsubsection{Proof of Lemma~\ref{lem:section5.1_exp}}
				
				We turn now to the exponential case (2.), assume that the eigenvalues satisfy~\eqref{def:exp_regime} then 
				for $\ell\geq J$, 
				\[
				\frac{\lambda_\ell\lambda_{\ell+1}}{(\lambda_\ell-\lambda_{\ell+1})^2}\leq \frac{c_1^2 \exp(-\ell\gamma)\exp(-(\ell+1)\gamma)}{c_2^2\exp(-2\ell\gamma)}\leq \frac{c_1^2e^{-\gamma}}{c_2^2}
				\]
				meaning that, for $\ell\geq J+1$, 
				\[
				r_\ell(\Gamma_{P_X})\leq \frac{c_1^2e^{-\gamma}}{c_2^2}
				\]
				which implies the expected result by choosing the constant $c>0$ large enough.
				
				\subsubsection{Proof of Theorem \ref{thm:regu_kern}}
				\begin{proof}
					The proof proceeds in three parts.
					
					\textbf{1. Differentiability up to order $m$.}
					For a multi-index $\operatorname{v} = (\operatorname{v}_1, \operatorname{v}_2)$ with $|\operatorname{v}| = r \leq m$, the formal partial derivative is
					$$
					\partial^{\operatorname{v}} K(s,t) = \sum_{j=1}^\infty \lambda_j \, \partial^{\operatorname{v}_1} \psi_{j}(s)  \, \partial^{\operatorname{v}_2} \psi_{j}(t).
					$$
					By~\eqref{eq:derivative_bound}, the absolute value of each term is bounded by:
					$$
					\bigg|\lambda_j \, \partial^{\operatorname{v}_1} \psi_{j}(s) \,\partial^{\operatorname{v}_2} \psi_{j}(t)\bigg| \leq cM_{\operatorname{v}_1} M_{\operatorname{v}_2}j^{-(\gamma+1)+\operatorname{v}_1+\operatorname{v}_1+2\zeta}= C_{\operatorname{v}} j^{-(1+\gamma) + r + 2\varsigma} = C_{\operatorname{v}} j^{-(1 + \alpha - r)},
					$$
					with $C_{\operatorname{v}}=cM_{\operatorname{v}_1} M_{\operatorname{v}_2}.$
					The series $\sum j^{-(1 + \alpha - r)}$ converges if $1 + \alpha - r > 1$, i.e., $r < \alpha$. By definition of $m$, this condition holds for all $r \leq m$ and the series for $\partial^{\operatorname{v}} K$ converges uniformly, establishing that $K \in C^m([0,1]^2)$.
					
					\textbf{2. Hölder regularity of $m$-th derivatives.} 
					We set
					$$G(s,t) = \frac{\partial^m K}{\partial s^m}(s,t)=\sum_{j=1}^\infty \lambda_j \, \partial^{m} \psi_{j}(s) \,\psi_{j}(t),\quad (s,t)\in [0,1]^2.$$
					Since for $(s,s')\in [0,1]^2$, $s\neq s'$, we have, by the mean value inequality and using multiple times~\eqref{eq:derivative_bound}
					$$\big|\partial^{m} \psi_{j}(s)-\partial^{m} \psi_{j}(s')\big|\leq M_{m+1}j^{m+1+\varsigma}|s-s'|,$$
					$$\big|\partial^{m} \psi_{j}(s)-\partial^{m} \psi_{j}(s')\big|\leq 2M_mj^{m+\varsigma}$$
					and 
					$$|\psi_{j}(t)|\leq M_0j^{\varsigma},\quad \lambda_j\leq cj^{-(\gamma+1)},$$ we obtain, since $\beta=\gamma-2\zeta-m$,
					\begin{align*}
						\bigg|\lambda_j \,\partial^{m} \psi_{j}(s) \,\psi_{j}(t)\bigg|&\leq C_g\min\Big(j^{-(\gamma+1)+m+2\varsigma},|s-s'|j^{-(\gamma+1)+m+1+2\varsigma}\Big)\\
						&\leq C_g\min\Big(j^{-(1+\beta)},|s-s'|j^{-\beta}\Big),
					\end{align*}
					where $C_g = c M_0 \max(2M_m, M_{m+1})$. 
					Now, we set $N = \lfloor |s-s'|^{-1} \rfloor$. We obtain the bound:
					$$
					|G(s,t) - G(s',t)| \leq C_g \left( |s-s'|\sum_{j=1}^N j^{-\beta} + \sum_{j=N+1}^\infty j^{-(1+\beta)} \right).
					$$
					The high-frequency sum is bounded by an integral: 
					$$\sum_{j=N+1}^\infty j^{-(1+\beta)} \leq \int_N^\infty x^{-(1+\beta)}dx = \frac{1}{\beta}N^{-\beta} \leq \frac{2^\beta}{\beta}|s-s'|^\beta.$$
					The regularity is thus determined by the low-frequency sum (for $j \leq N$).
					
					\textbf{- Case 1: $\alpha\not\in \mathbb{N}$ i.e. ($\beta \in (0,1)$).}
					The low-frequency sum is bounded by 
					$$\sum_{j=1}^N j^{-\beta} \leq 1 + \int_1^N x^{-\beta}dx \leq 1 + \frac{N^{1-\beta}}{1-\beta}.$$
					The corresponding term in the increment bound is thus
					$$|s-s'|\sum_{j=1}^N j^{-\beta}\leq|s-s'| \left(1 + \frac{N^{1-\beta}}{1-\beta}\right) \leq |s-s'| + \frac{|s-s'| |s-s'|^{-(1-\beta)}}{1-\beta} = |s-s'| + \frac{|s-s'|^\beta}{1-\beta}.$$
					Since $|s-s'| \leq |s-s'|^\beta$ for $|s-s'|\leq 1$, this part is of order $O(|s-s'|^\beta)$. Combining with the high-frequency part, we find $|K(s,t) - K(s',t)| \leq L|s-s'|^\beta$ for some constant $L$. This establishes that $K \in C^{m,\beta}([0,1]^2)$.
					
					\textbf{- Case 2: $\alpha\in\mathbb{N}$ i.e. ($\beta = 1$).} 
					The low-frequency sum becomes the harmonic sum: 
					$$\sum_{j=1}^N \frac{1}{j} \leq 1 + \ln(N).$$ 
					The increment is bounded by:
					$$ |s-s'| (1 + \ln(N)) \leq |s-s'|(1 + \ln(|s-s'|^{-1})) = |s-s'|(1 - \ln|s-s'|). $$
					Remark that, for all $\delta\in(0,1)$, the function $f_\delta(x) = x^{1-\delta}(1-\ln(x))$ is uniformly bounded on $(0,1]$ by $L_\delta=e^{-\delta}/(1-\delta)$ which implies that 
					$$|s-s'|\sum_{j=1}^N j^{-1}\leq L_\delta |s-s'|^\delta.$$
					Combining with the high-frequency part, we obtain the following bound 
					$$
					|G(s,t)-G(s',t)|\leq \frac{2^\beta}{\beta}|s-s'|^{\beta}+L_\delta|s-s'|^{\delta}\leq \left(\frac{2^\beta}{\beta}+L_\delta\right)|s-s'|^{\delta}, 
					$$
					hence $K \in C^{m,\delta}([0,1]^2).$
					
					The same frequency-splitting argument applies to any mixed partial derivative 
					$\partial^\nu K$ with $|\nu|=m$; hence all derivatives of total order $m$ are 
					H\"older continuous with the same exponent $\beta$ (or any $\delta<1$ when 
					$\beta=1$).
				\end{proof}
				\subsection{Technical lemmas}\label{sec:technicallemmas}
				The first lemma is central to establishing minimax lower bounds as a function of the sample size. It enables us to work with distances between Gaussian vectors instead of Gaussian processes, through straightforward reductions based on the periodic bases introduced in the lemma.
				\begin{lem}\label{lemm:trigo}
					We introduce for $j\geq $1,
					\[
					\begin{aligned}
						a_j : [0,1] &\longrightarrow [-\sqrt{2},\sqrt{2}] \\
						x &\longmapsto \sqrt{2}\sin(2\pi j x).
					\end{aligned}
					\]
					Then, the sequence of functions $(a_j)_{j=1}^\infty$  is an orthonormal family of $\mathbb{L }_2([0,1])$.\\
					Furthermore, if we write $t_k=\frac{k-1}{p}$ for all $p\in[\![4,\infty[\![$, $k\in[\![1,p]\!]$ and
					$$\boldsymbol{a}_j:=\left(a_j(t_1),\dots,a_j(t_{p})\right)^t,$$
					for $1\leq j\leq j' <\frac{p}{2}$ then 
					$$\langle \boldsymbol{a}_{j}, \boldsymbol{a}_{j'} \rangle_{\ell_2}=p\,\mathbf{1}_{j=j'}.$$
					
					Moreover, for all $j\geq 1$, for all $u,u'\in(0,1)$, for all $\beta\in(0,1],m\in\mathbb{N}$ and $\alpha=m+\beta$, 
					\begin{equation}\label{eq:ajHolder_grid_p}
						|a_j^{(m)}(u)-a_j^{(m)}(u')|\leq 2\sqrt{2} (2\pi j)^{\alpha} |u-u'|^\beta. 
					\end{equation}
				\end{lem}
				\begin{proof}
					Let $k,\ell \in\mathbb{N}\setminus\{0\},k\neq \ell $. On one hand
					\begin{eqnarray*}
						\left \| a_k \right \|^2
						&=&2\int_{0}^{1}\sin^2(2k\pi x)\,dx=\int_0^1(1-\cos(4k\pi x))\,dx=1-\frac{1}{4\pi k}\int_0^{4\pi k}\cos(u)\,du\\
						&=&1-\frac{1}{4\pi k}\big[\sin(u)\big]_{0}^{4\pi k}
						=1.
					\end{eqnarray*}
					and, for $k\neq \ell$, 
					\begin{align*}
						\left \langle a_{k},a_{\ell} \right \rangle
						&=2\int_{0}^{1} \sin(2k\pi x)\sin(2\ell\pi x)\,dx\\
						&=\int_0^1 \big(\cos(2(k-\ell)\pi x)-\cos(2(k+\ell)\pi x)\big)\,dx\\
						&=\Big[\frac{\sin(2(k-\ell)\pi x)}{2\pi(k-\ell)}\Big]_0^1
						-\Big[\frac{\sin(2(k+\ell)\pi x)}{2\pi(k+\ell)}\Big]_0^1=0.
					\end{align*}
					On the other hand, let $1 \leq j, j' < \frac{p}{2}$.
					\begin{equation}\label{eq:prod_scal_a_j_grid_p}
						\begin{aligned}
							\langle \boldsymbol{a}_{j}, \boldsymbol{a}_{j'} \rangle_{\ell_2} 
							&= \sum_{k=0}^{p-1} 2\sin\left( \frac{2\pi j k}{p} \right)\sin\left( \frac{2\pi j' k}{p} \right) \\
							&= \sum_{k=0}^{p-1} \cos\left( \frac{2\pi k (j-j')}{p} \right) 
							- \sum_{k=0}^{p-1} \cos\left( \frac{2\pi k (j+j')}{p} \right).
						\end{aligned}
					\end{equation}
					Then, note that if $q$ is a not a multiple of $p$ then
					\[
					\sum_{k=0}^{p-1} \cos\left(\frac{2\pi k q}{p}\right) 
					= \operatorname{Re}\left( \sum_{k=0}^{p-1} e^{i\frac{2\pi k q}{p}} \right) 
					= \operatorname{Re}\left( \frac{1 - e^{i 2\pi q}}{1 - e^{i \frac{2\pi q}{p}}} \right) 
					= 0, \quad \text{since } e^{i 2\pi q} = 1 \text{ and } e^{i \frac{2\pi q}{p}} \neq 1,
					\]
					and if $q$ is a multiple of $p$
					\[
					\sum_{k=0}^{p-1} \cos\left(\frac{2\pi k q}{p}\right) 
					= \sum_{k=0}^{p-1} \cos(0) 
					= \sum_{k=0}^{p-1} 1 
					= p.
					\]
					We then apply this to Equation \eqref{eq:prod_scal_a_j_grid_p}. Given the bounds on $j$ and $j'$, we have $|j-j'| < p$ and $0 < j+j' < p$.
					Consequently, $j+j'$ is never a multiple of $p$ and $j-j'$ is a multiple of $p$ if and only if $j=j'$ (in which case it is $0$). Thus:
					\[
					\langle \boldsymbol{a}_{j}, \boldsymbol{a}_{j'} \rangle_{\ell_2} 
					= p\,\mathbf{1}_{j=j'} - 0 
					= p\,\mathbf{1}_{j=j'}.
					\]
					Finally, we prove~\eqref{eq:ajHolder_grid_p}. Let $\beta\in(0,1]$ and $m\in\mathbb{N}$ and define $\alpha=m+\beta$. We use the algebraic decomposition:
					\[
					|a_j^{(m)}(u)-a_j^{(m)}(u')| 
					= |a_j^{(m)}(u)-a_j^{(m)}(u')|^{1-\beta} 
					\cdot |a_j^{(m)}(u)-a_j^{(m)}(u')|^\beta.
					\]
					For the first factor, we apply the triangle inequality. Since 
					$a_j^{(m)}(x)=\sqrt{2}(2\pi j)^m \sin(2\pi j x + m\pi/2)$, 
					we have $|a_j^{(m)}(x)|\leq \sqrt{2}(2\pi j)^m$ for all $x\in (0,1)$, and thus:
					\[
					|a_j^{(m)}(u)-a_j^{(m)}(u')|^{1-\beta} 
					\leq \left( 2\sqrt{2}(2\pi j)^m \right)^{1-\beta}.
					\]
					For the second factor, we apply the Mean Value Theorem. The derivative is $a_j^{(m+1)}$, which is bounded by $\sqrt{2}(2\pi j)^{m+1}$. Hence:
					\[
					|a_j^{(m)}(u)-a_j^{(m)}(u')|^\beta 
					\leq \left( \sqrt{2}(2\pi j)^{m+1} |u-u'| \right)^\beta.
					\]
					Multiplying these two bounds yields:
					\begin{align*}
						|a_j^{(m)}(u)-a_j^{(m)}(u')| 
						&\leq 2^{1-\beta} \left(\sqrt{2}\right)^{1-\beta+\beta} 
						(2\pi j)^{m(1-\beta) + (m+1)\beta} |u-u'|^\beta \\
						&= 2^{1-\beta} \sqrt{2} (2\pi j)^{m+\beta} |u-u'|^\beta \\
						&\leq 2\sqrt{2} (2\pi j)^{\alpha} |u-u'|^\beta.
					\end{align*}
				\end{proof}
				The following Lemma is a classical minimax lower-bound Lemma found in \cite{Tsybakov2009IntroductionTN} (see Theorem 2.2), and corresponds to Le Cam's method for binary hypothesis testing.
				\begin{lem}{(Minimax lower bound for two hypotheses based on Hellinger distance)}\\ \label{lemm1}
					Let $\{P_0, P_1\}$ be two probability measures on a measurable space $(\mathcal{X}, \mathcal{A})$. Let $\theta_0, \theta_1$ be parameters associated with these measures and $d$ a semi-metric.
					If
					$$ d(\theta_0, \theta_1) \geq 2s > 0 \quad \text{and} \quad H^2(P_0, P_1) \leq \alpha < 2, $$
					then
					$$ \inf_{\widehat{\theta}} \max_{k \in \{0,1\}} \mathbb{E}_{k} \left[d^2(\widehat{\theta}, \theta_k)\right] \geq s^2 \left(1 - \sqrt{\alpha\left(1-\frac{\alpha}{4}\right)}\right). $$
				\end{lem}
				The second Lemma allows us to compute the Hellinger divergence between two given normal processes.\\
				The result comes from \cite{Pardo2005StatisticalIB} (see p. 33) in which the result is obtained for general Rényi’s divergences.
				\begin{lem}
					{(Hellinger affinity for Gaussian distributions)}\label{lemm2}\\
					Let $d\in \mathbb{N}\setminus\{0\}$, $(\boldsymbol{\lambda}_0,\boldsymbol{\lambda}_1) \in (\mathbb{R}^d)^2$ and $(\Sigma_0, \Sigma_1)\in (\mathbb{R}^{d\times d})^2$ two symmetric positive definite matrices.\\
					Let $P_0,P_1$ be probability distributions of $\mathcal{N}(\boldsymbol{\lambda}_0,\Sigma_0)$ and $\mathcal{N}(\boldsymbol{\lambda}_1,\Sigma_1)$ respectively.\\
					The Hellinger affinity $\operatorname{A}$ between the distributions is then given by
					$$\operatorname{A}\left(P_{0}, P_{1}\right) =\frac{\operatorname{det}\left(\Sigma_{0}\Sigma_{1}\right)^{1/4}}{\operatorname{det}\left( \displaystyle\frac{\Sigma_{0}+ \Sigma_{1}}{2}\right)^{1 / 2}} \cdot \exp \left(-\frac{1}{8}\left(\boldsymbol{\lambda}_{0}-\boldsymbol{\lambda}_{1}\right)^{t}\left( \displaystyle\frac{\Sigma_{0}+ \Sigma_{1}}{2}\right)^{-1}\left(\boldsymbol{\lambda}_{0}-\boldsymbol{\lambda}_{1}\right)\right) .$$
				\end{lem}
				
				Bosq's inequality concerning eigenfunctions is given in \cite{Bosq2000LinearPI} (see Lemma 4.3, p.104) it is a Davis-Kahan type result. We used it to obtain a general upper-bound for the minimax risk of estimation of the eigenfunctions in Theorem \ref{thm:UB}.
				
				\begin{lem}{(Bosq's inequality for eigenfunctions)}\\\label{lemm4}
					Let $\Gamma$ and $\widehat{\Gamma}$ be two compact, self-adjoint operators on a separable Hilbert space $\mathcal{H}$. Let $(\lambda_j(\Gamma), \psi_j(\Gamma))_{j\geq 1}$ and $(\lambda_j(\widehat{\Gamma}), \psi_j(\widehat{\Gamma}))_{j\geq 1}$ denote their respective eigenvalues (sorted in decreasing order) and eigenfunctions.
					
					Assume that for a fixed $\ell \geq 1$, the eigenspace associated with $\psi_\ell$ is one-dimensional. Then, we have:
					\begin{equation*}
						\left\|\psi_\ell(\widehat\Gamma) - \operatorname{sign}\left(\langle \psi_\ell(\widehat\Gamma), \psi_\ell(\Gamma) \rangle\right)\psi_\ell(\Gamma) \right\| \leq \eta_\ell^{1/2} \|\widehat{\Gamma} - \Gamma\|_{\mathcal{L}},
					\end{equation*}
					where $\|\cdot\|_{\mathcal{L}}$ denotes the operator norm, and $\eta_\ell$ is defined as:
					$$
					\eta_{1}=8\left(\lambda_{1}(\Gamma)-\lambda_{2}(\Gamma)\right)^{-2}, \quad \text{and for } \ell \geq 2, \quad
					\eta_{\ell}=8 \left( \min \left(\lambda_{\ell}(\Gamma)-\lambda_{\ell+1}(\Gamma), \lambda_{\ell-1}(\Gamma)-\lambda_{\ell}(\Gamma)\right) \right)^{-2}.
					$$
				\end{lem}
				The following Lemma is used to obtain the upper-bound.
				\begin{lem}{(Efficient quadrature lemma)}\label{lem:approx_scal_prod}
					Let $p$ be a dyadic integer and set $j=\log_2(p)\in{\mathbb N}$.
					Assume that $K$ is symmetric and $m$-times differentiable with
					\[
					\left|\frac{\partial^m K}{\partial s^m}(s,t)
					-\frac{\partial^m K}{\partial s^m}(s',t)\right|
					\leq L|s-s'|^\beta,\qquad s,s',t\in[0,1],
					\]
					for some $\beta\in(0,1]$.
					Let the function $\phi$ have $m$ vanishing moments, namely we have 
					\[
					\int \phi(s)s^\ell ds=0,\quad \ell=1,\ldots,m,
					\]
					then, for any $k$ and any $k'$:
					\[
					\left|p\langle K,\phi_{j,k}\otimes\phi_{j,k'}\rangle
					-K\left(\frac{k}{p},\frac{k'}{p}\right)\right|
					\leq CLp^{-(m+\beta)},
					\]
					for $C$ a constant depending on $m,\beta$ and $\phi$.
				\end{lem}
				
				\begin{proof}
					Let $k$ and $k'$ be fixed. We bound the quantity
					\[
					p\langle K,\phi_{j,k}\otimes\phi_{j,k'}\rangle
					-K\left(\frac{k}{p},\frac{k'}{p}\right)
					\]
					as follows:
					\begin{align*}
						p\langle K,\phi_{j,k}\otimes\phi_{j,k'}\rangle
						-K\left(\frac{k}{p},\frac{k'}{p}\right)
						&=\iint 2^jK(s,t)\phi_{j,k}(s)\phi_{j,k'}(t)dsdt-K\left(\frac{k}{p},\frac{k'}{p}\right)\\
						&=\iint 2^{2j}K(s,t)\phi(2^js-k)\phi(2^jt-k')dsdt-K\left(\frac{k}{p},\frac{k'}{p}\right)\\
						&=\iint K\Bigg(\frac{s}{p},\frac{t}{p}\Bigg)\phi(s-k)\phi(t-k')dsdt-K\left(\frac{k}{p},\frac{k'}{p}\right)\\
						&=\iint K\Bigg(\frac{s+k}{p},\frac{t+k'}{p}\Bigg)\phi(s)\phi(t)dsdt-K\left(\frac{k}{p},\frac{k'}{p}\right)\\
						&=\iint \left(K\Bigg(\frac{s+k}{p},\frac{t+k'}{p}\Bigg)-K\left(\frac{k}{p},\frac{k'}{p}\right)\right)\phi(s)\phi(t)dsdt,
					\end{align*}
					since $\int\phi(s)ds=1$. We now study the term
					\[
					I_1:=\iint \left(K\Bigg(\frac{s+k}{p},\frac{t+k'}{p}\Bigg)
					-K\left(\frac{k}{p},\frac{t+k'}{p}\right)\right)\phi(s)\phi(t)dsdt.
					\]
					With $\alpha=m+\beta$, for $\beta\in (0,1]$, we have:
					\begin{align*}
						K\Bigg(\frac{s+k}{p},\frac{t+k'}{p}\Bigg)-K\left(\frac{k}{p},\frac{t+k'}{p}\right)
						=&\sum_{\ell=1}^{m-1}\frac{1}{\ell !}\frac{\partial^\ell K}{\partial s^\ell}
						\left(\frac{k}{p},\frac{t+k'}{p}\right)\left(\frac{s}{p}\right)^\ell\\
						&\hspace{-1cm}
						+\frac{1}{(m-1) !}\left(\frac{s}{p}\right)^m
						\int_0^1(1-x)^{m-1}\frac{\partial^m K}{\partial s^m}
						\left(\frac{k+xs}{p},\frac{t+k'}{p}\right)dx.
					\end{align*}
					Then,
					\begin{align*}
						I_1
						&=0+\iint \frac{1}{(m-1) !}\left(\frac{s}{p}\right)^m
						\left(\int_0^1(1-x)^{m-1}\left[
						\frac{\partial^m K}{\partial s^m}\left(\frac{k+xs}{p},\frac{t+k'}{p}\right)
						\right.\right.\\
						&\left.\left.\qquad-\frac{\partial^m K}{\partial s^m}\left(\frac{k}{p},\frac{t+k'}{p}\right)
						\right]dx\right)\phi(s)\phi(t)dsdt
					\end{align*}
					and 
					\begin{align*}
						|I_1|
						&\leq \iint\frac{1}{(m-1) !}\left|\frac{s}{p}\right|^m
						\int_0^1|1-x|^{m-1}L\left|\frac{xs}{p}\right|^{\beta}dx
						\,|\phi(s)|\,|\phi(t)|\,dsdt
						\;\;\le\; C_1 L p^{-(m+\beta)},
					\end{align*}
					for some constant $C_1$ depending only on $m,\beta$ and $\phi$.
					
					Using similar expansions for 
					\[
					I_2:=\iint \left(K\left(\frac{k}{p},\frac{t+k'}{p}\right)
					-K\left(\frac{k}{p},\frac{k'}{p}\right)\right)\phi(s)\phi(t)dsdt
					\]
					leads to the same bound. Now, by symmetry of $K$, we obtain the analogous control in the
					second argument.
					Hence,
					\[
					\left|p\langle K,\phi_{j,k}\otimes\phi_{j,k'}\rangle
					-K\left(\frac{k}{p},\frac{k'}{p}\right)\right|
					\le C L p^{-(m+\beta)},
					\]
					for a constant $C$ depending on $m,\beta$ and $\phi$.
				\end{proof}
				
				The following lemma are used to justify the numerical procedure in Section \ref{sec:simu}.
				
				
				
				\begin{lem}\label{lem:spectral_relation}
					Let $(e_j)_{j=1}^p$ be an orthonormal system in $\mathbb{L}_2([0,1])$ and let $G$ be an integral operator with image $\operatorname{Im}(G) = \operatorname{span}\{e_j\}_{j=1}^p$. Define the $p \times p$ matrix $\mathbf{G}$ by its entries $\mathbf{G}_{k',k} := \langle G e_k, e_{k'} \rangle$. Let $\operatorname{Sp}(G)$ and $\operatorname{Sp}(\mathbf{G})$ denote the spectra of the operator and the matrix, respectively. Then,
					$$ \operatorname{Sp}(G)\setminus\{0\} \subset \operatorname{Sp}(\mathbf{G}) \subset \operatorname{Sp}(G). $$
					In particular, if $0 \notin \operatorname{Sp}(\mathbf{G})$, then $\operatorname{Sp}(G) = \operatorname{Sp}(\mathbf{G})$.
					
					Furthermore, if $\mathbf{g}_\ell = (g_{1,\ell}, \dots, g_{p,\ell})^T$ is a unit-norm eigenvector of $\mathbf{G}$ with eigenvalue $\lambda_\ell$, then the function
					$$ \psi_\ell := \sum_{k=1}^{p} g_{k,\ell} e_k $$
					is a unit-norm eigenfunction of $G$ associated with the same eigenvalue $\lambda_\ell$.
				\end{lem}
				
				\begin{proof}
					First, we establish the spectral inclusions.
					\begin{enumerate}
						\item {\em First inclusion:} Let $\lambda \in \operatorname{Sp}(G) \setminus \{0\}$ with corresponding eigenfunction $\psi \in \mathbb{L}_2([0,1])$, such that $G\psi = \lambda\psi$. Since $\lambda \neq 0$, we have $\psi = \lambda^{-1}G\psi \in \operatorname{Im}(G)$. Thus, $\psi$ has a representation $\psi = \sum_{k=1}^p c_k e_k$ for a non-zero vector of coefficients $\mathbf{c} = (c_1, \dots, c_p)^T$. Projecting the eigenvalue equation onto each basis function $e_{k'}$ yields:
						$$ (\lambda\mathbf{c})_{k'}=\langle \lambda\psi, e_{k'} \rangle =\langle G\psi, e_{k'} \rangle = \sum_{k=1}^p c_k \langle G e_k, e_{k'} \rangle = \sum_{k=1}^p \mathbf{G}_{k',k} c_k = (\mathbf{Gc})_{k'}. $$
						This implies $\mathbf{Gc} = \lambda\mathbf{c}$, and since $\mathbf{c} \neq \mathbf{0}$, $\lambda \in \operatorname{Sp}(\mathbf{G})$. This shows $\operatorname{Sp}(G)\setminus\{0\} \subset \operatorname{Sp}(\mathbf{G})$.
						\item {\em Second inclusion:}  Let $\lambda \in \operatorname{Sp}(\mathbf{G})$ with a non-zero eigenvector $\mathbf{g} = (g_1, \dots, g_p)^T$. Define the function $\psi := \sum_{k=1}^p g_k e_k$, which is non-zero because $\{e_k\}$ are linearly independent. We verify that $\psi$ is an eigenfunction of $G$. For any $k' \in \{1, \dots, p\}$, the projection of $G\psi$ onto $e_{k'}$ is
						$$ \langle G\psi, e_{k'} \rangle = \left\langle \sum_{k=1}^p g_k G e_k, e_{k'} \right\rangle = \sum_{k=1}^p g_k \mathbf{G}_{k',k} = (\mathbf{Gg})_{k'}=(\lambda\mathbf{g})_{k'}=\lambda g_{k'}=\langle \lambda\psi, e_{k'} \rangle. $$
						As both $G\psi$ and $\lambda\psi$ lie in $\operatorname{span}\{e_j\}_{j=1}^p$ and their coordinates with respect to this basis are identical, the functions are equal. Thus $G\psi = \lambda\psi$, which shows $\lambda \in \operatorname{Sp}(G)$. This establishes $\operatorname{Sp}(\mathbf{G}) \subset \operatorname{Sp}(G)$.
					\end{enumerate}
					The final assertion of the lemma is a direct consequence of this second part. The function $\psi_\ell$ is constructed from the eigenvector $\mathbf{g}_\ell$ and is therefore an eigenfunction of $G$ with eigenvalue $\lambda_\ell$. Its norm is computed using the orthonormality of the basis:
					$$ \|\psi_\ell\|^2_{\mathbb{L}_2} = \left\langle \sum_k g_{k,\ell} e_k, \sum_j g_{j,\ell} e_j \right\rangle = \sum_{k,j} \overline{g_{k,\ell}} g_{j,\ell} \langle e_k, e_j \rangle = \sum_k |g_{k,\ell}|^2 = \|\mathbf{g}_\ell\|^2_2 = 1. $$
					This shows that $\psi_\ell$ is a unit-norm eigenfunction.
				\end{proof}

\bibliographystyle{alpha}

\begin{thebibliography}{CHPNR17}
	
	\bibitem[Bos00]{Bosq2000LinearPI}
	Denis Bosq.
	\newblock {\em Linear Processes in Function Spaces: Theory And Applications}.
	\newblock Springer, 2000.
	
	\bibitem[BPRR25]{BPRR24}
	Ryad Belhakem, Franck Picard, Vincent Rivoirard, and Angelina Roche.
	\newblock Minimax estimation of functional principal components from noisy
	discretized functional data.
	\newblock {\em Scand. J. Stat.}, 52(1):38--80, 2025.
	
	\bibitem[Bro03]{Bronski2003_fBm}
	Jared~C. Bronski.
	\newblock Small ball constants and tight eigenvalue asymptotics for fractional
	brownian motions.
	\newblock {\em Journal of Theoretical Probability}, 16(1):87--100, 2003.
	
	\bibitem[CFS99]{Cardot_Vieu_Sarda_Regression}
	Herv{\'e} Cardot, Frederic Ferraty, and Pascal Sarda.
	\newblock Functional linear model.
	\newblock {\em Statistics and Probability Letters}, 45:11--22, 02 1999.
	
	\bibitem[CHPNR17]{MR3620733}
	Micha\"el Chichignoud, Van~Ha Hoang, Thanh~Mai Pham~Ngoc, and Vincent
	Rivoirard.
	\newblock Adaptive wavelet multivariate regression with errors in variables.
	\newblock {\em Electron. J. Stat.}, 11(1):682--724, 2017.
	
	\bibitem[CWLY16]{Guanqun}
	Guanqun Cao, Lily Wang, Yehua Li, and Lijian Yang.
	\newblock Oracle-efficient confidence envelopes for covariance functions in
	dense functional data.
	\newblock {\em Statistica Sinica}, 26:359--383, 01 2016.
	
	\bibitem[CY11]{Cai2011mean}
	T.~Tony Cai and Ming Yuan.
	\newblock Optimal estimation of the mean function based on discretely sampled
	functional data: phase transition.
	\newblock {\em Ann. Statist.}, 39(5):2330--2355, 2011.
	
	\bibitem[Dau92]{Daubechies1992}
	Ingrid Daubechies.
	\newblock {\em Ten lectures on wavelets}.
	\newblock Society for Industrial and Applied Mathematics, USA, 1992.
	
	\bibitem[DP19]{descary2018}
	Marie-H{\'e}l{\`e}ne Descary and Victor~M. Panaretos.
	\newblock Functional data analysis by matrix completion.
	\newblock {\em The Annals of Statistics}, 47(1):1--38, 2019.
	
	\bibitem[DPR82]{DPR82}
	J.~Dauxois, A.~Pousse, and Y.~Romain.
	\newblock Asymptotic theory for the principal component analysis of a vector
	random function: Some applications to statistical inference.
	\newblock {\em Journal of Multivariate Analysis}, 12(1):136--154, 1982.
	
	\bibitem[FV06]{FerratyVieu}
	Fr\'{e}d\'{e}ric Ferraty and Philippe Vieu.
	\newblock {\em Nonparametric Functional Data Analysis: Theory and Practice
		(Springer Series in Statistics)}.
	\newblock Springer-Verlag, Berlin, Heidelberg, 2006.
	
	\bibitem[GHT03]{Gao2003_IB}
	F.~Gao, J.~Hannig, and T.~Torcaso.
	\newblock Integrated brownian motions and exact l2-small balls.
	\newblock {\em The Annals of Probability}, 31(3):1320--1337, 2003.
	
	\bibitem[GRL19]{pywt}
	Filip Wasilewski Kai Wohlfahrt Aaron~O'Leary Gregory R.~Lee, Ralf~Gommers.
	\newblock Pywavelets: A python package for wavelet analysis.
	\newblock {\em Journal of Open Source Software}, 4(36), August 2019.
	
	\bibitem[GRLG24]{gertheiss2023reviewFDA}
	Jan Gertheiss, David R{\"u}gamer, Bernard X.~W. Liew, and Sonja Greven.
	\newblock Functional data analysis: An introduction and recent developments.
	\newblock {\em Biometrical Journal}, 66(7), 2024.
	
	\bibitem[HE15]{HsingEubanks}
	Tailen Hsing and Randall Eubank.
	\newblock {\em Theoretical foundations of functional data analysis, with an
		introduction to linear operators}.
	\newblock Wiley, 2015.
	
	\bibitem[HH07]{Hall_Horowitz_2007}
	Peter Hall and Joel~L. Horowitz.
	\newblock Methodology and convergence rates for functional linear regression.
	\newblock {\em Ann. Statist.}, 35(1):70--91, 2007.
	
	\bibitem[JW23]{jirak2022}
	Moritz Jirak and Martin Wahl.
	\newblock Relative perturbation bounds with applications to empirical
	covariance operators.
	\newblock {\em Adv. Math.}, 412:Paper No. 108808, 59, 2023.
	
	\bibitem[KS20]{Keshri_Sharma_exploratory_anal_FPCA}
	Abhinav Keshri and Charu Sharma.
	\newblock {Exploratory Analysis of Functional Principal Components to Observe
		the Absorption of Election Sentiments in the Indian Stock Market}.
	\newblock MPRA Paper 122325, University Library of Munich, Germany, July 2020.
	
	\bibitem[LCS17]{Li_Chiou_Classif}
	Pai-Ling Li, Jeng-Min Chiou, and Yu~Shyr.
	\newblock {Functional data classification using covariate-adjusted subspace
		projection}.
	\newblock {\em Computational Statistics \& Data Analysis}, 115(C):21--34, 2017.
	
	\bibitem[Mal99]{Mallat1999}
	St{\'e}phane Mallat.
	\newblock {\em A Wavelet Tour of Signal Processing}.
	\newblock Elsevier, 1999.
	
	\bibitem[MR15]{mas}
	Andr\'e{} Mas and Frits Ruymgaart.
	\newblock High-dimensional principal projections.
	\newblock {\em Complex Anal. Oper. Theory}, 9(1):35--63, 2015.
	
	\bibitem[Par06]{Pardo2005StatisticalIB}
	Leandro Pardo.
	\newblock {\em Statistical Inference Based on Divergence Measures}.
	\newblock Taylor \& Francis, 2006.
	
	\bibitem[RS97]{RamsaySilverman}
	JO~Ramsay and BW~Silverman.
	\newblock {\em Functional data analysis}.
	\newblock Springer, 1997.
	
	\bibitem[RW20]{reiss2019nonasymptotic}
	Markus Reiss and Martin Wahl.
	\newblock Non-asymptotic upper bounds for the reconstruction error of pca.
	\newblock {\em The Annals of Statistics}, 48(2):1098--1123, 2020.
	
	\bibitem[Tsy09]{Tsybakov2009IntroductionTN}
	Alexandre~B. Tsybakov.
	\newblock Introduction to nonparametric estimation.
	\newblock In {\em Springer series in statistics}, 2009.
	
	\bibitem[WCM15]{Wang2015reviewFDA}
	Jane-Ling Wang, Jeng-Min Chiou, and Hans-Georg Mueller.
	\newblock Review of functional data analysis.
	\newblock {\em Annu. Rev. Statist.}, pages 1--41, 2015.
	
	\bibitem[WSCL24]{Wang2021optclassif_FDA}
	Shuoyang Wang, Zuofeng Shang, Guanqun Cao, and Jun Liu.
	\newblock Optimal classification for functional data.
	\newblock {\em Statistica Sinica}, 34:1545--1564, 2024.
	
	\bibitem[Xia20]{Xiao}
	Luo Xiao.
	\newblock Asymptotic properties of penalized splines for functional data.
	\newblock {\em Bernoulli}, 26(4):2847--2875, 2020.
	
	\bibitem[ZWY25]{zhou2024theoryFPCA}
	Hang Zhou, Dongyi Wei, and Fang Yao.
	\newblock Theory of functional principal component analysis for discretely
	observed data.
	\newblock {\em The Annals of Statistics}, 53(5):2103--2127, 2025.
	
\end{thebibliography}

\end{document}